\documentclass{amsart}
\usepackage[nohug]{diagrams}
\usepackage{amscd}
\usepackage[mathscr]{eucal}
\usepackage{amsfonts}
\usepackage{amsmath}
\usepackage{amsthm}
\usepackage{amssymb}
\usepackage{latexsym}
\usepackage{tabularx,array}
\usepackage[all]{xy}
\newfont{\msam}{msam10}

\newtheorem{theorem}[]{Theorem}
\newtheorem{proposition}[]{Proposition}
\newtheorem{corollary}[]{Corollary}
\newtheorem{lemma}[]{Lemma}
\theoremstyle{definition}
\newtheorem{definition}[]{Definition}
\newtheorem{remark}[]{Remark}
\newtheorem{example}[]{Example}

\def\remark{\noindent\textbf{Remark.}}

\newtheorem{prop}[theorem]{Proposition}
\newtheorem{cor}[theorem]{Corollary}

\let\nc\newcommand

\def\ni{\noindent}
\def\bthm{\begin{theorem}}
\def\ethm{\end{theorem}}
\def\blemma{\begin{lemma}}
\def\elemma{\end{lemma}}
\def\bproof{\begin{proof}}
\def\eproof{\end{proof}}
\def\bprop{\begin{proposition}}
\def\eprop{\end{proposition}}
\def\bcor{\begin{corollary}}
\def\ecor{\end{corollary}}
\nc{\la}{\label}
\def\A{\mathbb{A}}
\def\Z{\mathbb{Z}}
\def\N{\mathbb{N}}
\def\Q{\mathbb{Q}}

\def\X{\mathbb{X}}
\def\L{\boldsymbol{L}}
\def\R{\boldsymbol{R}}

\def\Com{\mathtt{Com}}

\def\Alg{\mathtt{Alg}}

\def\DGL{\mathtt{DGL}}
\def\Mod{\mathtt{Mod}}
\def\GrMod{\mathtt{GrMod}}
\def\Bimod{\mathtt{Bimod}}
\def\cAlg{\mathtt{Comm\,Alg}}

\def\Sets{\mathtt{Sets}}
\def\DGA{\mathtt{DGA}}
\def\DGBA{\mathtt{DGBA}}
\def\cDGMA{\mathtt{CDGMA}}

\def\cDGA{\mathtt{CDGA}}

\def\DGMod{\mathtt{DG\,Mod}}
\def\DGBimod{\mathtt{DG\,Bimod}}
\def\D{{\mathscr D}}
\def\C{{\mathscr C}}

\def\rtv#1{\!\!\sqrt[V]{#1}}

\def\Ho{{\mathtt{Ho}}}
\nc{\Ob}{{\rm Ob}}
\nc{\Mor}{{\rm Mor}}
\nc{\Hom}{{\rm{Hom}}}
\nc{\HOM}{\underline{\rm{Hom}}}
\nc{\DER}{\underline{\rm{Der}}}
\nc{\END}{\underline{\rm{End}}}
\nc{\Ext}{{\rm{Ext}}}
\nc{\Rep}{{\rm{Rep}}}
\nc{\DRep}{{\rm{DRep}}}
\nc{\NCRep}{\widetilde{\rm{Rep}}}
\nc{\RAct}{{\rm{RAct}}}
\nc{\bs}{\backslash}
\nc{\cn}{ \mbox{\rm c\^{o}ne} }
\nc{\ev}{{\tt{ev}}}
\nc{\n}{{\natural}}
\nc{\nn}{{{\natural} {\natural}}}
\nc{\colim}{{\tt{colim}}}
\nc{\B}{{\mathrm{B}}}
\nc{\Ba}{\overline{\mathrm{B}}}
\nc{\Ta}{\overline{\mathrm{T}}}
\nc{\bC}{\overline{C}}
\nc{\EXT}{\underline{\rm{Ext}}}
\nc{\TOR}{\underline{\rm{Tor}}}
\def\H{\mathrm H}
\def\tH{\mathrm{HC}}
\def\HC{\mathrm{HC}}
\def\HK{\mathrm{HK}}
\def\rHC{\overline{\mathrm{HC}}}
\def\rHH{\overline{\mathrm{HH}}}
\def\rCC{\overline{\mathrm{CC}}}
\def\CC{\mathrm{CC}}
\def\F{\mathcal F}
\def\G{\mathcal G}
\def\T{\mathrm T}
\def\ch{{\rm ch}}
\nc{\End}{{\rm{End}}}
\nc{\GL}{{\rm{GL}}}
\nc{\gl}{{\mathfrak{gl}}}
\nc{\g}{{\mathfrak{g}}}
\nc{\PGL}{{\rm{PGL}}}
\nc{\SL}{{\rm{SL}}}
\nc{\PSL}{{\rm{PSL}}}
\nc{\ad}{{\rm{ad}}}
\nc{\Ad}{{\rm{Ad}}}
\nc{\dlim}{\varinjlim}
\nc{\plim}{\varprojlim}
\def\eA{A^{\mbox{\scriptsize{\rm{e}}}}}

\def\eR{R^{\mbox{\scriptsize{\rm{e}}}}}

\def\cder{\mathrm{CDer}}

\def\DDER{\underline{{\mathbb D}{\rm er}}}

\newcommand{\HH}{{\rm{HH}}}

\newcommand{\Tor}{{\rm{Tor}}}
\newcommand{\Spec}{{\rm{Spec}}}

\newcommand{\Sym}{{\rm{Sym}}}
\newcommand{\Tot}{{\rm{Tot}}}
\newcommand{\bSym}{\boldsymbol{\Lambda}}
\newcommand{\bL}{\boldsymbol{\Lambda}}

\newcommand{\id}{{\rm{Id}}}
\newcommand{\bpar}{{\bar{\partial}}}
\newcommand{\Der}{{\rm{Der}}}
\newcommand{\der}{{\rm{Der}}}

\newcommand{\Tr}{{\rm{Tr}}}
\newcommand{\tTr}{\widetilde{\rm{Tr}}}
\newcommand{\NTr}{{\rm{NTr}}}
\newcommand{\tr}{{\rm{tr}}}
\newcommand{\Ker}{{\rm{Ker}}}

\newcommand{\im}{{\rm{Im}}}

\newcommand{\into}{\,\hookrightarrow\,}

\newcommand{\onto}{\,\twoheadrightarrow\,}
\newcommand{\sonto}{\,\stackrel{\sim}{\twoheadrightarrow}\,}
\newcommand{\sinto}{\,\stackrel{\sim}{\into}\,}
\newcommand{\rar}{\rightarrow}
\newcommand{\dga}{\mathtt{DGA}}

\newcommand{\FT}{\mathcal{C}}
\newcommand{\Hoch}{\mathcal{H}}
\newcommand{\eFT}{\widetilde{\FT}}
\newcommand{\Cyl}{\mathtt{Cyl}}
\newcommand{\Path}{\mathtt{Path}}
\newenvironment{lyxlist}[1]
{\begin{list}{}
{\settowidth{\labelwidth}{#1}
 \setlength{\leftmargin}{\labelwidth}
 \addtolength{\leftmargin}{\labelsep}
 }}
{\end{list}}
\numberwithin{equation}{section}
\numberwithin{theorem}{section}
\numberwithin{lemma}{section}
\numberwithin{proposition}{section}
\numberwithin{corollary}{section}
\numberwithin{example}{section}
\numberwithin{remark}{section}

\begin{document}

\title{Derived Representation Schemes and Cyclic Homology}
\author{Yuri Berest}
\address{Department of Mathematics,
 Cornell University, Ithaca, NY 14853-4201, USA}
\email{berest@math.cornell.edu}
\author{George Khachatryan}
\address{Department of Mathematics,
Cornell University, Ithaca, NY 14853-4201, USA}
\email{georgek@math.cornell.edu}
\author{Ajay Ramadoss}
\address{Departement Mathematik,
ETH Z\"urich,
8092 Z\"urich, Switzerland}
\email{ajay.ramadoss@math.ethz.ch}
%
%
\begin{abstract}
We describe the derived functor $ \DRep_V(A) $ of the affine representation scheme $ \Rep_V(A) $ parametrizing
the representations of an associative $k$-algebra $A$ on a finite-dimensional vector space $V$.
We construct the characteristic maps $ \Tr_V(A)_n:\, \HC_n(A) \to \H_n[\DRep_V(A)]\,$ extending
the canonical trace $ \Tr_V(A):\, \HC_0(A) \to k[\Rep_V(A)]\,$ to the higher cyclic homology of
the algebra $A$, and describe a related derived version of the representation functor introduced
recently by M.~Van den Bergh. We study various operations on the homology of
$ \DRep_V(A) $ induced by known operations on cyclic and Hochschild homology of $A$.
\end{abstract}
\maketitle

\vspace{-.2in}

\section{Introduction}
\la{S1}
Let $A$ be an associative unital algebra over a field $k$. The affine representation scheme parametrizing the $k$-linear representations of $A$ on a finite-dimensional vector space $V$ can be defined as the functor on the category of commutative algebras:
\begin{equation}
\la{rep}
\Rep_V(A):\ \cAlg_{k} \to \Sets\ ,\quad C \mapsto \Hom_{\Alg_k}(A,\, \End\,V \otimes_k C)\ .
\end{equation}
It is well known that \eqref{rep} is representable, and we denote the corresponding commutative $k$-algebra by $\, A_V = k[\Rep_V(A)] \,$.
Varying $ A $ (while keeping $ V $ fixed), one can regard $ \Rep_V(A) $ as a functor on the category of associative algebras
$ \Alg_k $. A natural problem then is to describe the higher derived functors of $ \Rep_V $ in the sense of non-abelian homological algebra \cite{Q1}. I.~Ciocan-Fontanine and M.~Kapranov proposed a geometric solution to this problem as part of a general program of deriving {\it Quot}\, schemes and other moduli spaces in algebraic geometry (see \cite{CK} and also \cite{BCHR, TV}). In this paper, we offer a different, more explicit construction which has its roots in classic works of
G.~Bergman and P.~M.~Cohn on universal algebra. As a result, we find a close relation of this problem to cyclic homology, which in retrospect looks very natural, almost inevitable, but was originally unexpected.

Our approach is motivated by recent developments in noncommutative geometry based on the heuristic principle \cite{KR}
that any geometrically meaningful structure on a noncommutative algebra $A$ should induce
standard commutative structures on all representation spaces $ \Rep_V(A) $
(see \cite{CQ, KR, G, LeB, LBW, CEG, vdB, vdB1, Ber}).
In practice, this principle works well only when $A$ is a (formally) smooth algebra, since in that case
$ \Rep_V(A) $ are smooth schemes for all $V$. Passing from $ \Rep_V(A) $ to the derived representation scheme
$ \DRep_V(A) $ amounts, in a sense, to desingularizing $ \Rep_V(A) $, so one may expect that
$ \DRep_V(A) $ will play a role in the geometry of arbitrary noncommutative algebras similar to the role of
$ \Rep_V(A) $ in the geometry of smooth algebras. In this paper, we make first steps in this direction by
constructing canonical trace maps with values in the homology of $ \DRep_V(A) $ and deriving the representation
functor on bimodules introduced recently in \cite{vdB}.

To clarify the idea consider the vector space $ \HC_0(A) := A/[A,A] $. The natural trace map
\begin{equation}
\la{trr}
\Tr_V(A):\, \HC_0(A) \to A_V\ ,\ \bar{a} \mapsto [\varrho \mapsto \Tr\, \varrho(a)]\ ,
\end{equation}
transforms the elements of $ \HC_0(A) $ to functions on $ \Rep_V(A) $ for each $V$. This suggests that the
0-th cyclic homology of $A$ should be thought of as a {\it space of functions} on the noncommutative `$ \Spec(A)$'
({\it cf.} \cite{KR}, Section~1.3.2).
Now, the derived functor of $ \Rep_V $ is represented (in the homotopy category of commutative DG algebras)
by a DG algebra $ \DRep_V(A) $. The homology of this DG algebra depends only on $A$ and $V$,
with $\ \H_0[\DRep_V(A)]  $ being isomorphic to $ A_V $. We denote $ \H_\bullet(A,V) := \H_\bullet[\DRep_V(A)] $ and
call $ \H_\bullet(A,V) $ the {\it representation homology} of $A$ with coefficients in $V$.
It turns out that representation homology is related to cyclic homology, and one of the main goals
of the present paper (and its sequel \cite{BR}) is to clarify this relation. We will construct functorial
trace maps
\begin{equation}
\la{trr1}
\Tr_V(A)_n:\, \HC_n(A) \to \H_n(A,V)\ ,\quad \forall\,n\ge 0\ ,
\end{equation}
extending the usual trace \eqref{trr} to higher cyclic homology.
In the spirit of Kontsevich-Rosenberg principle, the existence of \eqref{trr1} suggests that the full
cyclic homology $ \HC_\bullet(A) $ of an algebra $A$ should be thought of as a
{\it derived space of functions} on `$ \Spec(A)$.' We will study various operations on $\,\H_\bullet(A,V)\,$
arising from known operations on $ \HC_\bullet(A) $. The moral appears to be that all interesting
structures on cyclic and Hochschild homology (Bott periodicity, the Connes differential, the
Gerstenhaber bracket, $ \ldots $) induce via \eqref{trr1} interesting geometric structures on representation
homology. This provides a natural link between noncommutative algebraic geometry and the noncommutative
differential geometry developed by Tsygan, Tamarkin and others (see \cite{DGT}, \cite{TT}, \cite{TT1}, \cite{T1}
and references therein).

The results of this paper may be of interest beyond noncommutative geometry. Many important
varieties in algebra, geometry and physics can be realized as moduli spaces of
finite-dimensional representations of associative algebras and groups. The simplest
and most commonly used
are the character varieties parametrizing the isomorphism classes of semisimple representations.
When the base field $ k$ has characteristic zero (as we will always assume in this paper),
the character variety of $A$ in $V$ can be identified with the categorical quotient $ \Rep_V(A)/\!/\GL(V) $
of $ \Rep_V(A) $ by the natural action of $\GL(V) $.
The trace map \eqref{trr} takes its values in the corresponding commutative algebra $ A_V^{\GL(V)} $,
and these values are interpreted as characters of representations. Now, 
a well-known theorem of Procesi \cite{P} implies that the characters of $A$ actually generate $ A_V^{\GL(V)} $
as an algebra; in other words, the algebra map induced by \eqref{trr}:
\begin{equation}
\la{trra}
\Sym\,\Tr_V(A):\, \Sym\,[\HC_0(A)] \to A_V^{\GL(V)}
\end{equation}
is {\it surjective}. This result plays a crucial role in all applications as it
provides a natural presentation for character varieties.

The character varieties can be also `derived': the invariant functor $ A \mapsto A_V^{\GL(V)} $ has a left derived
functor, whose homology is isomorphic to $ \H_\bullet(A,V)^{\GL(V)} $ (see Section~\ref{S2.3.4}). The higher
traces \eqref{trr1} take their values in $ \H_\bullet(A,V)^{\GL(V)} $, and thus can be viewed as {\it derived
characters} of finite-dimensional representations of $A$. Assembled together, they define a homomorphism of graded commutative algebras
\begin{equation}
\la{trra1}
\bL \Tr_V(A)_\bullet:\,
{\bSym}[\HC_\bullet(A)] \to \H_\bullet(A,V)^{\GL(V)}\ ,
\end{equation}
where $ \bL $ stands for the graded symmetric algebra over $k$. Then, a natural question is
whether the Procesi Theorem extends to the derived setting: namely,
\begin{equation}
\la{quest}
\textit{Is the map \eqref{trra1} surjective}\,?
\end{equation}
We address this question in our subsequent paper \cite{BR}, where we show
that there are homological obstructions to surjectivity of \eqref{trra1} and give
examples where such obstructions do not vanish. Thus, the answer to \eqref{quest} turns 
out to be negative in general.


For non-unital algebras, the representation homology can be stabilized by passing to 
infinite-dimensional limit: $\, \dim_k V \to \infty $. In \cite{BR}, we will show that the maps \eqref{trra1} `converge' in that limit to an {\it isomorphism}
\begin{equation}
\la{trra2}
{\bL}[\HC_\bullet(A)] \stackrel{\sim}{\to} \H_\bullet(A, \infty)^{\Tr}\ ,
\end{equation}
where $\, \H_\bullet(A, \infty)^{\Tr} \,$ is the homology of a certain canonical (dense) DG subalgebra\footnote{We will review the construction of this subalgebra together with some of 
the main results of \cite{BR} in Section~\ref{S4.6} below.}
of $\,\DRep_{\infty}(A)^{\GL(\infty)} \cong \varprojlim\, 
\DRep_V(A)^{\GL(V)} $.  The isomorphism \eqref{trra2} is analogous to the well-known isomorphism of Loday, Quillen
\cite{LQ} and Tsygan \cite{T}:
\begin{equation}
\la{trra3}
\H_\bullet(\gl_\infty(A), k) \stackrel{\sim}{\to} \bL [\HC_{\bullet-1}(A)]\ ,
\end{equation}
which describes the stable homology of matrix Lie algebras $ \gl_V(A) $
in terms of cyclic homology. Although we establish a precise relation between \eqref{trra2}
and \eqref{trra3} (see Section~\ref{S4.5} below), this analogy still remains mysterious.

One puzzling question is a connection to algebraic $K$-theory. It is well known
that the Loday-Quillen-Tsygan isomorphism \eqref{trra3} has a natural `multiplicative'
analogue ({\it cf.} \cite{Q, L1}):
\begin{equation}
\la{trra4}
\H_\bullet(\GL_\infty(A), \Q) \cong {\bL}_{\Q}({\rm K}_{\bullet}(A) \otimes_{\Z} \Q)\ ,
\end{equation}
which identifies the rational $K$-theory of $A$ with the stable homology of general
linear groups $ \GL_V(A) $. Now, does \eqref{trra2} have a `multiplicative' analogue?
As we will see, there is a natural `noncommutative' version of representation homology,
which is related (via an appropriate trace map) to Hochschild homology
(see Section~\ref{S4.7} below). This relation is analogous to Loday's isomorphism
$\,{\rm HL}_\bullet(\gl_\infty(A), k) \cong {\rm T}_k[\HH_{\bullet-1}(A)] \,$, which
identifies the Leibniz homology of $ \gl_\infty(A) $ with the tensor algebra
of Hochschild homology of $A$ (see \cite{L}). Then, in the spirit
of Loday's conjectures \cite{L2}, one should expect that the answer
to the above question is affirmative.

Another question is more of practical nature. In characteristic zero, the cyclic homology
of  $A$ can be computed using Connes' cyclic complex $ \CC(A) $ ({\it cf.} Section~\ref{S4.3}).
Similarly, the Lie algebra homology of $ \gl_V(A) $ can be computed using the classical
Chevalley-Eilenberg complex. Both these complexes are `small'
in the sense that neither free algebras nor free products of $A$ are involved in
their construction. In view of the above relations to representation homology,
one might expect that the latter can be also computed using some `small' canonical
complex constructed out of $A$ and $V$. Having such a complex would simplify the results
of this paper and \cite{BR} in a substantial way. We note that our main construction
provides a `big' canonical complex computing $ \H_\bullet(A, V) $ for any algebra $A$,
namely
$$
C_\bullet(A,V) = (V^* \otimes_{\End\,V} (\End\,V \ast_k D_\bullet A) \otimes_{\End\,V} V)_\nn\ ,
$$
where $ D_\bullet(A) $ is the cobar-bar resolution of $A$ and $ (\,\mbox{--}\,)_\nn $ denotes
abelianization (see Section~\ref{S2}).

Finally, a few words about generalizations. In this paper, we focus on ordinary
associative DG algebras (more precisely, on morphisms of such algebras) as we believe
that these classical objects are  most important for applications. It is clear, however,
that our construction of derived representation schemes and the relation to cyclic homology
can be extended to other categories, including DG algebras over an arbitrary (cyclic) operad
and DG categories. In each case, the corresponding category has a natural model structure
(in the case of DG categories, this is the model structure introduced in \cite{Ta};
for the algebras over operads, see \cite{H}), and there is a relevant version of cyclic homology
(for DG categories, this is the version of cyclic homology introduced by Keller, see \cite{K3};
for algebras over cyclic operads, cyclic homology was defined in \cite{GK}).
Of special interest is an operadic analogue of our construction for the derived spaces
of algebra structures and derived Hilbert schemes ({\it cf.} \cite{Re}, \cite{CK1}); we plan to
discuss it elsewhere.

We now proceed with a detailed summary of the contents of the paper.

In Section~\ref{S2}, we present our construction of derived representation schemes and study
its basic properties. The key idea
is to extend the representation functor \eqref{rep} from the category of commutative algebras
to the category  of all associative algebras. It turns out that the functor \eqref{rep}
is still representable on this bigger category, and quite remarkably, its representing object
has a simple and explicit algebraic construction. In Section~\ref{S2.1}, we develop this idea
in greater generality needed for the present paper: we will work in the category $ \DGA_S $
of DG algebras over a fixed DG algebra $S$, and take $V$ to be a finite DG module over the
base algebra $S$. The main results of this section, Proposition~\ref{S2P1} and Theorem~\ref{S2T1},
establish the representability of the $\Rep_V$ functor in this relative DG setting.
In Section~\ref{S2.2}, we prove our first main result, Theorem~\ref{S2T2}, which says that
the representation functor $\,(\,\mbox{--}\,)_V :\,\DGA_S \to \cDGA_k \,$ is a left Quillen functor:
i.~e., it has a total left derived functor $\, \L(\,\mbox{--}\,)_V :\,\Ho(\DGA_S) \to \Ho(\cDGA_k) \,$,
which is part of a Quillen pair. Restricting $\, \L(\,\mbox{--}\,)_V \,$ to the category of
ordinary $S$-algebras, we define the relative derived representation scheme $ \DRep_V(S \bs A) $.
Section~\ref{S2.3} describes basic functorial properties of $ \DRep_V(S \bs A) $, which are used
throughout this paper and \cite{BR}. In particular, in \ref{S2.3.4}, we derive the invariant subfunctor
$ (\,\mbox{--}\,)^{\GL(V)}_V $ of the representation functor $ (\,\mbox{--}\,)_V $. We prove
(see Theorem~\ref{S2P4}) that $ (\,\mbox{--}\,)^{\GL(V)}_V $ has a total left derived functor
isomorphic to $ \L(\,\mbox{--}\,)^{\GL(V)}_V $, however, unlike the derived representation functor itself,
this functor is {\it not} a left Quillen functor (e.g., it does not seem to have a right adjoint).
Theorem~\ref{S2T44} clarifies the relation between $ \DRep_V(A) $ and the original construction in \cite{CK}
of the derived action spaces $ {\rm RAct}(A,V) \,$: in the case when $ V $ is a single vector
space, there is a natural isomorphism $\,  \DRep_V(A) \cong k[{\rm RAct}(A,V)] \,$ in the homotopy category of
(commutative) DG algebras. This isomorphism can be viewed as a strengthening of results of \cite{CK} as
it implies (in combination with our Theorem~\ref{S2T2}) that $ {\rm RAct}(A,V) $ is actually a right Quillen
derived functor on the category of DG schemes. Finally, in Section~\ref{S2.5}, we construct an explicit model
for $ \DRep_V(A) $ starting with a given almost free resolution of $ A $. Examples in the end of Section~\ref{S2} are
meant to illustrate how simple and explicit this model really is.

In Section~\ref{S4}, we give a self-contained exposition of Feigin-Tsygan's construction of relative cyclic homology
$ \HC_{\bullet}(S \bs A) $ as a non-abelian derived functor on the category of $S$-algebras (see \cite{FT, FT1}).
Our approach is a slightly more general than that of \cite{FT, FT1}, and our proofs are different from the proofs
in those papers.
In particular, the main results of this section, Theorem~\ref{FTT} and Theorem~\ref{TPT}, are
proved in a conceptual way, using simple homotopical arguments rather than
spectral sequences as in \cite{FT, FT1}.

In Section~\ref{TraceM}, we define canonical trace maps
$ \Tr_V(S \bs A)_n:\,\HC_{n-1}(S \bs A) \to \H_n(S\bs A, V) $
relating the cyclic homology of an $S$-algebra $A$ to its representation homology.
The main result of this section, Theorem~\ref{mcor}, describes an explicit chain
map $\, T: \CC(A) \to \DRep_V(A) \,$  that induces on homology the trace maps \eqref{trr1}.
The proof is based on Quillen's realization of the cyclic complex $ \CC(A) $ as the cocommutator
subspace of the reduced bar construction of $A$ ({\it cf.} Theorem~\ref{TQ}).
In Section~\ref{S4.5}, we explain the relation between the representation homology $ \H_\bullet(A,V) $
and the Lie algebra homology of $\gl_V(A)$ mentioned earlier in the Introduction. Finally, Section~\ref{S4.7}
describes a `noncommutative' version of the trace maps \eqref{trr1} defined on Hochschild homology.

In Section~\ref{S5}, we study various operations on representation
homology induced by well-known operations on cyclic homology.
The key result is Theorem~\ref{T2}, which extends our construction of the
derived representation functor from DG algebras to DG bimodules. This
functor should be viewed as a `correct' derived version of Van den Bergh's
functor introduced in \cite{vdB, vdB1}. In Section~\ref{S5.2},
we compute the derived tangent spaces for $ \DRep_V(A) $. In Section~\ref{S5.33},
we construct an analogue of the cyclic bicomplex for representation homology
and compute the effect of the periodicity operator
$ S: \HC_n(A) \to \HC_{n-2}(A) $ and the Connes differential
$ B: \HC_n(A) \to \HH_{n+1}(A) $ on it.
In Section~\ref{S5.4} we study a structure on $ \H_\bullet(A,V) $
induced by the canonical Gerstenhaber bracket on $ \HH^{\bullet+1}(A) $.


In Section~\ref{S4.6}, we outline a stabilization procedure for representation homology
of (augmented) algebras and explain in precise terms our answer to question \eqref{quest}.
This section is a brief announcement of the main results of \cite{BR}.
No proofs will appear here.

The paper ends with two Appendices, which we included for reader's convenience.
Our main results are proved using fundamental theorems of Quillen's theory of model categories.
Appendix~\ref{2} is a gentle introduction to this theory, where we gather together in a coherent
fashion all results needed in the body of the paper. In Appendix~\ref{3}, we recall basic facts about model
categories of DG algebras and prove some technical results
(Propositions~\ref{lhom} and~\ref{dgproper}), for which we could not find a reference in the literature.

\subsection*{Acknowledgements}{\footnotesize
We are grateful to G.~Felder, V.~Ginzburg, D.~Kaledin, B.~Keller, M.~Van den Bergh,
and especially M.~Kapranov, J.-L.~Loday and B.~Tsygan for interesting discussions and comments.
We are particularly indebted to G.~Powell who has carefully read the first version of this paper and sent us a list of questions and corrections.
We would also like to thank A.~Okounkov for a question that motivated this work, and M.~Kassabov
for sharing with us his ideas. We thank F.~Moore and G.~Felder for allowing us to use their
computer programs for {\tt Macaulay2} and for the help with actual computations. Many examples of
such computations can be found in the Ph.D. thesis of the second author (see \cite{Kha}).

The first author was partially supported by NSF grant DMS 09-01570, the second by an NSF Research Fellowship,
and the third by the Swiss National Science Foundation (Ambizione Beitrag Nr. PZ00P2-127427/1).}

\section*{Notation and Conventions}
Throughout this paper, $ k $ denotes a field of characteristic zero. An unadorned tensor product
$\, \otimes \,$ stands for the tensor product $\, \otimes_k \,$ over $k$. An algebra means an
associative $k$-algebra with $1$; the category of such algebras is denoted $ \Alg_k $.
Unless stated otherwise, all differential graded (DG) objects (complexes,
DG algebras, DG modules, etc.) are equipped with differentials of degree $-1$. The Koszul
sign convention is systematically used: whenever two DG maps (or operations) of degrees
$p$ and $q$ are permuted, the sign is multiplied by $ (-1)^{pq}$.
The categories of complexes of vector spaces, DG algebras and commutative
DG algebras over $k$ are denoted $ \Com_k $, $\,\DGA_k $ and $ \cDGA_k $, respectively.
As usual, $ \Alg_k $ will be identified with the full subcategory of
$ \DGA_k $ consisting of DG algebras with a single nonzero component in degree $0$.
If $V$ is a graded vector space, we will write $ \bL(V) $ for its graded
symmetric algebra  over $k\,$: thus $\, \bL(V) := \Sym_k(V_{\rm ev}) \otimes E_k(V_{\rm odd}) $,
where $ V_{\rm ev}$ and $ V_{\rm odd}$ are the even and the odd components of $V$.

\section{Derived Representation Schemes}
\la{S2}

\subsection{DG representation schemes}
\la{S2.1}
Let $ S \in \DGA_k $ be a DG algebra, and let $ \DGA_S $ denote the category of DG algebras over $S$.
By definition, the objects of $ \DGA_S $ are the DG algebra maps $\,S \to A \,$ in $ \DGA_k $
and the morphisms are given by the commutative triangles
\[
\begin{diagram}[small, tight]
  &       &     S     &        & \\
  &\ldTo  &           &  \rdTo &  \\
A &       & \rTo^{f}  &        & B
\end{diagram}
\]
We will write a map $\, S \to A \,$ as $\, S \bs A \,$, or simply $A$, when we regard it
as an object in $ \DGA_S $. For $ S \in \Alg_k $, we also introduce the category  $ \Alg_S $
of ordinary $S$-algebras (i.e. the category of morphisms $ S \to A $ in $\Alg_k$)
and identify it with a full subcategory of $ \DGA_S $ in the natural way.

Let $ (V, d_V) $ be a complex of $k$-vector spaces of finite (total) dimension, and
let $ \END\,V $ denote its graded endomorphism ring with  differential
$\, df = d_V f - (-1)^{i}f \,d_V \,$, where  $ f \in \END(V)_i\,$. Fix on $ V $ a
DG $S$-module structure, or equivalently, a DG representation $\,S \to \END\,V \,$.
This makes $ \END\,V $ a DG algebra over $S$, i.e. an object of $ \DGA_S $.
Now, given a DG algebra $ A \in \DGA_S $, an $S$-representation of $A$ in $V$ is, by definition,
a morphism $\,A \to \END\,V\,$ in $ \DGA_S $. Such representations form an affine DG scheme
which is defined as the functor on the category of commutative DG algebras:
\begin{equation}
\la{S2E1}
\Rep_V(S \bs A):\ \cDGA_{k} \to \Sets\ ,\quad C \mapsto \Hom_{\DGA_S}(A,\, \END\,V \otimes C)\ .
\end{equation}

To prove that \eqref{S2E1} is representable we will use a universal algebraic construction of `matrix reduction',
which (in the case of ordinary associative $k$-algebras) was introduced and studied in \cite{B} and \cite{C}. The
main advantage of this construction is that it produces the representing object for \eqref{S2E1} in a canonical form
as a result of application of some basic functors on the category of algebras.

Our starting point is the following simple observation. Denote by $ \DGA_{\End(V)} $ the category of DG algebras
over $ \END\,V \,$ and consider the natural functor
\begin{equation}
\la{S2E3}
\G:\, \DGA_{k} \to \DGA_{\End(V)}\ ,\quad B \mapsto \END\,V \otimes B \ ,
\end{equation}
where $ \END\,V \otimes B $ is viewed as an object in $ \DGA_{\End(V)} $ via the
canonical map $\,\END\,V \to \END\,V \otimes B\,$.
\blemma
\la{S2L1}
The functor \eqref{S2E3} is an equivalence of categories.
\elemma
\bproof
The inverse functor is given by
\begin{equation}
\la{S2E2}
\G^{-1}:\, \DGA_{\End(V)} \to \DGA_k\ ,\quad (\END\,V \to A) \mapsto A^{\,\END(V)}\ ,
\end{equation}
where $ A^{\,\END(V)}\,$ denotes the (graded) centralizer of the image of $ \END\,V $ in $A\,$. Indeed,
$\,\G^{-1} \circ \G\,$ being isomorphic to the identity functor on $ \DGA_k $ is obvious.
To prove that $\,\G \circ \G^{-1} \cong \id \,$ we need to show that, for any morphism
$\,f:\,\END\,V \to A\,$, the natural map
\begin{equation}
\la{S2E4}
\varphi:\, \END\,V \otimes A^{\END(V)}  \to A\ , \quad w \otimes a  \mapsto  f(w) \, a \ ,
\end{equation}
is an isomorphism. Since $ \varphi $ is obviously a morphism of DG algebras, it suffices to construct
a linear map  $\, \psi:\, A \to \END\,V \otimes A^{\,\END(V)} \,$, which is inverse to $\,\varphi\,$ as
a map of vector spaces.
Forgetting the differentials, we choose a linear basis $\,\{v_i\}\,$ in $\, V \,$ consisting of homogeneous elements,
and define the elementary endomorphisms $ \{e_{ij}\} $ in $ \END\,V $ by $ e_{ij}(v_k) = \delta_{jk} v_i $.
These endomorphisms are obviously homogeneous (the degree of $ e_{ij} $ being $\, |v_i| - |v_j|\,$)
and satisfy the relations
\begin{equation}
\la{S2E5}
\sum_{i=1}^d e_{ii} = 1 \ , \qquad e_{ij}\,e_{kl} = \delta_{jk}\, e_{il}\ ,
\end{equation}
where $ d := \dim_k V $. Now, identifying $ \END\,V $ with its image in $ A $ under the given map $ f $,
we define for each homogeneous element $ a \in A $ its `matrix' elements
\begin{equation}
\la{S2E6}
a_{ij} := \sum_{k=1}^d\, (-1)^{(|a|+|e_{ji}|) |e_{jk}|}\, e_{ki}\,a\,e_{jk} \ ,\quad i,\,j = 1,\,2,\,\ldots\,,\,d\ .
\end{equation}
Notice that each $ a_{ij} $ is homogeneous in $ A $ of degree
$\, |a_{ij}| = |a| + |e_{ji}| $. Moreover, by \eqref{S2E5}, we have
\begin{eqnarray*}
[a_{ij},\,e_{kl}] & = & a_{ij}\,e_{kl} - (-1)^{|a_{ij}||e_{kl}|}\, e_{kl}\,a_{ij}\\
&=& (-1)^{(|a|+|e_{ji}|) |e_{jk}|}\,  e_{ki}\,a\,e_{jl} -
(-1)^{(|a|+|e_{ji}|)(|e_{kl}| + |e_{jl}|)} \, e_{ki}\,a\,e_{jl} \\
&=& (-1)^{|a_{ij}||e_{jk}|} (1 - (-1)^{|a_{ij}|(|e_{kl}| + |e_{jl}| + |e_{jk}|)}) \, e_{ki}\,a\,e_{jl} \\
&=& (-1)^{|a_{ij}||e_{jk}|} (1 - (-1)^{2|a_{ij}|(|v_j| - |v_l|)}) \, e_{ki}\,a\,e_{jl} = 0
\end{eqnarray*}
for all $\,i,j, k,l = 1,2,\ldots,d\,$. Since $ \{e_{kl}\} $ span $ \END\,V $ as a vector space,
this shows that $ a_{ij} \in A^{\,\END(V)} $ for all $ a \in A $. Using \eqref{S2E6}, we can now define
the linear map
\begin{equation}
\la{S2E7}
\psi:\, A \to \END\,V \otimes A^{\END(V)}\ ,\quad a \mapsto \sum_{i,j = 1}^d\,e_{ij} \otimes a_{ij}\ .
\end{equation}
A straightforward calculation shows that $ \psi $ is indeed the inverse of $ \varphi $ as a map of
vector spaces. Since $\varphi$ is a DG algebra map, it follows that $ \psi $ is a DG algebra isomorphism.
\eproof
Next, we introduce the following functors
\begin{eqnarray}
&& \rtv{\,\mbox{--}\,}:\, \DGA_S \to \DGA_k\ ,\quad S \bs A \mapsto (\END\,V \ast_S A)^{\END(V)}\ ,\la{S2E9} \\*[1ex]
&& (\,\mbox{--}\,)_V:\, \DGA_S \to \cDGA_k\ ,\quad S \bs A \mapsto (\rtv{S \bs A})_\nn\ , \la{S2E10}
\end{eqnarray}
where $\,\ast_S \,$ denotes the coproduct in the category $ \DGA_S $ and
$\,(\mbox{--})_\nn:\,\DGA_k \to \cDGA_k \,$ stands for abelianization, i.e.
taking the quotient of a DG algebra $R$ by its two-sided commutator ideal:
$\,R_\nn := R/\langle[R, R]\rangle\,$.
\bprop
\la{S2P1}
For any $ S\bs A \in \DGA_S $, $ B \in \DGA_k $ and $ C \in \cDGA_k $, there are natural bijections

\vspace{0.8ex}

$(a)$ $\,\Hom_{\DGA_k}(\rtv{S\bs A},\,B) \cong \Hom_{\DGA_S}(A,\,\END\,V \otimes B) \,$,

\vspace{0.8ex}

$(b)$ $\, \Hom_{\cDGA_k}((S\bs A)_V,\,C) \cong \Hom_{\DGA_S}(A,\,\END\,V \otimes C) \,$.

\eprop
\begin{proof}
The tensor functor $\, B \mapsto \END\,V \otimes B \,$ in $(a)$ can be formally written as the
composition
\begin{equation}\la{compf}
\DGA_k \xrightarrow{\G} \DGA_{\End(V)} \xrightarrow{\F} \DGA_S\ ,
\end{equation}
where $\G$ is defined by \eqref{S2E3} and $ \F $ is the restriction functor via
the given DG algebra map $ S \to \END\,V $. Both
$ \F $ and $ \G $ have natural left adjoint functors: the left adjoint of
$ \F $ is obviously the coproduct $\,A \mapsto \END\,V \ast_S A \,$, while the left
adjoint of $ \G $ is $ \G^{-1} $, since $ \G $ is an equivalence of categories
(Lemma~\ref{S2L1}). Now, by definition, the functor $ \rtv{\,\mbox{--}\,} $ is the
composition of these left adjoint functors and hence the left adjoint to the
composition \eqref{compf}. This proves part $(a)$. Part $(b)$ follows from $(a)$ and
the obvious fact that the abelianization functor $\,(\mbox{--})_\nn:\,\DGA_k \to \cDGA_k \,$
is left adjoint to the inclusion $ \iota:\,\cDGA_k \into \DGA_k $.
\end{proof}
\remark\
The proof of Lemma~\ref{S2L1} shows that the DG algebra $ \rtv{S \bs A} $ has an explicit presentation:
namely, every element of $ \rtv{S \bs A} $ can be written in the form \eqref{S2E6},
where $ a \in \END\,V \ast_S A $. In this case, the map
\begin{equation}
\la{S2E8}
\psi:\, \END\,V \ast_S A \ \stackrel{\sim}{\to} \ \END\,V \otimes \rtv{S \bs A}\ ,
\quad a \mapsto \sum_{i,j = 1}^d \,e_{ij} \otimes a_{ij}
\end{equation}
yields a DG algebra isomorphism which is inverse to the canonical (multiplication) map
$$
\END\,V \otimes \rtv{S \bs A} \stackrel{\sim}{\to} \END\,V \ast_S A \ .
$$
Using \eqref{S2E8}, we can write the bijection of Proposition~\ref{S2P1}$(a)$ explicitly:
$$
\Hom_{\DGA_k}(\rtv{S\bs A},\,B) \to \Hom_{\DGA_S}(A,\, \END\,V \otimes B)\ ,\quad
f \mapsto (\id \otimes f) \circ \psi|_A\ ,
$$
where $ \psi|_A $ is the composition of \eqref{S2E8} with the canonical map $ A \to \END\,V \ast_S A \,$.

\vspace{1ex}

Part $(b)$ of Proposition~\ref{S2P1} can be restated in the following way, which shows that
$ \Rep_V(S \bs A) $ is indeed an affine DG scheme in the usual sense (see, e.g, \cite{CK}, Example~2.2.5).
\bthm
\la{S2T1}
For any $ S\bs A \in \DGA_S $, the commutative DG algebra $ (S \bs A)_V  $ represents the functor \eqref{S2E1}.
\ethm

Notice that if $S$ is an ordinary algebra and $V$ is a complex consisting of a single $S$-module in degree $0$,
the functors \eqref{S2E9} and \eqref{S2E10} restrict to
\begin{eqnarray}
&& \rtv{\,\mbox{--}\,}:\, \Alg_S \to \Alg_k\ ,\quad S \bs A \mapsto (\End\,V \ast_S A)^{\End(V)}\ ,\la{S2E11} \\*[1ex]
&& (\,\mbox{--}\,)_V:\, \Alg_S \to \cAlg_k\ ,\quad S \bs A \mapsto (\rtv{S \bs A})_\nn\  \la{S2E12}\ ,
\end{eqnarray}
and Proposition~\ref{S2P1} yields canonical isomorphisms
\begin{eqnarray}
&&  \Hom_{\Alg_k}(\rtv{S\bs A},\,B) \cong \Hom_{\Alg_S}(A,\,\End\,V \otimes B)\ ,\la{S2E11'} \\*[1ex]
&&  \Hom_{\cAlg_k}((S\bs A)_V,\,C) \cong \Hom_{\Alg_S}(A,\,\End\,V \otimes C)\ . \la{S2E12'}
\end{eqnarray}

In particular, when $ S = k $, the commutative algebra $ A_V := (k \bs A)_V $ represents the functor \eqref{rep}.

\vspace{1ex}

\remark\
For ordinary $k$-algebras, Proposition~\ref{S2P1} and Theorem~\ref{S2T1} were originally proven in \cite{B} (Sect.~7)
and \cite{C} (Sect.~6). In these papers, the functor \eqref{S2E11} was called the `matrix reduction' and a
different notation was used. Our notation $\, \rtv{\mbox{--}} \,$ is borrowed from \cite{LBW}, where
\eqref{S2E11} is used for constructing noncommutative thickenings of classical representation schemes.

\subsubsection{Generalizations}
\la{S2.1.1}
The above results can be extended to the relative case when $k$ is replaced by an arbitrary algebra.
Although we will not need it in the present paper, we will briefly outline this generalization.

Let $ R  \in \Alg_k $ be an algebra, and let $ V $ be a perfect complex over $R$
(i.~e., a bounded complex consisting of f.~g. right projective $R$-modules). Denote by
$ E := \END_R(V) \,$ the DG subalgebra of $ \END\,V $ consisting of $R$-linear endomorphisms
and define the functor ({\it cf.} \eqref{S2E3})
\begin{equation*}
\la{S2E3'}
\G_R:\, \DGA_{R} \to \DGA_{E}\ ,\quad B \mapsto \END_B(V \otimes_R B) = V \otimes_R B \otimes_R V^* \ ,
\end{equation*}
where $\,V^* := \HOM_R(V,\,R) \,$ is the dual complex of $V$ over $R$.

\blemma
\la{S2L1'}
Assume that $\, V^* \otimes_{E} V \cong R\,$.  Then $ \G_R $ is an equivalence of categories.
\elemma
\noindent
The inverse functor is given by
$\,
\G_R^{-1}:\,  \DGA_{E} \to \DGA_R\ ,\quad A \mapsto V^* \otimes_{E} A \otimes_{E} V\,$.
The generalization of Proposition~\ref{S2P1} now reads as follows.
\bprop
\la{S2P1'}
Let $ S $ be a DG algebra, and let $ S \to \END_R(V) $ be a DG algebra homomorphism making
$V$ a DG $S$-$R$-bimodule. Then, for any $ S \bs A \in \DGA_S $ and $ B \in \DGA_R $,
there is a functorial isomorphism
$$
\Hom_{\DGA_R}(\sqrt[V]{S\bs A/ R},\, B) \cong \Hom_{\DGA_S}(A,\, \END_B(V \otimes_R B))\ ,
$$
where $\,\sqrt[V]{S \bs A /R} :=  V^* \otimes_{E} (E \ast_S A) \otimes_{E} V\,$.
\eprop
More generally, Lemma~\ref{S2L1'} and Proposition~\ref{S2P1'} hold if we take $ R $ to
be an arbitrary DG algebra, and $ V $ a finite {\it semifree} DG module over $R$.
Our proofs given in the simplest case $ R = k $ work, {\it mutatis mutandis}, in general.
We leave the details as an exercise to the reader.

\subsection{Deriving the representation functor}
\la{S2.2}
The category $ \DGA_k $ and its subcategory $ \cDGA_k $ have natural Quillen model
structures, where the weak equivalences are the quasi-isomorphisms and the fibrations are
the degreewise surjective maps. We refer the reader to the Appendices
in the end of the paper, where we review basic results from the theory of model categories (Appendix~A)
and properties of model structures on categories of DG algebras (Appendix~B). Here, we only recall
that, for a fixed $ S \in \DGA_k $, the category $ \DGA_S $ inherits a model structure from
$ \DGA_k $, in which a morphism $\,f: S \bs A \to S \bs B \,$ is a weak equivalence (resp.,
fibration; resp., cofibration) iff $ f: A \to B $ is a weak equivalence (resp.,
fibration; resp., cofibration) in $ \DGA_k $ ({\it cf.} \ref{2.1.2}).
Every DG algebra $\, S \bs A \in \DGA_S \,$ (in particular, an ordinary algebra in $ \Alg_S$)
has a cofibrant resolution $\, Q(S \bs A) \sonto S \bs A \,$ in $ \DGA_S $, which is given by a
factorization $\,S \into Q \sonto A \,$ of the morphism $ S \to A $ in $ \DGA_k $.
Replacing DG algebras by their cofibrant resolutions one defines the homotopy category
$ \Ho(\DGA_S) $, in which the morphisms are given by the homotopy classes of morphisms between
cofibrant objects in $ \DGA_S $. The category
$ \Ho(\DGA_S) $ is equivalent to the (abstract) localization of $ \DGA_S $ at the class
of weak equivalences in $ \DGA_S $ ({\it cf.} Theorem~\ref{Tloc}).
We denote the corresponding localization functor by $\,\gamma:\,\DGA_S \to \Ho(\DGA_S)\,$;
by definition, $ \gamma $ acts as identity on the objects while maps each morphism
$ f: S \bs A \to  S \bs B $ to the homotopy class of its cofibrant lifting
$ \tilde{f}: Q(S \bs A) \to Q(S \bs B) $  in $\DGA_S \,$, see \ref{2.2.4}.

We are now in position to state one of the main results of this paper.
\bthm
\la{S2T2}
$(a)$ The functors $\,(\,\mbox{--}\,)_V :\,\DGA_S \rightleftarrows \cDGA_k\,: \END\,V \otimes {\mbox{--}}\,$
form a Quillen pair.

$(b)$ $\, (\,\mbox{--}\,)_V $ has a total left derived functor defined by
$$
\L(\,\mbox{--}\,)_V:\,\Ho(\DGA_S) \to \Ho(\cDGA_k) \ ,
\quad S \bs A  \mapsto  Q(S \bs A)_V\ , \quad \gamma f \mapsto \gamma(\tilde f_V)\ .
$$

$(c)$ For any $\, S \bs A \in \DGA_S \,$ and $ B \in \cDGA_k $, there is a canonical isomorphism
\begin{equation*}
\la{S2E15}
\Hom_{\Ho(\cDGA_k)}(\L(S \bs A)_V,\, B) \cong
\Hom_{\Ho(\DGA_S)}(A,\,\END\,V \otimes B)\ .
\end{equation*}
\ethm
\bproof
By Proposition~\ref{S2P1}$(b)$, the functor $\,(\,\mbox{--}\,)_V \,$ is left adjoint to the composition
$$
\cDGA_k \stackrel{\iota}{\into} \DGA_k \xrightarrow{\END\,V \otimes\,-\,} \DGA_S \ ,
$$
which we still denote $\,\END\,V \otimes {\mbox{--}}\ $.
Both the forgetful functor $ \iota $ and the tensoring with
$ \END\,V $ over a field are exact functors on $ \Com_k $; hence, they map fibrations (the surjective morphisms in $ \DGA_k $) to fibrations and also preserve the class of weak
equivalences (the quasi-isomorphisms). It follows that $\,\END\,V \otimes {\mbox{--}}\,$ preserves fibrations as well as acyclic
fibrations. Thus, by Lemma~\ref{Qpair}, $\,(\,\mbox{--}\,)_V :\,\DGA_S \rightleftarrows \cDGA_k\,: \END\,V \otimes {\mbox{--}}\,$ is a Quillen pair. This proves part $(a)$. Part $(b)$ and $(c)$ now follow
directly from Quillen's Fundamental Theorem (see Theorem~\ref{Qthm}).
For part $(c)$, we need only to note that $\,G := \END\,V \otimes {\mbox{--}} \,$ is an exact functor in Quillen's sense,
i.e. $ \R G = G $, since $ \cDGA_k $ is a fibrant model category.
\eproof

\begin{definition}
\la{S2D1}
By Theorem~\ref{S2T2}, the assignment $\, S\bs A \mapsto Q(S\bs A)_V \,$ defines a functor
$$
\DRep_V:\ \Alg_S \to \Ho(\cDGA_k)
$$
which is independent of the choice of resolution $ Q(S\bs A) \sonto S\bs A $ in $ \DGA_S $. Abusing terminology, we call
$ \DRep_V(S \bs A) $ a relative {\it derived representation scheme} of $ A $. The homology of
$ \DRep_V(S \bs A) $ is a graded commutative algebra, which depends only on $ S \bs A$ and $V$. We write
\begin{equation}
\la{S2E13}
\H_\bullet(S \bs A,\,V) := \H_\bullet[\DRep_V(S\bs A)]
\end{equation}
and refer to \eqref{S2E13} as {\it representation homology} of $S \bs A$ with coefficients in $V$. In the absolute case
when $ S = k $, we simplify the notation writing $ \DRep_V(A) := \DRep_V(k \bs A) $ and
$ \H_\bullet(A,\,V) := \H_\bullet(k \bs A,\,V) $.
\end{definition}

\vspace{1ex}

We now make a few remarks related to Theorem~\ref{S2T2}.

\subsubsection{}
\la{S2.2.2}
For any cofibrant resolutions $\,p: Q(S\bs A) \sonto S\bs A \,$ and $\,p': Q'(S\bs A) \sonto S\bs A \,$
of a given $ S \bs A \in \DGA_S $, there is a quasi-isomorphism $\,f_V:\,Q(S\bs A)_V
\stackrel{\sim}{\to} Q'(S\bs A)_V \,$ in $ \cDGA_k $. Indeed, by \ref{2.2.4}, the identity map on $A$
lifts to a morphism $ f: Q(S\bs A) \stackrel{\sim}{\to} Q'(S\bs A) $ such that $\,p'\,f = p\,$.
This morphism is automatically a weak equivalence in $ \DGA_S$, so $ \gamma f $ is an
isomorphism in $ \Ho(\DGA_S) $. It follows that $ \L(\gamma f)_V $ is an isomorphism in
$ \Ho(\cDGA_k)$. But $ Q(S\bs A) $ and $ Q'(S\bs A) $ are both cofibrant objects, so $ \L(\gamma f)_V =
\gamma(f_V) $ in $ \Ho(\cDGA_k) $. Thus $ f_V $ is a quasi-isomorphism in $\cDGA_k$.

\subsubsection{}
\la{S2.2.3}
The analogue of Theorem~\ref{S2T2} holds for the pair of functors
$\,
\rtv{\,\mbox{--}\,} :\ \DGA_S \rightleftarrows \DGA_k\,: \ \END\,V \otimes {\mbox{--}}
\,$,
which are adjoint to each other by Proposition~\ref{S2P1}$(a)$. Thus, $\, \rtv{\,\mbox{--}\,} \,$ has
the left derived functor
\begin{equation*}
\la{S2E15'}
\L \rtv{\,\mbox{--}\,}:\, \Ho(\DGA_S) \to \Ho(\DGA_k) \ ,\quad \L \rtv{S \bs A} := \rtv{Q(S \bs A)}\ ,
\end{equation*}
which is left adjoint to $\,\END\,V \otimes {\mbox{--}}\,$ on the homotopy category $ \Ho(\DGA_k) $.

\subsubsection{}
\la{S2.2.4}
If $V$ is a complex concentrated in degree $0$, the functors
$ (\,\mbox{--}\,)_V $ and $\,\End\,V \otimes \mbox{--} \,$ restrict to the category of
{\it non-negatively} graded DG algebras and still form the adjoint pair
\begin{equation*}
\la{plus}
(\,\mbox{--}\,)_V :\,\DGA^+_S \rightleftarrows \cDGA^+_k\,: \End\,V \otimes {\mbox{--}}\quad .
\end{equation*}
The categories $ \DGA_S^+ $ and $ \cDGA_k^+ $ have natural model structures (see Theorem~\ref{modax1}), for which
all the above results, including  Theorem~\ref{S2T2}, hold, with proofs being identical to the unbounded case.
Note that, by Proposition~\ref{3.2for}, a cofibrant resolution of $ S \bs A  $ in $  \DGA_S^+ $
is also a cofibrant resolution of $ S \bs A $ as an object in $ \DGA_S $. This implies that the derived functor of
the restriction $ (\,\mbox{--}\,)_V $ to $ \DGA_S^+ $ agrees with the restriction of the derived functor
$ \L(\,\mbox{--}\,)_V $ to $ \Ho(\DGA_S^+) \into \Ho(\DGA_S) $.
The advantage of working in $ \DGA_S^+ $ is that the cofibrant resolutions in this category can be chosen to be almost free extensions of the base DG algebra $S$. Since $ \DGA_S^+ $ contains $ \Alg_S $, using this kind of
resolutions suffices for applications. When dealing with ordinary algebras, we will thus often
restrict our DG representation functors to $ \DGA_S^+ $.

\subsection{Basic properties of $ \DRep_V(S \bs A)$}
\la{S2.3}

\subsubsection{}
\la{S2.3.1}
We begin by clarifying how $ \DRep_V $ depends on $V$.
Let $ \DGMod(S) $ be the category of DG modules over $S$, and let $V$ and $W$ be two modules in
$ \DGMod(S) $ each of which has finite dimension over $k$.

\bprop
\la{S2P3}
If $\,V$ and $W$ are quasi-isomorphic in $\DGMod(S)$, the corresponding derived functors
$ \L (\,\mbox{--}\,)_V $ and $ \L (\,\mbox{--}\,)_W :  \Ho(\DGA_S) \to \Ho(\cDGA_k) $ are naturally equivalent.
\eprop
The proof of this proposition is based on the following lemma, which is probably known to the experts.
\blemma
\la{S2L3}
Let $V$ and $W$ be two bounded DG modules over $S$, and assume that there is a quasi-isomorphism $\,f: V \stackrel{\sim}{\to} W \,$
in $\DGMod(S)$. Then the DG algebras $ \END\,V $ and $ \END\,W $ are weakly equivalent in $ \DGA_S $, i.e. isomorphic in $ \Ho(\DGA_S) $.
\elemma
\bproof
Write $ C := \Cyl(f) $ for the mapping cylinder of $\,f\,$: by definition, this is the graded $S$-module
$\, V \oplus V[1] \oplus W \,$ equipped with differential
$$
d(v_i, v_{i-1}, w_i) := (d_V(v_i) + v_{i-1}\,,\,- d_V(v_{i-1})\, ,\, d_W(w_i) - f(v_{i-1}))\ ,
$$
where $\,(v_i, v_{i-1}, w_i) \in V \oplus V[1] \oplus W \,$.
Both $V$ and $W$ naturally embed in $ C $ as DG modules. The embedding $\,W \into C \,$ is always a quasi-isomorphism,
while $ V \into C $ is a quasi-isomorphism whenever $f$ is a quasi-isomorphism (see \cite{GM}, Lemma~III.3.3).

Now, consider the DG algebra $ \END\,C $, and let $ \END(V)^{\sim} $ denote its DG subalgebra
consisting of endomorphisms that preserve the image of $ V $ in $ C $.  We claim that the natural maps
\begin{equation}
\la{nmaps}
\END\,V \,\stackrel{\sim}{\leftarrow}\, \END(V)^{\sim} \, \stackrel{\sim}{\into}\, \END\,C\ .
\end{equation}
are quasi-isomorphisms of DG $S$-algebras. Indeed, since the inclusion $ V \into C $ is a map of DG modules over $S$,
both arrows in \eqref{nmaps} are DG algebra maps over $S$. To see that these are actually quasi-isomorphisms
consider the commutative diagram of complexes
\begin{equation}
\la{dmaps}
\begin{diagram}[small, tight]
        &           &     \HOM(C,V)     &             & \\
        &\ldOnto    &       \dInto      &  \rdInto    &  \\
\END\,V & \lTo      & \END(V)^{\sim}    &  \rInto       & \END\,C\ ,
\end{diagram}
\end{equation}
induced by $ V \into C $.  Since $ V \into C $ is a quasi-isomorphism, its cokernel
$ C\!/V $ is acyclic, and therefore the endomorphism ring $ \END(C\!/V) $ is acyclic
({\it cf.} \cite{FHT}, Prop.~2.3(ii)). It follows from the natural exact sequence
$$
0 \to \HOM(C,V) \to \END(V)^{\sim} \to \END(C\!/V) \to 0
$$
that the vertical arrow in \eqref{dmaps} is a quasi-isomorphism. A similar argument shows the other two arrows in \eqref{dmaps}
starting from $ \HOM(C,\,V) $ are also quasi-isomorphisms. By commutativity of the diagram, the two
horizontal arrows in \eqref{dmaps}, which are the maps \eqref{nmaps}, must then be quasi-isomorphisms as well.

In a similar fashion, replacing $V$ by $W$, we get
\begin{equation}
\la{nmaps1}
\END\,W \,\stackrel{\sim}{\leftarrow}\, \END(W)^{\sim} \, \stackrel{\sim}{\into}\, \END\,C\ .
\end{equation}
From \eqref{nmaps} and \eqref{nmaps1}, we see that $ \END\,V \cong \END\,C \cong \END\,W $ in $ \Ho(\DGA_S) $.
\eproof
\begin{proof}[Proof of Proposition~\ref{S2P3}]
By Theorem~\ref{S2T2}$(c)$, the functors $\L (\,\mbox{--}\,)_V $ and $ \L (\,\mbox{--}\,)_W\,$
are left adjoint to the functors $\, \END\,V \otimes {\mbox{--}} \,$ and
$\, \END\,W \otimes {\mbox{--}} \,$ on the homotopy category $\, \Ho(\cDGA_k)\,$.
By Lemma~\ref{S2L3}, these last two functors are isomorphic if $V$ and $W$ are
quasi-isomorphic.
Hence, by uniqueness of the adjoint functors, $\,\L (\,\mbox{--}\,)_V $ and $ \L (\,\mbox{--}\,)_W\,$
are isomorphic as functors on $ \Ho(\DGA_S)$.
\end{proof}
As an immediate consequence of Proposition~\ref{S2P3}, we get
\begin{corollary}
\la{S2C1}
If $ V $ and $ W $ are quasi-isomorphic $S$-modules, then
$\,\DRep_V(S \bs A) \cong \DRep_W(S \bs A) \,$ for any algebra
$\, S \bs A\in \Alg_S $. In particular, $\,\H_\bullet(S \bs A,V)
\cong \H_\bullet(S \bs A,W)\,$ as graded algebras.
\end{corollary}

\subsubsection{Base change}
\la{S2.4}
If $S $ is an ordinary algebra, and $f: S \to A $ is a fixed morphism in $ \Alg_k $, the classical
(relative) representation scheme $ \Rep_V(S\bs A) $ can be identified with the fibre of the
restriction map $\,f_*:\, \Rep_V(A) \to \Rep_V(S) \,$ over a given representation $\,\varrho:\,
S \to \End\,V \,$. Dually, this means that the coordinate algebra $ (S \bs A)_V $ of
$ \Rep_V(S\bs A) $ fits in the commutative diagram
\begin{equation}
\la{drepd1}
\begin{diagram}[small, tight]
S_V              &  \rTo^{f_V}        & A_V\\
\dTo^{\varrho_V} &                    & \dTo \\
k                &  \rInto            & (S\bs A)_V
\end{diagram}
\end{equation}
which is a universal cocartesian square in $ \cAlg_k $.
We want to give a similar universal characterization for the derived representation scheme $ \DRep_V(S\bs A) $.
%
%

Let $\, R \xrightarrow{g} S \xrightarrow{f} A\,$ be morphisms in $\DGA_k$. Fix a DG representation
$\,\varrho:\, S \to \END\,V $, and let $\,\varrho_R := \varrho \circ g\,$. Using $ \varrho $
and $ \varrho_R $, define the representation functors $\,(S \bs\,\mbox{--}\,)_V:\,\DGA_S \to \cDGA_k\,$,
and $\, (R\bs\,\mbox{--}\,)_V :\,\DGA_R \to \cDGA_k \,$ and consider the corresponding derived functors
$ \L(S\bs\,\mbox{--}\,)_V $ and $ \L(R\bs\,\mbox{--}\,)_V $.
\blemma
\la{S2L9}
The commutative diagram
\begin{equation}
\la{drepd11}
\begin{diagram}[small, tight]
(R\bs S)_V              &  \rTo^{\,(R\bs f)_V\,}            & (R\bs A)_V\\
\dTo^{(R\bs \varrho)_V} &                                   & \dTo \\
k                       &  \rInto                           & (S\bs A)_V
\end{diagram}
\end{equation}
is a cocartesian square in $ \cDGA_k $.
\elemma
\bproof
To simplify the notation set $ \C := \cDGA_k $ and denote by $ \C^\D $ the category of diagrams
of shape $ \{\ast \leftarrow \ast \rightarrow \ast \}\,$ in $ \C $. The pushout construction
defines a functor $\, \colim:\,\C^\D \to \C \,$ which is left adjoint to the diagonal functor
$ \Delta:\,\C \to \C^\D $ ({\it cf.} Example~\ref{colim}). We need to show that
\begin{equation}
\la{coliso}
(S\bs A)_V \cong \colim[\,k \xleftarrow{\,\varrho_V\,} (R\bs S)_V  \xrightarrow{\,f_V\,} (R\bs A)_V \,]\ .
\end{equation}
To this end, we denote the above colimit by $ \colim(\varrho_V,\,f_V) $,
choose an arbitrary $ C \in \cDGA_k $ and identify
\begin{eqnarray*}
\Hom_{\C}(\colim(\varrho_V,\,f_V),\,C) \,\, & \cong & \Hom_{\C^{\D}}((\varrho_V,\,f_V),\,\Delta C)\\
                          &\cong  & \{\alpha \in \Hom_{\C}((R \bs A)_V,\,C)\ :\ \alpha\,f_V = \varrho_V\}\\
                          &\cong& \{\tilde{\alpha} \in \Hom_{\DGA_R}(A,\,\End\,V \otimes C)\ :\ \tilde{\alpha}\,f = \varrho\}\\
                          &\cong& \Hom_{\DGA_S}(A,\,\End\,V \otimes C) \\
                          &\cong& \Hom_{\cDGA_k}((S \bs A)_V,\,C)\ ,
\end{eqnarray*}
where the third and the last bijections are given by Proposition~\ref{S2P1}$(b)$.
The isomorphism \eqref{coliso} follows now by Yoneda's Lemma.
\eproof

\begin{theorem}
\la{hopushout}
There is a commutative diagram in $ \Ho(\cDGA_k) $
\begin{equation}
\la{drepd2}
\begin{diagram}[small, tight]
\L(R\bs S)_V           &  \rTo^{\,\L(R\bs f)_V\,}        & \L(R\bs A)_V\\
\dTo^{\L(\varrho)_V} &                    & \dTo \\
k                &  \rInto            & \L(S\bs A)_V
\end{diagram}
\end{equation}
which is universal in the sense that $ \L(S\bs A)_V $ is
the homotopy cofibre of $ (R\bs f)_V $.
\end{theorem}
\begin{proof}
By Lemma~\ref{S2L9}, there is an isomorphism of functors $\, \DGA_S \to \cDGA_k \,$:
\begin{equation}\la{lccol0}
(S\bs \, \mbox{--}\,)_V  \cong \colim[\,k \xleftarrow{} (R\bs S)_V
\xrightarrow{\,(R\bs\, \mbox{--}\,)_V\,} (R\bs \, \mbox{--}\,)_V \,]\ .
\end{equation}
By uniqueness of the derived functors, \eqref{lccol0} induces an isomorphism of
functors $\, \Ho(\DGA_S) \to \Ho(\cDGA_k)
\,$:
\begin{equation}\la{lccol}
\L(S\bs \, \mbox{--}\,)_V  \cong \L\colim[\,k \xleftarrow{\,\varrho_V\,} (R\bs S)_V
\xrightarrow{\,(R\bs\, \mbox{--}\,)_V\,} (R\bs \, \mbox{--}\,)_V \,] \ .
\end{equation}
Now, to compute $\,\L \colim \,$ in \eqref{lccol} we should replace $ f:\, S \to A $ by
its cofibrant resolution $ i:\, S \to Q(S\bs A) $ in $ \DGA_S $; then
\begin{equation}
\la{lccol1}
\L\colim[\,k \xleftarrow{\,\varrho_V\,} (R\bs S)_V  \xrightarrow{\,f_V\,} (R\bs A)_V \,]
\cong \gamma\,\colim[\,k \xleftarrow{\,\varrho_V\,} (R \bs S)_V  \xrightarrow{\,i_V\,} (R\bs Q(S \bs A))_V\,]\ ,
\end{equation}
where the last `$ \colim $' is taken in $ \cDGA_k $ and $\,\gamma: \cDGA_k \to \Ho(\cDGA_k) \,$.

We will use a specific resolution $ i:\, S \to Q(S\bs A) $ in \eqref{lccol1}, which we construct in the following way.
First, we choose a cofibrant resolution $\,p_{R\bs S}:\, Q(R\bs S) \sonto S $ of $ S $ in $\,\DGA_R\,$.
Then, by \cite{FHT}, Prop.~3.1, there is a cofibration
$\,\tilde{f}:\, Q(R\bs S) \into  Q(R \bs A) \,$ in $ \DGA_k $ and a surjective quasi-isomorphism
$\,p_{R\bs A}:\, Q(R\bs A) \sonto A $ such that $\, p_{R\bs A} \circ \tilde{f} =
f \circ p_{R\bs S} \,$. We take $ p_{R\bs A} $ to be a cofibrant resolution of $ A $ in $ \DGA_R$.
Now, we define $\, i: S \to  Q(S\bs A) $ by pushing out $ \tilde{f} $ along $ p_{R\bs S} $ in $ \DGA_R \,$:
\begin{equation}
\la{drepd3}
\begin{diagram}[small, tight]
Q(R \bs S)                  &  \rInto^{\ \tilde{f}\ }           & Q(R \bs A) \\
\dOnto^{\,p_{R \bs S}\,}   &                                   & \dTo^{\tilde{\varphi}} \\
S                           &  \rInto^{i\ }                     & Q(S\bs A)
\end{diagram}
\end{equation}
By the universal property of pushouts, there is a (unique) morphism $\,p_{S \bs A}:\,Q(S\bs A) \to A \,$
in $ \DGA_R $ such that $\,p_{S \bs A} \circ i = f \,$ and  $\,p_{S\bs A} \circ \tilde{\varphi} = p_{R\bs A}\,$.
We claim that $ p_{S \bs A} $ is a cofibrant resolution of $ f $ in $ \DGA_S $. Indeed,
by Lemma~\ref{popb}, $\,i\,$ is a cofibration in $ \DGA_k $, so $ Q(S \bs A) $ is a cofibrant object
in $ \DGA_S $. On the other hand,  $\, \tilde{\varphi} \,$  is a pushout of a quasi-isomorphism
along a cofibration in $ \DGA_k $, hence, by Proposition~\ref{dgproper}, this map is
a quasi-isomorphism as well. Finally, $\,p_{S\bs A} \circ \tilde{\varphi} = p_{R\bs A}\,$ implies that
$ p_{S\bs A} $ is a surjective quasi-isomorphism, since so is $ p_{R \bs A} $.

Now, since $\,(R\bs \, \mbox{--}\,)_V: \DGA_R \to \cDGA_k\,$ is a left adjoint functor,
it preserves colimits: hence, it follows from \eqref{drepd3} that
\begin{equation}
\la{drepd33}
(R \bs Q(S\bs A))_V  \cong \colim[\,(R\bs S)_V \xleftarrow{} Q(R\bs S)_V \xrightarrow{\tilde{f}_V} Q(R\bs A)_V \,]
\end{equation}
On the other hand, since $\,(R\bs \, \mbox{--}\,)_V \,$ is a left Quillen functor, it preserves cofibrations:
hence $ \tilde{f}_V $ in \eqref{drepd33} is a cofibration in $ \cDGA_k $ since so is $ \tilde{f} $ in \eqref{drepd3}.
Combining \eqref{lccol}, \eqref{lccol1} and \eqref{drepd33}, we get
\begin{eqnarray*}
\L(S\bs A\,)_V
&\cong&
\gamma\,\colim[\,k \xleftarrow{ \tilde{\varrho}_V } Q(R\bs S)_V  \xrightarrow{\tilde{f}_V} Q(R\bs A)_V \,] \nonumber \\*[1ex]
&\cong&
\colim[\,k \xleftarrow{ \gamma \tilde{\varrho}_V } \gamma\,Q(R\bs S)_V
\xrightarrow{\gamma \tilde{f}_V} \gamma\, Q(R\bs A)_V \,] \la{drepd34} \\*[1ex]
&\cong &
\colim[\,k \xleftarrow{\,\L(\varrho)_V\,} \L(R\bs S)_V  \xrightarrow{\,\L(f)_V\,} \L(R\bs A)_V \,] \la{drepd35} \ ,
\end{eqnarray*}
where the second isomorphism is a consequence of Proposition~\ref{prmcat}, and the third holds by Theorem~\ref{S2T2}.
Note that Proposition~\ref{prmcat} applies in our situation, since $ \cDGA_k $ is a proper model category
({\it cf.} Proposition~\ref{dgproper}).
\end{proof}

Let us state the main corollary of Theorem~\ref{hopushout}, which may be viewed
as an alternative definition of $ \DRep_V(S\bs A) $. This corollary shows that our
construction of relative $ \DRep_V $ is a `correct' one from homotopical point of view
({\it cf.} \cite{Q2}, Part~I, 2.8).
\begin{corollary}
\la{hlimr}
For any $ (S \xrightarrow{f} A) \in \Alg_S $, $\,\DRep_V(S\bs A) \,$ is the homotopy
cofibre of $\, f_V \,$, {\it i.e.}
\begin{equation*}
\begin{diagram}[small, tight]
\DRep_V(S)          &  \rTo        &  \DRep_V(A)\\
\dTo &                    & \dTo \\
k                &  \rInto            & \DRep_V(S\bs A)
\end{diagram}
\end{equation*}
\end{corollary}

\subsubsection{A cofibration spectral sequence}
\la{S2.41}
The above result suggests that the homology of $ \DRep_V(S\bs A) $ should be related to
the homology of $ \DRep_V(S) $ and $ \DRep_V(A) $ through a standard spectral sequence
associated to a cofibration. To simplify matters we will assume that $V$ is a $0$-complex
and work in the category $ \DGA_k^+$ of non-negatively graded DG algebras
({\it cf.} Remark~\ref{S2.2.4}).
\begin{theorem}
\la{EMsps}
Given $\, R \xrightarrow{} S \xrightarrow{} A\,$ in $\DGA^+_k$ and a representation $ S \to \End(V) $,
there is an Eilenberg-Moore spectral sequence with
$$
E^2_{\ast,\,\ast} = \Tor_{\ast,\,\ast}^{\H_\bullet(R\bs S, V)}(k,\,\H_\bullet(R\bs A, V))
$$
converging to $ \H_{\bullet}(S\bs A, V) $.
\end{theorem}
\begin{proof}
To construct this spectral sequence we will use the cofibrant resolutions
constructed in the proof of Theorem~\ref{hopushout}.
In addition, we will assume that $ Q(R\bs S) $ is chosen to be almost free
in $ \DGA_R^+ $ ({\it cf.}  \cite{FHT}, Proposition~3.1). 
Then the pushout diagram \eqref{drepd3} gives the pushout
in the category $ \cDGA_k^+ $:
\begin{equation}
\la{dred}
\begin{diagram}[small, tight]
Q(R\bs S)_V   &  \rTo^{\tilde{f}_V} & Q(R \bs A)_V \\
\dTo^{\tilde{\varrho}_V} &          & \dTo \\
k             &  \rInto             & Q(S \bs A)_V
\end{diagram}
\end{equation}
where $\, \tilde{\varrho}_V :\,Q(R\bs S)_V \to  k \,$ is the map adjoint to
$\,\tilde{\varrho} := \varrho \circ p_{R\bs S}: Q(R\bs S) \to \End\,V \,$.
It follows from \eqref{dred} that
$$
k \otimes_{Q(R \bs S )_V} Q(R \bs A)_V \cong Q(S\bs A)_V \ .
$$
Moreover, for the almost free DG resolutions $\, Q(R\bs S),\,Q(R \bs A)\,$ and $\, Q(S \bs A)\,$,
the commutative DG algebras $ Q(R\bs S)_V,\,Q(R \bs A)_V $ and $ Q(S \bs A)_V $
are almost free. The existence of the Eilenberg-Moore spectral is now standard (see, e.g.,
\cite{McC}, Theorem~7.6).
\end{proof}

\subsubsection{The zero homology}
\la{S2.3.3}
The next result shows that $ \DRep_V(S \bs A) $ is indeed the `higher' derived functor of
the classical representation scheme $ \Rep_V(S\bs A) $ in the sense of homological algebra.
\bthm
\la{S2T4}
Let $ S \in \Alg_k $ and $ V $ concentrated in degree $0$. Then, for any $ S \bs A \in \Alg_S $,
$$
\H_0(S\bs A,V) \cong (S\bs A)_V
$$
where $(S\bs A)_V$ is a commutative algebra representing $ \Rep_V(S \bs A) $, see \eqref{S2E12}.
\ethm
\bproof
Fix a cofibrant resolution $ Q(S \bs A) \sonto S \bs A $ in $ \DGA^+_S $. Then, for any
$ C \in \cAlg_k $, there are natural isomorphisms
\begin{eqnarray*}
\Hom_{\cAlg_k}(A_V,\,C) \,\,& \cong & \Hom_{\Alg_S}(A,\,\End\,V \otimes C)\\
                          &\cong  & \Hom_{\DGA^+_S}(Q(S \bs A),\, \End\,V \otimes C)\\
                          &\cong& \Hom_{\cDGA^+_k}(Q(S \bs A)_V ,\, C)\\
                          &\cong& \Hom_{\cAlg_k}(\H_0[Q(S \bs A)_V],\, C)\ ,
\end{eqnarray*}
where the first isomorphism is \eqref{S2E11'}; the second
follows from the fact that $ \End\,V  \otimes C $ is concentrated in degree $0$;
the third is the result of Proposition~\ref{S2P1}$(b)$ and the last follows again from the
fact that $C$ is concentrated in degree $0$. By Yoneda's Lemma, we now conclude that
$ \H_0[Q(S \bs A)_V] \cong (S \bs A)_V $. On the other hand, by definition of $ \DRep_V(S\bs A)$,
$\,\H_0[Q(S\bs A)_V] \cong \H_0[\DRep_V(S\bs A)] = \H_0(S\bs A,\,V)\,$.
\eproof

\vspace{1ex}

\remark\
Theorem~\ref{S2T4} shows that $ \DRep_V(S\bs A) $ is trivial whenever $ \Rep_V(S \bs A) $ is trivial.
Indeed, if $ \Rep_V(S \bs A) $ is empty, then $ (S\bs A)_V = 0 $. By Theorem~\ref{S2T4}, this means
that $\,1 = 0\,$ in $ \H_\bullet[\DRep_V(S\bs A)] $ and hence $ \H_\bullet[\DRep_V(S\bs A)] $ is the zero
algebra. This, in turn, means that $ \DRep_V(S\bs A) $ is acyclic and hence $ \DRep_V(S \bs A) = 0 $
in $\Ho(\cDGA^+_k) $ as well.
\begin{example}
\la{S2Ex1}
The above remark applies, for example, to the Weyl algebra $\, A = k \langle x, y \rangle/(xy-yx-1) \,$.
In fact, since $k$ has characteristic zero, $A$ has no (nonzero) finite-dimensional modules.
So $ \Rep_V(A) $ is empty and $ \DRep_V(A) = 0 $ for all $ V $.
\end{example}
\subsubsection{$\GL(V)$-invariants}
\la{S2.3.4}
We will keep the assumption that $V$ is a $0$-complex and assume, in addition, that $S=k$.
Let $ \GL(V) \subset \End(V) $ denote, as usual, the group of invertible endomorphisms of $V$.
Consider the right action of $ \GL(V) $ on $ \End(V) $ by conjugation, $\,\alpha \mapsto g^{-1} \alpha g \,$,
and extend it naturally to the functor $\,\End\,V \otimes \mbox{--}\,$; thus, we have a
homomorphism from $ \GL(V) $ to the opposite automorphism group of this functor. Now,
by `general nonsense', the automorphism groups of adjoint functors are isomorphic.
Hence, if we regard $\,\End\,V \otimes \mbox{--}\,$ as functor $\,\cDGA_k \to \DGA_k \,$, then
the right action of $ \GL(V) $  on $\,\End\,V \otimes \mbox{--}\,$ translates (through the adjunction of
Proposition~\ref{S2P1}$(b)$) to a (left) action on $\, (\,\mbox{--}\,)_V: \DGA_k \to \cDGA_k \,$. Using this last action, we define the invariant subfunctor
\begin{equation}
\la{S2E16}
(\,\mbox{--}\,)_V^{\GL}\,:\ \DGA_k \to \cDGA_k\ , \quad A \mapsto A_V^{\GL(V)}\ .
\end{equation}
On the other hand, we can also regard $\,\End\,V \otimes \mbox{--}\,$ as a functor on the homotopy category $ \Ho(\cDGA_k) $ as in Theorem~\ref{S2T2}$(c)$. Then, using Quillen's adjunction \eqref{S2E15}, we can transfer the action of $ \GL(V) $ on $\,\END\,V \otimes
\mbox{--}\,$ to the derived functor $ \L (\,\mbox{--}\,)_V $. It follows that $\GL(V)$  acts naturally on the homology $ \H_{\bullet}(A,V) $.
\bthm
\la{S2P4}
$(a)$ The functor \eqref{S2E16} has a total left derived functor\footnote{Abusing notation, we will often write $ \DRep_V(A)^{\GL} $ for $ \L(A)_V^{\GL} $.}
$$
\L(\,\mbox{--}\,)_V^{\GL}:\,\Ho(\DGA_k) \to \Ho(\cDGA_k)\ .
$$

$(b)$ For any $ A \in \DGA_k $, there is a natural isomorphism of graded algebras
$$
\H_\bullet[\L(A)_V^{\GL}] \cong \H_\bullet(A, V)^{\GL(V)}\ .
$$
\ethm

Before proving Theorem~\ref{S2P4}, we note that, unlike $\,(\,\mbox{--}\,)_V\,$,
the functor $ (\,\mbox{--}\,)_V^{\GL} $ does not seem to have a right adjoint, so we
can't deduce the existence of its left derived functor from Quillen's Adjunction
Theorem; instead, we will use another general result from homotopical algebra --
the Brown Lemma~\ref{BrL}.

First, we recall the notion of polynomial homotopy between DG algebra maps. Let
$ \Omega := k[t] \oplus k[t]\,dt $ denote the algebraic de Rham complex of the affine
line\footnote{Following our general convention, we assume that the differential on
$ \Omega $ has degree $-1$.}.
We say that two maps $\,f,g:\,A \to B $ are {\it polynomially homotopic}
if there is a DG algebra map (homotopy) $\, h: A \to B \otimes \Omega \,$ such that
$ h(0) = f $ and $ h(1) = g $, where $\, h(a): A \to B \,$, $\, a \in k \,$, denotes
the composition of $h$ with the quotient map $\,B \otimes \Omega \onto B \otimes
\Omega/(t-a) \cong B \otimes k = B\,$. We refer the reader to the Appendix,
Sect.~\ref{3.4}, for properties of polynomial homotopy which we will
use in the proof of the following lemma.
\begin{lemma}
\la{rephom}
Let $\,h:\, A \to B \otimes \Omega $ be a polynomial homotopy between maps $\,f:= h(0) $
and $g:=h(1)$ in $ \DGA_k $. Then \\
$(i)$  There is a map $\,h_V:\,A_V \rar B_V \otimes \Omega $ in $\cDGA_k$ such that
$\,h_V(0)=f_V$ and $h_V(1)=g_V$.\\
$(ii)$ $\ h_V$ restricts to a morphism $\,h_V^{\GL}:\,A_V^{\GL} \rar B_V^{\GL} \otimes \Omega $ in $\cDGA_k$.
\end{lemma}
\begin{proof}
Given a DG algebra $ A \in \DGA_k $, let $ \pi_A :\, A \rar \End\,V \otimes A_V \,$ denote the universal
representation corresponding to the identity map under the adjunction of Proposition~\ref{S2P1}$(b)$.
Then the composite map
$$
A \stackrel{h}{\rar} B \otimes \Omega \xrightarrow{\pi_B \otimes \id_{\Omega}} \End(V) \otimes B_V \otimes \Omega
$$
corresponds to the morphism $\,h_V:\,A_V \to  B_V \otimes \Omega \,$ which satisfies the conditions of $(i)$.

To prove $(ii)$, it suffices to show that $h_V$ is $\GL(V)$-equivariant. For this, we fix
$\,g \,\in \,\GL(V)\,$ and verify that $\,h_V \circ g = (g \otimes \id_{\Omega}) \circ h_V\,$ as maps
$\,A_V \rar B_V \otimes \Omega $. By Proposition~\ref{S2P1}, it suffices to show that the outer square
in the following diagram commutes:
\begin{equation*}
\la{pushde}
\begin{diagram}[small, tight]
A                    & \rTo^{\pi_A}    & \End(V) \otimes A_V  & \rTo^{\id_{\End(V)} \otimes g} & \End(V) \otimes A_V \\
\dTo^{h}             &         & \dTo^{\id_{\End(V)} \otimes h_V} &                  & \dTo^{\id_{\End(V)} \otimes h_V} \\
B \otimes \Omega  & \rTo^{\,\pi_B \otimes \id_{\Omega}\,}    & \End(V) \otimes B_V \otimes \Omega &
\rTo^{\,\Ad_g \otimes \id_{B_V \otimes \Omega}\,}             &  \End(V) \otimes B_V \otimes \Omega
\end{diagram}
\end{equation*}
But the two inner squares in this diagram obviously commute, hence so does the outer square.
\end{proof}

\begin{proof}[Proof of Theorem~\ref{S2P4}]
$(a)$ By Lemma~\ref{BrL}, it suffices to show that $\,(\,\mbox{--}\,)_V^{\GL}$ maps any
acyclic cofibration $\,i: A \sinto B \,$ between cofibrant objects in $\DGA_k$ to a weak equivalence in
$\cDGA_k$. By \ref{2.4.1}, given $\,i: A \sinto B \,$, there is a map $\,p:B \rar A\,$ in $\DGA_k$
such that $\,pi=\id_A\,$ and $ip$ is homotopic to $\id_B$. Hence, $p_V^{\GL} \circ i_V^{\GL}= \id_{A_V^{\GL}}$.
On the other hand, by Proposition~\ref{lhom}, there is a morphism $\,h:B \rar B \otimes \Omega $ in $\DGA_k$
with $\,h(0)=\id_B\,$ and $\,h(1)=ip\,$. Lemma~\ref{rephom}$(ii)$ yields a morphism
$h_V^{\GL}:B_V^{\GL} \rar B_V^{\GL} \otimes \Omega $ such that $h_V(0)=\id_{B_V^{\GL}}$ and $h_V(1)=i_V^{\GL} \circ p_V^{\GL}$.
Finally, by Remark~\ref{3.4.4}, $i_V^{\GL} \circ p_V^{\GL}$ induces the identity map at the level of homology.
Hence $i_V^{\GL}$ is a weak equivalence in $\cDGA_k$. This proves $(a)$.

$(b)$ The functor $ \L[(\,\mbox{--}\,)_V^{\GL}] $ maps $\,\gamma A \in \Ho(\DGA_k)\,$ to the class
of the invariant subalgebra $ (QA_V)^{\GL(V)} $ of $ (QA)_V $, where $ QA $ is a cofibrant replacement of $A$. By $(a)$, this
map is well-defined, i.e. does not depend on the choice of $ QA $.
On the other hand, $ \mathrm{H}_{\bullet}(\mbox{--}, V)^{\GL} $ maps $ \gamma A $ to $\mathrm{H}_{\bullet}( [\gamma(QA_V)])^{\GL(V)} $. Since $k$ is of characteristic $0$ and since $\GL(V)$ is reductive,
$\,\mathrm{H}_{\bullet}(\gamma[(QA_V)^{\GL(V)}]) \cong [\mathrm{H}_{\bullet}(\gamma(QA_V))]^{\GL(V)} \,$ for all
$A \in \DGA_k$. This proves part $(b)$.
\end{proof}

\remark\ The above argument (based on Lemma~\ref{rephom}$(i)$) gives an alternative proof of the existence
of the derived functor $ \L(\,\mbox{--}\,)_V $ on $ \Ho(\DGA_S) $, which is part $(b)$ of Theorem~\ref{S2T2}.
However, the other two parts of Theorem~\ref{S2T2}, including the adjunction at the level of homotopy categories,
do not follow from Lemma~\ref{rephom} and need a separate proof.

\subsubsection{Relation to the Ciocan-Fontanine-Kapranov construction}
\la{S2.3.5} For an ordinary $k$-algebra $A$ and a $k$-vector space $V$, Ciocan-Fontanine and
Kapranov introduced a derived affine scheme, $ \RAct(A, V) $, which they called the {\it derived space of actions of $A$}
(see \cite{CK}, Sect.~3.3). Although the construction of $ \RAct(A, V) $ is quite different from our construction
of $ \DRep_V(A) $, Proposition~3.5.2 of \cite{CK} shows that, for a certain {\it specific} resolution of $A$,
the DG algebra $ k[\RAct(A, V)] $ satisfies the adjunction of Proposition~\ref{S2P1}$(b)$.
Since $ k[\RAct(A, V)] $ and $ \DRep_V(A) $ are independent of the choice of resolution, we conclude
\begin{theorem}
\la{S2T44}
If $\,A \in \Alg_k \,$ and $V$ is a $0$-complex,
then $\, k[\RAct(A,\,V)] \cong \DRep_V(A) \,$ in $ \Ho(\cDGA^+_k) $.
\end{theorem}
We should mention that the fact that $ k[\RAct(A, V)] $ is independent of resolutions was proved in
\cite{CK} by an explicit but fairly involved calculation using spectral sequences. Strictly speaking, this
calculation does not show that $ \RAct(\,\mbox{--}\,, V) $ is a Quillen derived functor (in the sense of
Theorem~\ref{Qthm}). In combination with Theorem~\ref{S2T44}, our main Theorem~\ref{S2T2} can thus be
viewed as a strengthening of \cite{CK} as it implies that $ \RAct(A, V) $ is indeed a (right) Quillen
derived functor on the category of DG schemes.

\subsection{Explicit presentation}
\la{S2.5}
Let $ A \in \Alg_k $. Given an almost free resolution $ R \sonto A $ in $\DGA^+_k$, the DG algebra $ R_V $ can be
described explicitly. To this end, we extend a construction of Le Bruyn and van de Weyer (see \cite{LBW}, Theorem~4.1)
to the case of DG algebras. Assume, for simplicity, that $ V = k^d $. Let $\{x^\alpha\}_{\alpha \in I}$ be a set of
homogeneous generators of an almost free DG algebra $R$, and let $ d_R $ be its differential. Consider a free graded algebra
$\tilde{R}$ on generators $\,\{x^{\alpha}_{ij}\,:\, 1\leq i,j\leq d\, ,\, \alpha \in I\}\,$, where
$\,|x^{\alpha}_{ij}|=|x^{\alpha}|\,$ for all $ i,j $. Forming matrices $ X^\alpha :=
\| x^{\alpha}_{ij}\| $ from these generators, we define the algebra map
$$
\pi:\, R \to M_d(\tilde{R})\ , \quad  x^\alpha \mapsto X^\alpha \ ,
$$
where $ M_d(\tilde{R}) $ denotes the ring of $ (d \times d)$-matrices with entries in $\tilde{R}$.
Then, letting $\, \tilde{d}(x^{\alpha}_{ij}) := \| \pi(d x^{\alpha}) \|_{ij}\,$, we define a differential
$ \tilde{d} $ on generators of $ \tilde{R} $ and extend it to the whole of $ \tilde{R}$ by linearity
and the Leibniz rule. This makes $ \tilde{R} $ a DG algebra. The abelianization of
$\tilde{R} $ is a free (graded) commutative algebra generated by (the images of) $\,x^\alpha_{ij}\,$ and the differential
on $ \tilde{R}_\nn $ is induced by $ \tilde{d}$.
\bthm
\la{comp}
There is an isomorphism of DG algebras $\,  \rtv{R} \cong \tilde{R} \,$. Consequently, $\,R_V \cong \tilde{R}_\nn\,$.
\ethm
\begin{proof}
By Proposition~\ref{S2P1}$(a)$, it suffices to show that
$$
\Hom_{\DGA^+_k}(\tilde{R},B) \,\cong\, \Hom_{\DGA^+_k}(R,\,M_d(B))
$$
for any DG algebra $ B \in \DGA_k^+ $. Given $\,\phi \,\in\, \Hom_{\DGA^+_k}(\tilde{R},B)\,$, define
$\tau(\phi) \,\in\,  \Hom_{\DGA^+_k}(R,\,M_d(B))$ by
$$
\|\tau(\phi)(x^\alpha)\|_{ij} \,=\, \phi(x^\alpha_{ij}) \ .
$$
Conversely, given $f \,\in\,  \Hom_{\DGA^+_k}(R,\,M_d(B))$, define $ \eta(f)\,\in\,  \Hom_{\DGA^+_k}(\tilde{R},\,B)$ by
$$
\eta(f)(x^\alpha_{ij})\,=\, \|f(x^\alpha)\|_{ij} \ .
$$
It is easy to check that both $\eta(f)$ and $\tau(\phi)$ respect differentials,
and the maps $\tau$ and $\eta$ are mutually inverse bijections.
\end{proof}
Using Theorem~\ref{comp}, one can construct a finite presentation for $ R_V $
whenever a finite almost free resolution $ R \to A $ is available. In practice, many interesting algebras admit such resolutions.
For example, the DG algebras introduced recently in \cite{G1} and \cite{K2} provide very interesting (in a sense, canonical)
finite resolutions for many 2D and 3D Calabi-Yau algebras.
\begin{example}\la{Ex2D}
Let $\,A = k[x,y]\,$ be the polynomial algebra in two variables. One can check ({\it cf.} \cite{G1}, Corollary~3.6.5) that $A$
has an almost free resolution in $ \DGA_k^+ $ of the form $\,R := k\langle x,\,y;\,t\rangle \,$, where $\,x,y\,$ have degree $0$,
$\,t$ has degree $1$, and the differential on $ R $ is defined by $\,dx = dy = 0\,$,  $\,dt=xy-yx\,$.
For $ V = k^d $, Theorem~\ref{comp} implies that $\, R_V \cong k[x_{ij},\,y_{ij};\,t_{ij}]\,$, where the generators
$\,x_{ij},\,y_{ij} \,$ have degree $0$, $\,t_{ij} $'s have degree $ 1$, and the differential on $ R_V $ is defined by
$$
dx_{ij} = dy_{ij} = 0 \ ,\quad dt_{ij} = \sum_{k=1}^d\, (x_{ik}y_{kj} - y_{ik}x_{kj})\ .
$$
It is convenient to use the matrix notation $\, X := \|x_{ij}\|$, $\, Y := \|y_{ij}\|$, $\, T := \|t_{ij}\|$,
writing the above relations in the form
\begin{equation}
\la{cdpol}
dX = dY = 0 \ , \quad  dT = [X,\,Y] \ .
\end{equation}

Now, for small $\,d \,$, we can actually compute the homology of $\, R_V\,$. If $ d=1 $, $\, R_V\,$ is just the abelianization of $ R $. Hence, $\,\H_{\bullet}(k[x,y], k) \cong 
k[x,y] \otimes k[t]/(t^2)\,$, where $t$ has homological degree $1$.  For $ d=2 $, the result is more complicated:
\begin{equation}
\la{reph2}
\mathrm{H}_{\bullet}(k[x,y], k^2)\,\cong\, (k[x,y]_2 \otimes \Lambda_k \mathfrak{g})/\mathbf{I}\ .
\end{equation}
Here $\,k[x,y]_2 \,$ is the coordinate ring of $\, \Rep_2(k[x,y])\,$ that has eight commuting generators $\,x_{11},\,x_{12},\,x_{21},\,x_{22} \,$ and $\,y_{11},\,y_{12},\,y_{21},\,y_{22} \,$,
each of homological degree $0$,  satisfying
\begin{eqnarray*}
&& x_{12} \,y_{21} - y_{12} \,x_{21} = 0\ ,\\
&& x_{11}\, y_{12} + x_{12} \,y_{22} - y_{11} \,x_{12} - y_{12}\, x_{22} = 0\ ,\\
&& x_{21} \,y_{11} + x_{22}\, y_{21} - y_{21}\, x_{11} - y_{22} \, x_{21} = 0\ ;
\end{eqnarray*}
$\,\mathfrak{g}\,$ is a vector space spanned by three generators $ \xi $, $\, \eta $ and $ \tau $ of homological degree $1$. The ideal of relations $\, \mathbf{I} \subseteq k[x,y]_2 \otimes \Lambda_k \mathfrak{g}\,$ is generated by \\

(1)\quad $x_{12}\,\eta - y_{12}\,\xi \,=\, (x_{12}y_{11}-y_{12}x_{11})\,\tau \,$,\\

(2)\quad $x_{21}\,\eta - y_{21}\,\xi \,=\, (x_{21}y_{22}-y_{21}x_{22})\,\tau\,$,\\

(3)\quad $(x_{11}-x_{22})\,\eta- (y_{11}-y_{22})\,\xi \,=\, (x_{11}y_{22}- y_{11}x_{22})\,
\tau\,$,\\

(4)\quad $\xi \eta\,=\, y_{11}\,(\xi\tau) - x_{11}\,(\eta\tau) \, =\, y_{22}\,(\xi\tau)-x_{22}\,(\eta\tau)\,$.\\

\noindent
Note that $\, \Rep_d(k[x,y])\,$ is just the classical scheme of pairs of commuting matrices.
For $ d = 2 $, it is known to be a reduced and irreducible variety ({\it cf.} \cite{Ger}).
\end{example}

\begin{example}\la{Ex3D}
Let $ U(\mathfrak{sl}_{2}) $ be the universal enveloping algebra of the Lie algebra $ \mathfrak{sl}_{2}(k) $.
By \cite{G1}, Example~1.3.6, it has a finite DG resolution $ R \to U(\mathfrak{sl}_{2}) $, where
$ R = k \langle x,y,z; \xi, \theta, \lambda; t\rangle $ is the free graded algebra with generators $ x, y, z $ of degree $0$;
$ \xi, \theta, \lambda $ of degree $1$ and $t$ of degree $2$. The differential on $R $ is defined by
$$
d\xi = [y,z] + x\ ,\quad d\theta = [z,x] + y\ ,\quad d\lambda = [x,y] + z\ ,\quad dt = [x, \xi] + [y, \theta] + [z, \lambda]\ .
$$
For $ V = k^d $, Theorem~\ref{comp} then implies that
$$
R_V \cong k[x_{ij},\,y_{ij},\,z_{ij};\, \xi_{ij},\, \theta_{ij},\, \lambda_{ij};\,t_{ij}]\ ,
$$
where the generators $ x_{ij} \,$,$\, y_{ij} \,$,$ \,z_{ij} \,$ have degree zero,
$\,\xi_{ij} \,$, $\,\theta_{ij}\, $, $\,\lambda_{ij} \,$ have degree $ 1$,
and $\, t_{ij} $ have degree $ 2 $. Using the matrix notation $ X = \|x_{ij}\| \,$,
$\, Y = \| y_{ij}\|\,$, etc., we can write the differential on $\,R_V\,$ in the form
\begin{equation*}
d\Xi = [Y,Z] + X\ ,\quad d\Theta = [Z,X] + Y\ ,\quad d\Lambda = [X,Y] + Z\ ,\quad
dT = [X, \Xi] + [Y, \Theta] + [Z, \Lambda]\ .
\end{equation*}
For $ d = 1 $,  it is easy to compute
$$
\H_\bullet(U(\mathfrak{sl}_{2}),\,k) \cong k[t] \ ,
$$
where $\,k[t]\,$ is the free polynomial algebra generated by one variable of homological degree $2$.
A similar explicit description of derived representation schemes can be given for other three-dimensional Calabi-Yau algebras. For more examples, we refer to \cite[Sect.~6]{BFR}.
\end{example}

\section{Relative Cyclic Homology}
\la{S4}
In this section, we recall the definition of relative cyclic homology, due to Feigin and Tsygan \cite{FT, FT1}. We will introduce this notion in a slightly greater generality than in \cite{FT, FT1}, working over an arbitrary DG algebra. To the best of our knowledge,
this material has not appeared in standard textbooks on cyclic homology (like, e.g., \cite{L}),
so we will provide proofs, which are different from the original proofs in \cite{FT, FT1}.
In the next section, we will use these results to construct trace maps relating cyclic homology to representation homology.

\subsection{The cyclic functor}
\la{S4.1}
If $ A $ is a DG algebra, we write $\, A_\natural := A/[A,A] \,$, where $ [A,A] $ is the commutator
subspace of $ A $. The assignment $\, A \mapsto A_\natural \,$ is obviously a functor from $ \DGA_k $ to the category of complexes
$ \Com(k)$: thus, a morphism of DG algebras $ f: S \to A $ induces a morphism of complexes $\,f_\natural: S_\natural \to A_\natural \,$. Fixing $ S \in \DGA_k $, we now define the functor
\begin{equation}
\la{cycd}
\FT:\ \DGA_S \rar \Com(k)\ ,\quad (S \xrightarrow{f} A) \mapsto \cn(f_\natural)\ ,
\end{equation}
where `$ \cn $' refers to the mapping c\^{o}ne in $ \Com(k) $.
The next lemma shows that \eqref{cycd} behaves well with respect to homotopy.
\begin{lemma}
\la{cychom1}
If $ h: A \to B \otimes \Omega $ is a morphism in $ \DGA_S $ with $ h(0) = f $ and $ h(1) = g $, then
$$
\FT(f) \sim \FT(g) \quad \mbox{in}\quad \Com(k) \ .
$$
\end{lemma}
\bproof
We recall that $\, \Omega = k[t] \oplus k[t] dt \,$ denotes the algebraic de Rham complex of $ \A^1_k $, which is a commutative
DG algebra. The natural maps $\,B \into B \otimes \Omega \hookleftarrow \Omega \,$ induce then an isomorphism of complexes $\,B_\n \otimes \Omega \xrightarrow{\sim} (B \otimes \Omega)_\n \,$.
Identifying these complexes, we have the commutative diagram
\begin{equation*}
\begin{diagram}[small, tight]
S_{\natural} & \rTo  & B_{\natural} \otimes \Omega\\
             & \rdTo & \dTo_{\id \otimes \ev_a}\\
             &       & B_{\natural}
\end{diagram}
\end{equation*}
which induces a map of c\^{o}nes $\,\widetilde{\ev}_a:\,
\cn(S_\n \to B_\n \otimes \Omega)\,\to \,\cn(S_\n \to B_\n)\,$
for each $ a \in k $. (Here, $ \ev_a:\, \Omega \to k $ stands for
the evaluation map $\,f(t) + g(t) dt \mapsto f(a)\,$, where $ a \in k $.)

If we now apply $ \FT $ to $ h $, we get a morphism
$$
\FT(h):\,\cn(S_\n \to A_\n)\,\to \,\cn(S_\n \to B_\n \otimes \Omega)\ ,
$$
such that $\,\FT(h)(0) := \widetilde{\ev}_0 \circ \FT(h) = \FT(f) \,$ and $\,\FT(h)(1) := \widetilde{\ev}_1 \circ \FT(h) =
\FT(g) \,$. To show that $\,\FT(f) \sim \FT(g) \,$ it thus suffices to show that
$\,\widetilde{\ev}_1 \sim \widetilde{\ev}_0\,$. But this last equivalence is immediate from
the fact that, for all $ a \in k \,$, $\,\ev_a \sim \ev_0\,$ as morphisms $ \Omega \to k $;
the corresponding homotopy $ \Omega \to k[1] $ is given by $\,f(t)+ g(t)\,dt \mapsto \int_0^a\,g(\tau)\,d\tau\,$,
see \ref{3.4.4}.
\eproof

Now, to state the main theorem we recall that
$ \Com(k) $ has a natural model structure with quasi-isomorphisms being the weak equivalences and the epimorphisms being the fibrations ({\it cf.} \ref{2.2ex}). The corresponding homotopy
category $ \Ho(\Com(k)) $ is isomorphic to the (unbounded) derived category $ \D(k) $, see \ref{2.2.7}.
\bthm
\la{FTT}
The functor \eqref{cycd} has a total left derived functor $\,\L\FT:\,\Ho(\DGA_S) \to \D(k)\,$
given by
$$
\L\FT(S \bs A) = \cn(S_\n \to Q(S\bs A)_\n)\ ,
$$
where $\, S \to Q(S\bs A)\,$ is a cofibrant resolution of $ S \to A $ in $ \DGA_S $.
\ethm
\bproof
The proof is analogous to the proof of Theorem~\ref{S2P4}. By Brown's Lemma, it suffices to show that $\,\FT\,$ maps
acyclic cofibrations $\,i: A \sinto B \,$ between cofibrant objects in $\DGA_S $ to weak equivalences
(quasi-isomorphisms) in $\Com(k) $. Given $\,i: A \sinto B \,$, there is a map $\,p:\,B \to A\,$ in $\DGA_S$ such that
$\,p\,i=\id_A\,$ and $\,i\,p \sim \id_B\,$ ({\it cf.} \ref{2.4.1}). Hence $ \FT(p) \,\circ\, \FT(i) = \id_{\FT(A)} $, and
by Proposition~\ref{lhom}, there is a morphism $\,h: B \rar B \otimes \Omega $ in $\DGA_S $
with $\,h(0)=\id_B\,$ and $\,h(1)=i\,p\,$. Lemma~\ref{cychom1} now implies that $ \FT(ip) \sim \id_{\FT(B)} $.
Hence $ \FT(i p) $ induces the identity map on homology, and therefore $ \FT(i) $ is a quasi-isomorphism.
\eproof
Theorem~\ref{FTT} implies that the homology of $\, \L \FT(S \bs A)\,$ depends only on the morphism $ S \to A $.
Thus, we may give the following

\vspace{1ex}

\begin{definition}
\la{S3D1}
The (relative) {\it cyclic homology} of $\, S \bs A \in \DGA_S \,$ is defined by
\begin{equation}
\la{hcyc}
\HC_{n-1}(S \bs A) := \H_{n}[\L \FT(S \bs A)] = \H_{n} [\cn(S_\n \to Q(S\bs A)_\n)]\ .
\end{equation}
\end{definition}

\vspace{1ex}

\remark\ 1. We draw reader's attention to the fact that the homology groups \eqref{hcyc} are defined
(and in general nonzero) in all degrees, including the negative ones. However, if $ S $
is non-negatively graded and $ S \bs A \in \DGA_S^+ $, then  $ \, \HC_{<0}(S \bs A) = 0 \, $.
In fact, if $ S \bs A \in \DGA_S^+ $, we can take for a cofibrant resolution of $ S \bs A $ in $ \DGA_S $ its
cofibrant resolution $ Q(S\bs A) $ in $ \DGA_S^+ $, and $ \cn(S_\n \to Q(S\bs A)_\n) $ is
then a non-negatively graded complex.

2.\ If $ S \to A $ is a map of ordinary algebras and $ S \stackrel{i}{\into} QA \sonto A $ is a
cofibrant resolution of $ S \to A $ such that $ i $ is an almost free extension in $ \DGA^+_S $,
then the induced map $\, i_\natural:\,S_\n \into (QA)_\n \,$ is injective, and
\begin{equation}
\la{cyc1}
\FT(S \bs QA) = \cn(i_\n) \cong (QA)_\natural/S_\natural \cong QA/([QA, QA] + i(S)) \ .
\end{equation}
In this special form, the functor $ \FT $ was originally introduced by Feigin and Tsygan in \cite{FT} (see also \cite{FT1}); they proved that the homology groups \eqref{hcyc} are independent of the choice of resolution using spectral sequences. Theorem~\ref{FTT} is not explicitly stated in \cite{FT, FT1}, although it is implicit in several calculations.  We  emphasize that, in the case when $ S $ and $A$ are ordinary algebras, our definition of relative cyclic homology
\eqref{hcyc} agrees with the Feigin-Tsygan one.

\subsection{A fundamental exact sequence for cyclic homology}
\la{S4.2}
To study the properties of the cyclic functor \eqref{cycd} it is convenient to extend it to the category $ \Mor(\DGA_k) $ of {\it all} morphisms in $ \DGA_k $. Thus, we define
\begin{equation}
\la{cycdn}
\eFT:\ \Mor(\DGA_k) \rar \Com(k)\ ,\quad (S \xrightarrow{f} A)
\mapsto \cn(f_\natural)\ .
\end{equation}
The category $ \Mor(\DGA_k) $ has a natural model structure induced from $ \DGA_k $ (see \ref{2.1.2m}), and the next
result shows that \eqref{cycdn} is well-behaved with respect to this structure.
\bthm
\la{TPT}
The functor \eqref{cycdn} has a total left derived functor
$ \L\eFT:\,\Ho(\Mor(\DGA_k)) \to \D(k) $ which is given by
\begin{equation}
\la{isff}
\L\eFT(S \bs A) = \cn[(QS)_\n \to Q(QS\bs A)_\n]\ ,
\end{equation}
where $ QS \sonto S $ is a cofibrant resolution of $S$ in $ \DGA_k $
and  $ Q(QS\bs A) $ is a cofibrant resolution of $ QS \to S \to A $
in $ \DGA_{QS} $.
\ethm
\bproof
We verify the assumptions of Brown's Lemma for $ \eFT $. Let $ A \to B $ be a
cofibrant object in $ \Mor(\DGA_k)$, which means that $ A \to B $ is a cofibration
in $ \DGA_k $, with $A$ being a cofibrant object.
Let
\begin{equation*}
\begin{diagram}[small, tight]
A                 & \rTo^{\alpha}    &  C         \\
\dTo          &                      & \dTo   \\
B                 & \rTo^{\beta}     &  D
\end{diagram}
\end{equation*}
be an acyclic cofibration in $ \Mor(\DGA_k) $, which means that
$ \alpha $ is an acyclic cofibration, $\, \beta $ is a weak equivalence, and
the induced map $\, B \coprod_A C \to D \,$ is a cofibration in $ \DGA_k $.
Since $ A \into B $ is a cofibration, the pushout $ B \to B \coprod_A C $ of $ \alpha $
along $ A \to B $ is an acyclic cofibration. Since
$\,\beta:\, B \to B \coprod_A C \to D $ is acyclic, the cofibration
$\, B \coprod_A C \to D \,$ is also acyclic. It follows that
$ \beta $ itself is an acyclic cofibration. The standard argument
(see the proof of Theorem~\ref{FTT} above) implies now that $ \alpha_\n $ and
$ \beta_\n $ are quasi-isomorphisms. Hence, the pair
$(\alpha_\n,\,\beta_\n) $ induces a quasi-isomorphism
$\,\eFT(A \bs B) \xrightarrow{\sim} \eFT(C \bs D) \,$.
By Brown's Lemma, $\, \L\eFT $ exists and is given by formula
\eqref{isff}, which is independent of the choice of resolutions.
\eproof
Next, we prove
\blemma
\la{LtP}
Given DG algebra maps $ R \xrightarrow{g} S \xrightarrow{f} A $, there is a distinguished triangle
in $ \D(k) $ of the form
\begin{equation}
\la{isct}
\L\eFT(R \bs S) \to \L\eFT(R \bs A) \to \L\eFT(S \bs A) \to
\L\eFT(R \bs S)[1]\ .
\end{equation}
\elemma
\bproof
Fix a cofibrant resolution $\, p_R: QR \sonto R \,$ in $ \DGA_k $. Then choose cofibrant resolutions $\, p_S: QS \sonto S \,$ of $ g \circ p_R $ in $ \DGA_{QR} $ and
$p_A: QA \sonto A\,$ of $ f\circ p_S $ in $ \DGA_{QS}$. Note that both $QS$ and $QA$ are cofibrant
in $ \DGA_k $, and there are liftings $ \tilde{g}: QR \into QS $ and  $ \tilde{f}: QS \into QA $ which are cofibrations.
Now, by Octahedron Axiom, associated to $ (QR)_\n \xrightarrow{\tilde{g}} (QS)_\n \xrightarrow{\tilde{f}} (QA)_\n $
is a distinguished triangle in $ \D(k)$ of the form
$$
\cn(\tilde{g}_\n) \to \cn(\tilde{f}_\n \circ \tilde{g}_\n) \to \cn(\tilde{f}_\n) \to \cn(\tilde{g}_\n)[1] \ .
$$
By \eqref{isff}, $\,\L\eFT(g) = \cn(\tilde{g}_\n)$, and similarly for $f$ and $ f \, g $. The above triangle is thus
isomorphic to \eqref{isct}.
\eproof
\blemma
\la{LtP1}
For any $ S \to A $ in $ \DGA_k $, there is a natural isomorphism
$\,\L\eFT(S\bs A) \cong \L\FT(S \bs A)\,$ in $ \D(k)$.
\elemma
\bproof
The inclusion functor $ i:\,\DGA_S \to \Mor(\DGA_k) $ maps
weak equivalences to weak equivalences and hence descends to the
homotopy categories $\,\bar{i}:\,\Ho(\DGA_S) \to \Ho(\Mor(\DGA_k))\,$. Since the derived functors
$ \L\FT $ and $ \L\eFT $ both exist, the universal property of derived functors yields a natural
transformation
\begin{equation}
\la{edfnas}
\L\eFT \circ \bar{i} \to \L\FT\ .
\end{equation}
We claim that \eqref{edfnas} is an isomorphism of functors $\,\Ho(\DGA_S) \to \D(k)\,$. To prove this,
we first verify that
\begin{equation}
\la{edfs}
\L\eFT(S\bs 0) \cong \L\FT(S\bs 0)\ .
\end{equation}
Let $ p: Q \sonto S $ be a cofibrant resolution of $S$ in $\DGA_k$. Then $\,\L\eFT(S\bs 0) \cong Q\langle x \rangle_\n/Q_\n\,$
and $\,\L\FT(S\bs 0) \cong S\langle x \rangle_\n/S_\n\,$, where $ dx =1 $, and
$ p $ induces an isomorphism $ \L\eFT(S\bs 0) \stackrel{\sim}{\to} \L\FT(S\bs 0)$. More generally,
if $ p: R \to S $ is any quasi-isomorphism in $ \DGA_k $, then $ p $ induces
$ \L\FT(R\bs 0) \stackrel{\sim}{\to} \L\FT(S\bs 0)$.
Hence, for $ S \to A \to 0 $, the same argument as in the proof of Lemma~\ref{LtP} shows that there is an exact triangle
\begin{equation}
\la{extr}
\L\FT(S \bs A) \to \L\FT(S \bs 0) \to \L\FT(A \bs 0)\to \L\FT(S \bs A)[1]\ .
\end{equation}
By \eqref{edfs}, we now conclude
that $\,\L\eFT(S\bs A) \cong \L\FT(S \bs A)\,$ for all $ S \to A $.
\eproof
As a consequence of Lemma~\ref{LtP} and Lemma~\ref{LtP1}, we
get the following result proved in \cite{FT} by a different method
(see \cite{FT}, Theorem~2).
\bthm
\la{S4C1}
Given DG algebra maps $ R \xrightarrow{} S \xrightarrow{} A $,
there is a long exact sequence in cyclic homology
\begin{equation}
\la{ex2}
\ldots \to \HC_n(R \bs S) \to \HC_n(R \bs A) \to \HC_n(S \bs A) \to \HC_{n-1}(R \bs S) \to \ldots
\end{equation}
\ethm
\bproof
By Lemma~\ref{LtP} and Lemma~\ref{LtP1}, there is a distinguished
triangle in $ \D(k)$ of the form
$$
\L\FT(R\bs S) \to \L\FT(R\bs A) \to \L\FT(S\bs A) \to \L\FT(R\bs S)[1]
$$
The exact sequence \eqref{ex2} arises from applying the homology functor to this triangle.
\eproof
\subsection{Connes' cyclic complex}
\la{S4.3}
If $A$ is an ordinary algebra over a field of characteristic zero,
its {\it cyclic homology} $ \HC_\bullet(A) $ is usually defined  as the homology of the cyclic complex $ \CC(A) $ ({\it cf.} \cite{L}, Sect.~2.1.4):
\begin{equation}
\la{ccc}
\CC_n(A) := A^{\otimes (n+1)}/\im(\id-t_n)\ , \quad b_n:\,\CC_n(A) \to \CC_{n-1} (A)\ ,
\end{equation}
where $ b_n $ is induced by the standard Hochschild differential and
$ t_n $ is the cyclic operator defining an action of $ \Z/(n+1)\Z $ on $ A^{\otimes (n+1)} $:
\begin{equation}
\la{cop}
t_n:\, A^{\otimes (n+1)} \to A^{\otimes (n+1)}\ , \quad (a_0,\, a_1, \,\ldots \,, \,a_n) \mapsto
(-1)^{n}(a_n,\, a_0, \,\ldots \,, \,a_{n-1})\ .
\end{equation}
%
%
The complex $ \CC(A) $ contains the canonical subcomplex $ \CC(k) $; the homology of the corresponding quotient complex $\, \rHC_\bullet(A) := \H_\bullet[\CC(A)/\CC(k)]\,$
is called the {\it reduced cyclic homology} of $A$. Both $ \HC(A) $
and $ \rHC(A) $ are special cases of relative cyclic homology in the sense of Definition~\eqref{hcyc}. Precisely, we have the following
result (due to Feigin and Tsygan \cite{FT}).
\bprop
\la{S4P1}
For any $k$-algebra $A$, there are canonical isomorphisms

\vspace{1ex}

$(a)$\ $\,\HC_n(A) \cong \HC_n(A \bs 0) \,$ for all $\,n\ge 0\,$,

\vspace{1ex}

$(b)$\ $\,\rHC_n(A) \cong \HC_{n-1}(k \bs A) \,$ for all $\,n\ge 1\,$.

\eprop
\bproof
$(a)$\ For any algebra $A$,
the canonical morphism $\, A \into A\langle x \rangle \,$ provides a cofibrant resolution of $ A \to 0 $ in $ \DGA_A $. In this case, we can identify
$$
\L\FT(A\bs 0) \cong \cn(A_\n \to A\langle x \rangle_\n) \cong
A\langle x \rangle/(A + [A\langle x \rangle,\,A\langle x \rangle])
\cong \CC(A)[1]\ ,
$$
where the last isomorphism (in degree $ n>0$) is given by
$\,
a_1 \,x\, a_2\, x\, \ldots\, a_n\, x \leftrightarrow a_1 \otimes a_2 \otimes \ldots \otimes a_n\,$.
On the level of homology, this induces isomorphisms
$ \HC_{n-1}(A \bs 0) \cong \H_n(\CC(A)[1]) = \HC_{n-1}(A) $.

$(b)$ With above identification, the triangle \eqref{extr} associated to the canonical maps $\,k \to A \to 0 \,$ yields
$$
\L\FT(k\bs A) \cong
\cn(\L\FT(k\bs 0) \to \L\FT(A \bs 0))[-1] \cong
\cn[\CC(k) \to \CC(A)]\ .
$$
Whence $ \HC_{n-1}(k \bs A) \cong \rHC_n(A) $ for all $ n \ge 1 $.
\eproof

\vspace{1ex}

\remark\ The isomorphism of Proposition~\ref{S4P1}$(a)$ justifies the shift of indexing in our definition \eqref{hcyc} of
relative cyclic homology. In \cite{FT, FT1}, cyclic homology
is referred to as an {\it additive} $K$-theory, and a different notation is used. The relation between
the Feigin-Tsygan notation and our notation is $\,  K_{n}^+(A,S) = \HC_{n-1}(S \bs A)$ for all $\,n \ge 1\,$.

\vspace{1ex}

As consequence of Theorem~\ref{S4C1} and Proposition~\ref{S4P1}$(a)$, we get the fundamental exact sequence associated
to an algebra map $\, S \to A \,$:
\begin{equation*}
\la{ex1}
\ldots \to \HC_n(S\bs A) \to \HC_n(S) \to \HC_n(A) \to \HC_{n-1}(S\bs A) \to \ldots \to \HC_0(S) \to \HC_0(A) \to 0\ .
\end{equation*}
In particular, if we take $ S = k $ and use the isomorphism of Proposition~\ref{S4P1}$(b)$, then \eqref{ex1} becomes
\begin{equation}
\la{rseq}
\ldots \to \HC_n(k) \to \HC_n(A) \to \rHC_n(A) \to \HC_{n-1}(k) \to \ldots \to \HC_0(A) \to \rHC_0(A) \to 0\ .
\end{equation}
\subsection{Relative Hochschild homology}
\la{S3.4}
We now briefly explain how to describe Hochschild homology in the Feigin-Tsygan approach. For this,
we consider the following simple modification of the functor \eqref{cycd}:
\begin{equation}
\la{hoch}
\Hoch:\ \DGA_S \rar \Com(k)\ ,\quad (S \to A) \mapsto
\cn(S_\n \to A_{\n,S})\ ,
\end{equation}
where $\,A_{\n,S} := A/[A,S]\,$. The same arguments as in Lemma~\ref{cychom1} and Theorem~\ref{FTT}
show that $ \Hoch $ has a total left derived functor $\,\L\Hoch:\,\Ho(\DGA_S) \to \D(k)\,$, which is
given by
\begin{equation}
\la{hoch1}
\L\Hoch(S\bs A) =  \cn(S_\n \to (QA)_{\n,S})\ .
\end{equation}
Thus, we can define
\begin{equation}
\la{hoch2}
\HH_{n-1}(S\bs A) := \H_n[\L\Hoch(S\bs A)] =  \H_n[\cn(S_\n \to (QA)_{\n,S})]\ .
\end{equation}
Applying \eqref{hoch} to the canonical resolution $ S \into S\langle x \rangle $ of
the zero morphism $ S \to 0 $ in $ \DGA_S $, we get
$$
\L\Hoch(S\bs 0) \cong S\langle x \rangle/(S + [S, S\langle x \rangle]) \cong C_\bullet(S,S)[1]\ ,
$$
where $ C_\bullet(S,S) $ is the standard Hochschild chain complex of $S$. Thus,
$$
\HH_n(S\bs 0) \cong \HH_n(S) \ ,\quad \forall\, n\ge 0 \ .
$$

In general, given an arbitrary morphism $\, S \to A \,$ in $ \Alg_k$, we can view $A$ as a bimodule over $S$ and
the map $\, S \to A \,$ itself can then be identified with the 2-term complex $\,0 \to S \to A \to 0 \,$ in the derived
category of $ S$-bimodules. With this identification, we have isomorphisms
\begin{equation}
\la{hoch3}
\HH_{n-1}(S\bs A) \cong \mathbb{HH}_n(S,\,S \to A)\ ,\quad n \ge 1\ ,
\end{equation}
where $\, \mathbb{HH}_n(S,\,\mbox{--}\,) :=  \H_n(S \otimes^{\L}_{S^{e}} \mbox{--}\,)\,$ denotes
the $n$-th Hochschild (hyper)homology of $S$-bimodules. We leave the proof of this result
as an exercise to the reader.

\section{Trace Maps}
\la{TraceM}
In this section, we construct canonical trace maps
$\,\Tr_V(S \bs A)_n:\,\HC_{n-1}(S \bs A) \to \H_n(S \bs A, V)\,$ relating the cyclic homology of an $S$-algebra
$ A \in \Alg_S $ to its representation homology. In the case when $ S = k $ and $V$ is concentrated in degree
$ 0 $, these maps can be viewed as {\it derived characters} of finite-dimensional representations of $A$.
We study functorial properties of  $ \Tr_V(S \bs A)_\bullet $ and construct an explicit formula
for these traces in terms of Connes' cyclic complex. We also construct `noncommutative' trace maps
defined on Hochschild homology of the morphism $ S \to A $.

\subsection{Main construction}\la{S4.1e}
Let $V$ be a complex of $k$-vector spaces of total dimension $d$. The natural map
$\, k \into \END(V) \onto \END(V)_\n \,$ is an isomorphism of complexes, which we can use to
identify $\, \END(V)_\n = k \,$. This defines a canonical (super) trace map
$\,\Tr_V:\,\END\,V \to k \,$ on the DG algebra $ \END\,V $. Explicitly, $ \Tr_V $
is given by
$$
\Tr_V(f) = \sum_{i=1}^d (-1)^{|v_i|} f_{ii}\ ,
$$
where $ \{v_i\} $ is a homogeneous basis in $V$ and $ \| f_{ij} \| $
is the matrix representing $\,f \in \END\,V \,$ in this basis.

Now, fix $\, S \in \DGA_k \,$ and a DG algebra map
$\,\varrho:\,S \to \END\,V \,$ making $V$ a DG module over $S$.
For an $S$-algebra $ A \in \DGA_S $, consider the (relative) DG representation
scheme $ \Rep_V(S \bs A) $, and let $\,\pi_V:\,A \to \END\,V \otimes (S \bs A)_V \,$
denote the universal representation of $A$ corresponding to the identity map
in the adjunction of Proposition~\ref{S2P1}$(b)$. Consider the morphism of complexes
\begin{equation}
\la{E1S4}
A \xrightarrow{\pi_V} \END\,V \otimes (S \bs A)_V \xrightarrow{\Tr_V \otimes \id } (S \bs A)_V\ .
\end{equation}
Since $ \pi_V $ is a map of $S$-algebras, and the $S$-algebra structure on $ \END\,V \otimes (S \bs A)_V $
is of the form $ \varrho \otimes \id $, \eqref{E1S4} induces a map
$\, \Tr_V \circ \pi_V: \, A_\n \to (S \bs A)_V $, which fits in the commutative diagram
\begin{equation}
\la{D1S4}
\begin{diagram}[tight, small]
  S_\n                  &  \rTo^{}  & A_\n  \\
\dTo^{\Tr_V\circ \varrho} &           & \dTo_{\Tr_V \circ \pi_V} \\
  k                     &  \rTo^{}  & (S \bs A)_V
\end{diagram}
\end{equation}
This, in turn, induces a morphism of complexes
\begin{equation}
\la{E2S4}
\cn(S_\n \to A_\n) \to \overline{(S \bs A)}_V\ ,
\end{equation}
where we write $\, \bar{R} = R/k\cdot 1_R \,$ for a unital DG algebra $R$.
The family of morphisms \eqref{E2S4} defines a natural transformation of functors
$\,\Tr_V:\, \FT \to \overline{(\mbox{---})}_V\,$, which descends to a natural
transformation of the derived functors (see \ref{L1S4}):
\begin{equation}
\la{E3S4}
\Tr_V:\, \L \FT \to  \overline{\L(\mbox{---})}_V \ .
\end{equation}

Next, observe that, for any (non-acyclic) unital DG algebra $ R $,
\begin{equation}
\la{hvan}
\H_{n}(\bar{R}) \cong
\left\{
\begin{array}
{lcl}
\overline{\H_0(R)} & , & n = 0 \\*[1ex]
\H_{n}(R)           & , & n \not= 0
\end{array}
\right.
\end{equation}
This is immediate from the long homology sequence arising from
$ 0 \to k \to  R \to \bar{R} \to 0 $.
Hence, if $ A \in \Alg_S $ is an ordinary algebra, applying the natural transformation
\ref{E3S4} to $ S \bs A $ and using \eqref{hvan}, we can define
\begin{equation}
\la{char}
\Tr_V(S \bs A)_n:\ \HC_{n-1}(S\bs A) \to \H_{n}(S \bs A, V)\ ,\ n \ge 1\ .
\end{equation}
Assembled together, these trace maps define a homomorphism of
graded commutative algebras
\begin{equation}
\la{lchar}
\bL \Tr_V(S \bs A)_\bullet :\ \bL(\HC(S\bs A)[1]) \to \H_{\bullet}(S \bs A, V)
\ ,
\end{equation}
where $ \bL $ denotes the graded symmetric algebra of a graded $k$-vector space $W$.

\subsection{The absolute case}
\la{abscase}
We examine the trace maps \eqref{char} and \eqref{lchar} in the special case
when $S = k$ and $V$ is a single vector space concentrated in degree $0$. In this case, by Proposition~\ref{S4P1}, the maps \eqref{char} relate
the reduced cyclic homology of $A$ to the (absolute) representation
homology:
\begin{equation}
\la{char1}
\Tr_V(A)_n:\ \rHC_{n}(A) \to \H_{n}(A, V)\ ,\ n \ge 1\ .
\end{equation}
Now, for each $n$, there is a natural map $ \HC_n(A) \to \rHC_n(A) $ induced by the projection of complexes
$ \CC(A) \onto \overline{\CC}(A) $, {\it cf.} \eqref{rseq}. Combining this map with \eqref{char1}, we get
\begin{equation}
\la{char2}
\Tr_V(A)_n:\ \HC_{n}(A) \to \H_{n}(A, V)\ , \ n\ge 0\ .
\end{equation}
Notice that \eqref{char2} is defined for all $\, n $, including $ n = 0 $. In the latter case,
$\, \H_0(A, V) \cong A_V \,$ (by Theorem~\ref{S2T4}), and $\, \Tr_V(A)_0:\, A_\n \to A_V\,$ is
the usual trace induced by $\,A \xrightarrow{\pi_V} A_V \otimes \End_k V \xrightarrow{\id \otimes \Tr} A_V\,$.

The linear maps \eqref{char2} define an algebra homomorphism
\begin{equation}
\la{lchar1}
\bL \Tr_V(A)_\bullet :\ \bL[\HC(A)] \to \H_{\bullet}(A, V)\ .
\end{equation}
Since, for $n\ge 1$, \eqref{char2} factor through reduced cyclic homology, \eqref{lchar1} induces
\begin{equation}
\la{trhomr}
\overline{\Tr}_V(A)_\bullet :\ \Sym(\HC_0(A)) \otimes \bL(\rHC_{\ge 1}(A)) \to \H_\bullet(A,\,V)\ ,
\end{equation}
where $ \rHC_{\ge 1}(A) := \bigoplus_{n \ge 1} \rHC_{n}(A) $.
\bprop
The image of the maps \eqref{lchar1} and \eqref{trhomr} is contained in $ \H_\bullet(A, V)^{\GL(V)}$.
\eprop
\bproof
Recall that, when $V$ is concentrated in degree $0$, we have the invariant subfunctor of $ (\mbox{---})_V $:
$$
(\mbox{---})_V^{\GL} :\ \DGA_k \to \cDGA_k\ ,\quad A \mapsto A_V^{\GL(V)}\ .
$$
It is clear from \eqref{E1S4} that the natural transformation $ \Tr_V $ maps $ \FT $ to this subfunctor:
\begin{equation}
\la{E8S4}
\Tr_V:\ \FT \to \overline{(\mbox{---})}_V^{\GL}
\end{equation}
Now, by Theorem~\ref{S2P4}, $\, (\mbox{---})_V^{\GL} $ has a total left derived functor, which is isomorphic to
the invariant subfunctor $ \L(\mbox{---})_V^{\GL}$ of $ \L(\mbox{---})_V $. By \ref{L1S4},
\eqref{E8S4} descends to the homotopy categories, i.e. induces a natural transformation
$\, \Tr_V:\, \L\FT \to \overline{\L(\mbox{---})}_V^{\GL}\,$. On the homology level, this gives
\begin{equation}
\la{E9S4}
\bL \Tr_V(A)_\bullet :\ \bL(\HC(A)) \to \H_{\bullet}(A, V)^{\GL(V)}\ ,
\end{equation}
which is the statement of proposition.
\eproof
\subsection{Trace formulas}
\la{S4.4}
The trace maps \eqref{char2} are linear functionals defined on the cyclic homology of
the algebra $A$. Given a cofibrant resolution $ R \sonto A $, these functionals arise from certain
cyclic cocycles with values in $ R_V $. Our next goal is to construct an explicit morphism
of complexes $\, T:\,\CC(A) \to R_V \,$ that induces the maps \eqref{char2}. We will use some basic
tools of differential homological algebra, which we will briefly introduce in the next section.
The standard reference for this material is \cite{HMS}.

\subsubsection{The bar construction and twisting cochains}
Let $ R $ be a non-negatively graded (nonunital) DG algebra with
multiplication map $\,m: R \otimes R \to R \,$, and let $ C $ be a (noncounital) DG coalgebra
with comultiplication map $ \Delta: C \to C \otimes C $. We assume that the
differentials on both $ R $ and $C$ have degree $-1$.
The homogeneous linear maps $\, C \to R \,$ then form a cochain DG algebra
\begin{equation}
\la{halg}
\HOM^\bullet(C, R) := \bigoplus_{n \in \Z} \HOM^n(C, R) \ ,
\end{equation}
with $n$-th graded component $ \HOM^n(C, R) $ consisting of maps of degree $\, -n $.
The differential on $ \HOM^\bullet(C, R) $ has degree $+1$ and is given by
$\,d f :=  d_R \, f - (-1)^{n} f \, d_C\,$, where $ f \in \HOM^n(C, R) $.
The product is defined by convolution $\, f \cdot g := m \, (f \otimes g) \, \Delta\,$.

Now, a {\it twisting cochain} from $ C $ to $ R $ is a linear map $ \theta:\,C \to R $ of degree $ -1 $
satisfying the equation
\begin{equation}
\la{twc}
d_R \, \theta + \theta \, d_C + m \, (\theta \otimes \theta) \, \Delta = 0\ .
\end{equation}
Equivalently, $ \theta $ is an element of $ \HOM^1(C, R) $ satisfying $\, d\theta + \theta^2 = 0 \,$.
If $ \theta: C \to R $ is a twisting cochain on $C$ and
$ f: C' \to C $ is a map of DG coalgebras, then $ \theta' = \theta \,f $ is
a twisting cochain on $ C' $.  In this case, we say that $ \theta' $ is induced from $ \theta $ by the
coalgebra map $f$. The bar construction of $R$ is a DG coalgebra, which is a universal model
for twisting cochains with values in $R$. It is defined as follows.

Let $\, \Ta(R[1]) := \bigoplus_{n \ge 1} R[1]^{\otimes n} \,$ be the reduced tensor coalgebra generated
by the graded vector space $ R[1] $.
%
Write $\,p:  \Ta(R[1]) \onto R[1] \,$ for the canonical projection and $\,s: R \to R[1] \,$ for the natural map
of degree $ +1 $, giving isomorphisms $ R_i \cong R[1]_{i+1} $ for all $i$. Note that
$\,d_{R[1]} s + s\,d = 0 \,$. Now, there are two differentials on $ \Ta(R[1]) \,$. The first differential
$ \partial $ is the (unique) coderivation of degree $-1$ lifting the differential on $ R[1] $,
i.~e. $\,p\,\partial = d_{R[1]}\,p\,$. The second differential $ \delta $ lifts the multiplication
map on $R$ so that $\,p\,\delta = \tilde{m}\,(p \otimes p)\,\Delta\,$, where
$ \tilde{m}: R[1] \otimes R[1] \to R[1] $ is a map of degree $-1$
defined by $\,\tilde{m} \,(s \otimes s) = s\,m \,$. Note that $ \partial $
depends only on the underlying complex of $ R $, while $ \delta $ depends on
the multiplicative structure on $R$. It is easy to check that
\begin{equation*}
\la{cdr}
\partial^2 = 0\ ,\quad \partial\,\delta + \delta\,\partial = 0\ ,\quad \delta^2 = 0 \ .
\end{equation*}
The {\it bar construction} $ \Ba(R) $ of $R$ is defined to be the coalgebra $ \Ta(R[1]) $
equipped with the total differential $ d_B := \partial + \delta $. This DG coalgebra
comes together with the canonical map $ \hat{\theta} := s^{-1} p:\, \Ba(R) \to R $, which
is the {\it universal} twisting cochain with values in $R$ (see \cite{HMS}, Prop.~II.3.5).

Now, given a DG coalgebra $C$ and a DG algebra $R$, we write $\,\natural:\,C^{\natural} \into C \,$ for
the natural inclusion of the cocommutator subcomplex of $ C $ and $\, \tau:\, R \onto R_\natural \,$ for
the natural projection onto the quotient by the commutator complex of $R$.
The map $\,\natural\,$ is the universal cotrace for $ C $ in the same way as $\, \tau \,$ is the
universal trace for $ R $. We will need two simple lemmas.
\blemma
\la{ltr}
$(a)$ The linear map
$\,
\tau^\natural:\ \HOM^\bullet(C, R) \to \HOM^\bullet(C^\natural, R_\natural)\, ,\, f \mapsto \tau f \natural \,$,
is a closed trace.

$(b)$ If $ \theta:\,C \to R $ is a twisting cochain, then
$\,
s\tau^\natural(\theta):\,C^\natural \to R_\natural[1]\,
$
is a morphism of complexes.
\elemma
\bproof
$(a)\,$ It is clear that $ \tau^\natural $ commutes with differentials, since so do both $ \tau $ and $ \natural $.
To see that $ \tau^\natural $ vanishes on commutators, we check
$$
\tau^\natural(f \cdot g) = \tau\,m\,(f \otimes g)\,\Delta\,\natural =
\tau\,m\,\sigma(f \otimes g)\,\sigma\,\Delta\,\natural = (-1)^{|f|\,|g|}
\tau\,m\,(g \otimes f)\,\Delta\,\natural = (-1)^{|f|\,|g|} \tau^\natural(g \cdot f)
$$
for any $ f,g \in \HOM^\bullet(C, R) $. Thus $ \tau^\natural $ is a closed trace on the DG algebra
$ \HOM^\bullet(C, R) $.

$(b)\,$ Recall that $s: R \to R[1] $ is the canonical map of degree $+1$ satisfying $ s\,d_R + d_{R[1]} s = 0 $.
Hence $ s\tau^\natural(\theta) $ is a degree zero element in $ \HOM^\bullet(C^\natural, R_\natural[1]) $
satisfying the equation
\begin{eqnarray*}
d[s\tau^\natural(\theta)] &=& d_{R[1]} s\,\tau^\natural(\theta) - s\,\tau^\natural(\theta)\,d_C
                          = - s\,(d_{R}\,\tau^\natural(\theta) + \tau^\natural(\theta)\,d_C)\\
                          &=& - s\,d[\tau^\natural(\theta)] = -s\,\tau^\natural(d \theta) =
                          s \tau^\natural(\theta^2) = \frac{1}{2}\,s \tau^\natural([\theta,\theta]) = 0   \ ,
\end{eqnarray*}
where the last equality holds by part $(a)$.
\eproof
\blemma
\la{lqsm}
Let $ R $ be an almost free DG algebra in $ \DGA_k^+$.
Then the trace map defined by the universal twisting cochain in $R$ induces the quasi-isomorphism
\begin{equation}
\la{qtsm}
s \tau^\natural(\hat{\theta}):\ \Ba(\bar{R})^\natural \,\stackrel{\sim}{\to}\, \bar{R}_\natural[1]\ ,
\end{equation}
where $\,\Ba(\bar{R})^\natural := \Ba(R)^\natural/\,\Ba(k)^\natural\,$ and $\,\bar{R}_\natural := R_\n/k $.
\elemma
\bproof
It is clear that the induced map \eqref{qtsm} is well defined, since $ \hat{\theta} $ carries the subcomplex
$ \Ba(k) \subset \Ba(R) $ to the image of $ k $ in $R$.
To prove that \eqref{qtsm} is a quasi-isomorphism consider the first quadrant double complex $(\Ba(\bar{R})^\natural,\,\partial,\,\delta) $. Its column in degree $ p $ is the invariant subcomplex of $ \bar{R}^{\otimes (p+1)} $ under cyclic permutations (see \cite{Q4}, Prop.~3.3). The horizontal differential $ \delta $ is the natural extension to graded algebras of
the usual boundary operator $b'$, and the total complex $ (\Ba(\bar{R})^\natural,\,d_B) $ can be
identified (up to shift) with the reduced cyclic complex of $R$ (see \cite{Q4}, Sect.~3.3, or Theorem~\ref{TQ} below).
Now, observe that \eqref{qtsm} is precisely an edge homomorphism for the spectral sequence
\begin{equation}
\la{spsq}
E^2_{p,q} = \H_p(\partial,\,\H_q(\Ba(\bar{R})^\natural, \delta))\ \Rightarrow\ \H_n(\Ba(\bar{R})^\natural,\,d_B)\ ,
\end{equation}
associated to $(\Ba(\bar{R})^\natural,\,\partial,\,\delta) $.
(In this spectral sequence, one is taking first homology of the rows and then homology of the columns.)

When $k$ has characteristic zero, the reduced cyclic homology of a free (tensor) algebra $ T_k V $ vanishes
in degrees $ q > 0 $ (see \cite{L}, Theorem~3.1.6). Since $ R $ is such an algebra (with differential forgotten),
the row homology of $ \Ba(\bar{R})^\natural $ vanishes in higher degrees, i.e. $ \H_q(\Ba(\bar{R})^\natural, \delta) = 0 $
for all $ q > 0 $. This means that the spectral sequence \eqref{spsq} collapses, and its edge homomorphism \eqref{qtsm} is
indeed a quasi-isomorphism.
\eproof
\subsubsection{The norm maps}
\la{ccnm}
We now consider the special case when $ C = \Ba(A) $ is the bar construction of an ordinary algebra.
The $\partial$-differential on $ \Ba(A) $ is zero in this case, so the $ \delta$-differential
coincides with the bar differential, which is traditionally denoted $ b' \,$:
\begin{equation}
\la{bprime}
b' = \sum_{i=1}^{n-1} (-1)^{i-1} (\id_A^{\otimes (i-1)} \otimes m_A \otimes \id_A^{\otimes (n-1-i)})\ .
\end{equation}
Next, we define the following norm maps
\begin{equation}
\la{norm}
N_n:\,\CC_{n-1}(A) \to \Ba(A)_n\ ,\quad N_n := \sum_{k=0}^{n-1}\, t^k\ ,\quad n\ge 1\ ,
\end{equation}
where $\, t:\,(a_1, a_2, \ldots, a_n) \mapsto (-1)^{n-1}(a_n, a_1, \ldots a_{n-1})\, $ is a cyclic operator.
The maps \eqref{norm} are compatible with differentials and thus give a morphism of complexes
\begin{equation}
\la{norm1}
N:\, \CC(A)[1] \to \Ba(A)\ .
\end{equation}
The relevance of \eqref{norm1} becomes clear from the following
theorem, which is one of the main observations of \cite{Q3} (see {\it op. cit.}, Lemma~1.2).
\begin{theorem}[Quillen]
\la{TQ} For any algebra $A$, the map \eqref{norm1} is injective, and its image
is precisely $ \Ba(A)^\natural $. Thus, we have the isomorphism of complexes $\,N:\,\CC(A)[1] \cong
\Ba(A)^{\natural}\,$.
\end{theorem}
We will use Theorem~\ref{TQ} to identify $ \Ba(A)^\natural = \CC(A)[1]$. With this identification, the
trace map of Lemma~\ref{ltr} becomes
$$
\tau^\natural: \ \HOM^\bullet(\Ba(A),\,R) \to \HOM^\bullet(\CC(A)[1],\,R_\natural)\ ,\quad f \mapsto \tau\,f\,N\ .
$$
Explicitly, this map takes a cochain $\, f: \Ba(A) \to R \,$ with components
$\, f_n: A^{\otimes n} \to R \,$, $ n \geq 1\,$, to the cyclic cochain
$ \tau^\natural(f):\, \CC(A)[1] \to R_\natural $ with components $\,\tau^\natural(f)_n:\,\CC_{n-1}(A) \to R_\natural \,$
given by
\begin{eqnarray}
\la{cch}
\tau^\natural(f)_n(a_1, a_2, \ldots, a_n) &=& \sum^{n-1}_{k=0} (-1)^{k(n-1)}\tau[f_n(a_{n-k+1}, \ldots, a_n, a_1, \ldots, a_{n-k})] \\
                 &=& \sum_{k \in \Z_n} (-1)^{k(n-1)} f_n(a_{1+k}, a_{2+k}, \ldots, a_{n+k})\quad \mbox{mod}\,[R,R]\ ,  \nonumber
\end{eqnarray}
where the last sum is taken over the elements of the cyclic group $ \Z_n $.
In particular, by Lemma~\ref{ltr}$(b)$, a twisting cochain $ \theta: \Ba(A) \to R $ with components
$ f_n: A^{\otimes n} \to R_{n-1} $, $\,n\ge 1 \,$, produces a map of complexes
\begin{equation}
\la{cch1}
s\tau^\natural(\theta):\, \CC(A) \to R_\natural\ ,
\end{equation}
which in degree $ (n-1) $ is given by formula \eqref{cch}.

\subsubsection{Trace formulas}
\la{aff}
Let $ \pi: R \sonto A $ be an almost free resolution of an algebra $A$ in $ \DGA_k^+$.
By functoriality of the bar construction, the map $ \pi $ extends to a surjective quasi-isomorphism
of DG coalgebras $ \Ba(R) \sonto \Ba(A)$, which we still denote by $ \pi $. This quasi-isomorphism has a
section $ f: \, \Ba(A) \into \Ba(R) $ in the category of DG coalgebras, which is uniquely determined
by the twisting cochain $ \theta_\pi := \hat{\theta}\,f:\, \Ba(A) \to R $.
The components $\,f_n:\,A^{\otimes n} \to R_{n-1}\,$, $\,n \ge 1\,$, of $ \theta_\pi $ satisfy the equations
\begin{eqnarray}
&& \pi\, f_1 = \id_A \la{rr1} \\
&& d_R f_2 =  f_1 m_A - m_R(f_1 \otimes f_1) \la{rr2} \\
&& d_R f_n =  \sum_{i=1}^{n-1} (-1)^{i-1} f_{n-1}(\id_A^{\otimes(i-1)} \otimes m_A \otimes \id_A^{\otimes (n-1-i)})
+ \sum_{i=1}^{n-1} (-1)^{i} m_R(f_i \otimes f_{n-i})\ , \quad n\ge 2\ . \la{rrn}
\end{eqnarray}
Giving the maps $ \,f_n:\,A^{\otimes n} \to R_{n-1}\,$ is equivalent to
giving a quasi-isomorphism of $A_\infty$-algebras $ f: A \to R $, which induces the inverse of $ \pi $ on
the level of homology. The existence of such a quasi-isomorphism is a well-known result in the theory of
$A_\infty$-algebras (see \cite{K}, Theorem~3.3). Since $ \pi: R \to A $ is a homomorphism of {\it unital}
algebras, we may assume that $ f $  is a (strictly) unital homomorphism of  $A_\infty$-algebras: this means
that, in addition to  \eqref{rr1}-\eqref{rrn}, we have the relations ({\it cf.} \cite{K}, Sect.~3.3)
\begin{equation}
\la{rru}
f_1(1) = 1 \quad ,\qquad f_n(a_1, a_2, \ldots, a_n) = 0 \ , \quad n\ge 2\ ,
\end{equation}
whenever one of the $ a_i$'s equals $1$.

\bprop
\la{Tqsm}
The map \eqref{cch1} defined by the twisting cochain $ \theta_\pi $ induces a quasi-isomorphism
\begin{equation}
\la{cch2}
s\tau^\natural(\theta_\pi):\, \rCC(A) \to \bar{R}_\natural \ .
\end{equation}
The corresponding isomorphisms of homology
$\, \rHC_\bullet(A) \stackrel{\sim}{\to} \H_\bullet(\bar{R}_\natural)
\,$
depend only on $ \pi $.
\eprop
\bproof
First, note that the induced map \eqref{cch2} is well defined. Indeed, in terms of components of $ \theta $
the trace map $ s\tau^\natural(\theta) $ is given by formula \eqref{cch}.  If the components of $ \theta $
satisfy \eqref{rru}, this formula shows that the subcomplex $ \CC(k) \subseteq \CC(A) $ is mapped by
$ s\tau^\natural(\theta) $ to $ k \subseteq R $ and hence vanishes in $ \bar{R}_\natural $.

Now, since $f$ is a unital morphism, there is a commutative diagram
\begin{equation}\la{ddd}
\begin{diagram}[small, tight]
\Ba(\bar{A})     & \rTo^{f} & \Ba(\bar{R})    & \rTo^{\hat{\theta}} & \bar{R} \\
 \uInto^{\natural} &          & \uInto^{\natural} &                     &\dOnto_{s\tau} \\
\Ba(\bar{A})^\natural & \rTo^{f^\natural} & \Ba(\bar{R})^\natural & \rTo^{\ s \tau^\natural(\hat{\theta})\ } & \bar{R}_\natural[1]
\end{diagram}
\end{equation}
where $ \Ba(\bar{A}) := \Ba(A)/\Ba(k) $, $ \Ba(\bar{R}) := \Ba(R)/\Ba(k) $, etc.
The composition of the two horizontal arrows on top of \eqref{ddd} is $ \theta_\pi $, so if we identify
$ \Ba(\bar{A})^\natural = \rCC(A)[1] $ as in Theorem~\ref{TQ}, the map $ s\tau^\natural(\theta_\pi) $ is equal
to the composition of the two horizontal arrows on the bottom \eqref{ddd}. Each of these arrows is a
quasi-isomorphism: $ f^\natural $ is just the restriction of $ f $, which is a quasi-isomorphism
of DG coalgebras by construction, while $ s \tau^\natural(\hat{\theta})$ is a quasi-isomorphism
by Lemma~\ref{lqsm}. It follows that \eqref{cch2} is a quasi-isomorphism. Also, the factorization
$\, s\tau^\natural(\theta_\pi) = s \tau^\natural(\hat{\theta}) \circ f^\natural \,$ implies
$$
\H_\bullet(\tau^\natural{\theta_\pi}) =
\H_\bullet(s \tau^\natural \hat{\theta}) \circ \H_\bullet(f^\natural) = \H_\bullet(s \tau^\natural \hat{\theta})
\circ \H_\bullet(\pi^\natural)^{-1}\ ,
$$
where $\,\pi^\natural:\, \Ba(\bar{R})^\natural \to \Ba(\bar{A})^\natural\,$ is a quasi-isomorphism defined by $ \pi $.
This shows that the isomorphisms induced by \eqref{cch2} on homology depend only on the algebra map $ \pi: R \to A $.
\eproof

We can now state the main result of this section.
Fix a $k$-vector space $ V $ of (finite) dimension $d$ and, for the given almost free resolution
$\, R \to A \,$, consider the DG algebra $ R_V = (\!\sqrt[V]{R})_\nn $. As explained in Remark following
Proposition~\ref{S2P1}, the elements of $ R_V $ can be written in the `matrix' form as the images of
$\, a_{ij} \in \sqrt[V]{R} \, $, see \eqref{S2E6}, under the abelianization map
$ \sqrt[V]{R} \onto R_V $.  With this notation, we have

\bthm
\la{mcor}
The trace maps \eqref{char2} are induced by the morphism
of complexes $ T:\,\CC(A) \to R_V $, whose $n$-th graded component $\,
T_n:\,A^{\otimes (n+1)}/\im(1-t_n) \to (R_V)_n \,$ is given by
\begin{equation}
\la{trf}
T_{n}(a_1, a_2, \ldots, a_{n+1}) = \sum_{i = 1}^{d}\, \sum_{k \in \Z_{n+1}}
(-1)^{nk}  f_{n+1}(a_{1+k}, a_{2+k}, \ldots, a_{n+1+k})_{ii}\ ,
\end{equation}
where $ (f_1,\, f_2,\, \ldots) $ are defined by the relations \eqref{rr1}--\eqref{rrn} and \eqref{rru}.
\ethm
\bproof
This is a direct consequence of Proposition~\ref{Tqsm} and calculations of Section~\ref{ccnm}, see \eqref{cch}.
\eproof
For $n=0$, it is easy to see that \eqref{trf} induces
$$
\Tr_V(A)_0:\,A \to \H_0(A, V) = A_V\ ,\quad a \mapsto \sum_{i = 1}^{d} \, a_{ii}\ ,
$$
which is the usual trace map on $ \Rep_V(A) $.

We can also write an explicit formula for the first trace  $\,\Tr_V(A)_1: \HC_1(A) \to \H_1(A,V) $. For this, we fix a
section $\,f_1: A \to R_0 \,$ satisfying \eqref{rr1}, and let $ \omega:\,A \otimes A \to R_0 $ denote its `curvature':
$$
\omega(a,b) := f_1(a b) - f_1(a) f_1(b)\ ,\quad a,b \in A \ .
$$
Notice that, by \eqref{rr1}, $\,\im\,\omega \subseteq \Ker\,\pi\,$. On the other hand,
$\,\Ker\,\pi = dR_1 \cong R_1/dR_2\,$, since $ R $ is acyclic in positive degrees. Thus,
identifying $\,\Ker\,\pi = R_1/dR_2\,$ via the differential on $R$, we get a map
$\, \tilde{\omega}:\,A \otimes A \to R_1/dR_2 \,$ such that
$\,d\tilde{\omega} = \omega \,$. Using this  map, we
define\footnote{The notation `$ \ch_2 $' for \eqref{chern} is justified by
the fact that this map coincides with the second Chern character in Quillen's Chern-Weil theory
of algebra cochains (see \cite{Q3}). It would be interesting to see whether the higher traces
$ \Tr_V(A)_n $ can be expressed in terms of higher Quillen-Chern characters.}
\begin{equation}
\la{chern}
\ch_2:\, \CC_1(A) \to R_1/dR_2\ ,\quad
(a,b) \mapsto [\tilde{\omega}(a,b) - \tilde{\omega}(b,a)]\ \mbox{mod}\, dR_2\ .
\end{equation}
Since $\,\tilde{\omega} \equiv f_2\,(\mbox{mod}\, dR_2)\,$, {\it cf.} \eqref{rr2},
it follows from \eqref{trf} that $\,\Tr_V(A)_1 \,$ is induced by the map
\begin{equation}
\la{chern1}
 \Tr_V(A)_1:\, (a,b) \mapsto \sum_{i=1}^d\, \ch_2(a,b)_{ii} \ .
\end{equation}

\begin{example}
\la{Ex4.1}
For illustration, consider the polynomial algebra $ A = k[x,y] $.
Since $A$ is commutative, $ \HC_0(A) = A $ and $ \HC_1(A) $ is a quotient of
$ A \otimes A $. More precisely, by \cite{L}, Prop.~2.1.14, we have the isomorphism
$\, \rHC_1(A)  \cong \Omega_{\rm com}^1(A)/dA\,$, $\, [(a, b)] \leftrightarrow [a\,db]\,$,
where $\, \Omega_{\rm com}^1(A) \,$ is the space of $1$-forms (K\"ahler differentials) on $A$ and
$d A \subset \Omega_{\rm com}^1(A) $ is the subspace of exact forms.

On the other side, as a resolution of $A$, we take the 2-Calabi-Yau algebra $ R $ described
in Example~\ref{Ex2D}: then, for $ V = k^d $, the commutative DG algebra representing
$ \DRep_V(A) $ is given by $ R_V = k[x_{ij},\,y_{ij};\,t_{ij}] $ with differential defined
by the equations \eqref{cdpol}. Since $ R_0 = k \langle x, \,y \rangle $, with $ \pi:\,R_0 \onto A $
being the canonical projection,
we can define $ f_1: A \to R_0 $ by $\, x^k y^m \mapsto x^k y^m $. Then
$\, \omega(y,x) - \omega(x,y) = [x,y] = dt \,$ in $ R_0 $. So $\,\ch_2(y,x) = t\,(\mbox{mod}\, dR_2)\,$ and,
by \eqref{chern1}, $\, \Tr_V(A)_1: \Omega_{\rm com}^1(A)/dA \to \H_1(A,V) \,$ takes the class of $\,y\, dx \in \Omega_{\rm com}^1(A) \,$
to the class of  $\, \Tr(T) =  \sum_{i=1}^d t_{ii} \,$ and the class of $\, x\, dy \,$ to the class
of $ - \Tr(T) $. In general, using the matrix notation of Example~\ref{Ex2D}, it is easy to compute
$$
\Tr_V(A)_1:\ x^k y^m dx \mapsto \sum_{i=0}^{m-1}\,\Tr(Y^{m-1-i} X^k Y^i T)\ , \quad
x^k y^m dy \mapsto - \,\sum_{j=0}^{k-1}\,\Tr(X^{k-1-j} Y^m X^j T)\ .
$$
Since $ \Omega_{\rm com}^1(A) $ is spanned by  $ x^k y^m dx $ and $ x^k y^m dy $, the above formulas
determine $ \Tr_V(A)_1 $ uniquely.

Note that the higher traces $ \Tr_V(A)_n $ are zero in this example, since,
by construction, they factor through reduced cyclic homology, while $\, \rHC_n(A) = 0 \,$ for all
$ n \ge 2 $ (see \cite{L}, Theorem~3.2.5).
\end{example}

\subsection{Relation to Lie algebra homology}
\la{S4.5}
As mentioned in the Introduction, there is a close analogy between representation homology and
homology of matrix Lie algebras. We recall that a well-known theorem of Tsygan \cite{T} and Loday-Quillen
\cite{LQ} says that, for all $ n \ge 0 $, there are natural maps
\begin{equation}
\la{liehom}
\H_{n+1}(\gl_V(A); k) \to \HC_{n}(A) \ ,
\end{equation}
which, in the limit $\,\gl_V(A) \to \gl_\infty(A)\,$, induce an isomorphism of graded Hopf algebras
\begin{equation}
\la{liehom5}
\H_\bullet(\gl_{\infty}(A); k) \stackrel{\sim}{\to} \bL(\HC(A)[1])\ .
\end{equation}
Here, $A$ is a fixed unital $k$-algebra, $ V $ a finite-dimensional
$k$-vector space, and $ \H_\bullet(\gl_V(A); k) $ is the homology of the Lie algebra
$ \gl_V(A) := \mbox{Lie}(\End\,V \otimes A) $ with trivial coefficients.
Our aim now is to make the relation between representation homology of $A$ and the Lie algebra
homology of $ \gl_V(A) $ precise.

Let $ \DGL^+_k $ be the category of non-negatively graded (chain) DG Lie algebras over $k$.
This category has a natural model structure which is compatible with the model structure
on $ \DGA_k^+ $ via the forgetful functor $ \mbox{Lie}:\,\DGA^+_k \to \DGL^+_k $
(see \cite{Q2}, Part II, Sect. 5). Let $ \, L \sonto \gl_V(A) $ be a cofibrant resolution
of $\gl_V(A)$ in $\DGL^+_k$. The forgetful functor has left adjoint $\, U: \DGL^+_k \to \DGA^+_k $
which assigns to a DG Lie algebra $L$ its universal enveloping algebra $ U(L) $.
Hence, the Lie algebra map $ L \to \gl_V(A) $ yields a map of associative DG algebras\footnote{Note that, even though
$ L \stackrel{\sim}{\to} \gl_V(A) $ is a quasi-isomorphism, \eqref{ul} is {\it not}: in fact, since taking homology commutes
with $U$, $\, \H[U(L)] = U(\H(L)) = U(\gl_V(A)) \ne \gl_V(A)$.}
\begin{equation}
\la{ul}
U(L) \to \End\,V \otimes A\ .
\end{equation}
Next, we choose an almost free resolution $ R \sonto A $ of $A$ in $\DGA^+_k$. Tensoring
with $ \End\,V $ gives a surjective quasi-isomorphism of algebras
\begin{equation}
\la{ul1}
\End\,V \otimes R \sonto \End\,V \otimes A \ .
\end{equation}
Hence, we can lift through \eqref{ul1} the composition of
\eqref{ul} with the canonical projection $ T(L) \onto U(L) $,
where $T(L)$ is the tensor algebra of $L$. This yields a map of DG algebras:
$ T(L) \to \End\,V \otimes R $, which, in turn, induces a map of complexes
\begin{equation}
\la{ul2}
L/[L,L] \into T(L)_\n \to (\End\, V \otimes R)_\n \to R_\n \to R_V\ ,
\end{equation}
where the last two maps are induced by the natural traces.

Now, for any cofibrant resolution $ L \sonto \gl_V(A) $ in $ \DGL_k^+ $,
we have isomorphisms (see \cite{FT1}, p. 101)
$$
\H_n(L/[L,L]) \cong \H_{n+1}(\gl_V(A); k) \ ,\quad n \ge 0\ .
$$
Hence, on the level of homology the map of complexes $ L/[L,L] \to R_V $ induces
\begin{equation}
\la{liem}
\H_{n+1}(\gl_V(A); k) \to \H_n(A,V) \  ,\quad  \forall\,n \geq 0\ .
\end{equation}
For $n=0$, the map \eqref{liem} is simply the composition of the obvious traces
\begin{equation}
\la{liem0}
\H_{1}(\gl_V(A); k) = \gl_V(A)/[\gl_V(A),\, \gl_V(A)] = (\End\,V \otimes A)_\n \stackrel{\sim}{\to} A_\n \to A_V\ ,
\end{equation}
while, in general, we have the following
\bthm
\la{LQTT}
For each $ n \ge 0 $, the map \eqref{liem} factors through the Loday-Quillen-Tsygan map \eqref{liehom}. The induced maps are
precisely the trace maps \eqref{char2}.
\ethm
\bproof
Let $ (\Lambda^\bullet\,\gl_V(A), d) $ be the standard Chevalley-Eilenberg complex computing the
homology of $ \gl_V(A) $ with trivial coefficients. The Tsygan-Loday-Quillen map \eqref{liehom}
is induced by (see \cite{L}, 10.2.3)
\begin{equation}
\la{liem1}
\Lambda^{\bullet + 1}\,\gl_V(A) \xrightarrow{\vartheta_\bullet} \CC_\bullet(\End\,V \otimes A)
\xrightarrow{\tr_\bullet} \CC_\bullet(A)\ ,
\end{equation}
where $\,\vartheta_n:\,\Lambda^{n + 1}\,\gl_V(A) \to \CC_n(\End\,V \otimes A) \,$ is given by
$$
\vartheta_n(\xi_0 \wedge \xi_1 \wedge \ldots \wedge \xi_n) = \sum_{\sigma \in S_n} \,
\mbox{\rm sgn}(\sigma)\,(\xi_0,\,\xi_{\sigma(1)}, \ldots , \xi_{\sigma(n)})\ ,
$$
and $ \tr_\bullet $ is induced by the generalized trace maps
$\,\tr_n:\, (\End\,V \otimes A)^{\otimes (n+1)} \to A^{\otimes (n+1)}
\,$.

Now, for $ n = 0 $, we see that \eqref{liem0} is induced by \eqref{liem1}, so Theorem~\ref{LQTT} obviously holds
in this case. In general, for $ n \ge 1 $, the result follows from the diagram
\begin{equation}
\la{D48}
\begin{diagram}[tight, small]
L/[L,L]                         &  \rTo^{}  & \bar{R}_{\n}  & \rTo^{\Tr_V}    & \bar{R}_V  \\
\uTo^{\simeq}                   &           & \uTo^{\simeq}       & \ruTo_{T_\bullet}&              \\
\Lambda^{\bullet + 1}\,\gl_V(A) &  \rTo^{}  & \rCC(A) &                 &
\end{diagram}
\end{equation}
which commutes in the derived category $ \D(k) $. The top row in \eqref{D48}
is induced by \eqref{ul2}, the bottom is the map \eqref{liem1} composed
with the canonical projection $ \CC(A) \onto \rCC(A) $, the second vertical map is
the quasi-isomorphism given by \eqref{cch2}, and the first is also a quasi-isomorphism
which is constructed similarly to \eqref{cch2}. On the level of homology, the diagram
\eqref{D48} yields the required factorization.
\eproof

\remark\ For any Lie algebra $ \g $, there are canonical maps
$\, \H_n(\g,\,\g) \to \H_{n + 1}(\g,\,k) \,$, relating the homology of $ \g $
with coefficients in the adjoint representation to that with trivial coefficients.
At the level of standard complexes, this map is induced by $ \xi_0 \otimes \xi_1 \wedge \ldots \wedge \xi_n \mapsto
(-1)^n \xi_0 \wedge \xi_1 \wedge \ldots \wedge \xi_n $. In our case ($ \g = \gl_V(A) $),
composing with \eqref{liem} yields
\begin{equation}
\la{liemad}
\H_{n}(\gl_V(A),\,\gl_V(A)) \to \H_n(A,V) \ ,\quad \forall\, n \ge 0 \ .
\end{equation}
\subsection{Trace maps on Hochschild homology}
\la{S4.7}
We close this section by constructing a `noncommutative' analogue of the trace maps \eqref{char} defined
on the relative Hochschild homology ({\it cf.} Sect.~\ref{S3.4}). As in the cyclic (commutative) case,
we fix a finite-dimensional representation $ \varrho: S \to \END\,V $ of a DG algebra $ S \in \DGA_k $, and
for a given $S$-algebra $ A \in \DGA_S $, consider $\,\Pi_V:\,A \to \END\,V \otimes \rtv{S \bs A} \,$,
the universal {\it associative} representation of $ S\bs A$ corresponding to the identity map
in Proposition~\ref{S2P1}$(a)$. Combining this representation with the canonical trace
$ \Tr_V: \END\,V \to k $,  we define
\begin{equation}
\la{ntr}
\NTr_V:\, A \xrightarrow{\Pi_V} \END\,V \otimes \rtv{S \bs A} \xrightarrow{\Tr_V \otimes \id } \rtv{S \bs A}\ .
\end{equation}
Note that, even though the algebra $ \rtv{S \bs A} $ is noncommutative, the map \eqref{ntr} obviously factors through
$ A/[S,S] $. Slightly less obvious (but easy to check) is that \eqref{ntr} actually factors through the
`bigger' quotient $ A_{\n, S} := A/[A,S] $, and its induced map fits in the commutative diagram
\begin{equation}
\la{D1SS4}
\begin{diagram}[tight, small]
  S_\n                  &  \rTo^{}  & A_{\n,S}  \\
\dTo^{\Tr_V\circ \varrho} &           & \dTo_{\NTr_V \circ \Pi_V} \\
  k                     &  \rTo^{}  & \rtv{S \bs A}
\end{diagram}
\end{equation}
This yields a morphism of complexes
\begin{equation}
\la{E2SS4}
\Hoch(S\bs A) \to \overline{\rtv{S \bs A}}\ ,
\end{equation}
and hence the trace maps on relative Hochschild homology
\begin{equation}
\la{nchar}
\NTr_V(S\bs A)_n:\, \HH_{n-1}(S \bs A) \to \H_n(\L\rtv{S\bs A})\ ,\quad n \ge 1\ .
\end{equation}
Assembled together the linear maps \eqref{nchar} define a homomorphism of graded algebras
$$
T_k(\HH(S \bs A)[1]) \to \H_\bullet(\L\rtv{S\bs A})\ ,
$$
which is a noncommutative analogue of the homomorphism \eqref{lchar1}.

\vspace{2ex}

\noindent
{\bf Question.}\ Is there a relation (natural map) between the
noncommutative representation homology $ \H_\bullet(\L\rtv{S\bs A}) $ and the
noncommutative Lie algebra (or Leibniz) homology $ {\rm HL}_\bullet(\gl_V(A), k) \,$?
It is known that $\, {\rm HL}_\bullet (\gl_V(A), k) \,$ is related
to $\, \HH_{\bullet-1}(A) \,$  (see \cite{L}, Sect.~10.6), so one might expect that there
is an analogue of Theorem~\ref{LQTT} for Hochschild homology. However, the construction of
Section \ref{S4.5} does not seem to extend to this case, since the relative
Hochschild homology $ \HH_n(k \bs A) $ vanishes for all $ n \ge 0 \,$, {\it cf.} \eqref{hoch3}.

\section{The Derived Van Den Bergh Functor}
\la{S5}
In this section, we assume that $S = k$ and $ V $ is concentrated in degree $0$. For an algebra $ A \in \Alg_k $, we
let $\,\pi_V: A \to \End\,V \otimes A_V \,$ denote the universal representation of $A$ in $V$. Restricting
scalars via $\pi_V$, we can regard $\,\End\,V \otimes A_V \,$ as a
bimodule over $ A $, or equivalently, as a left module over the enveloping algebra $\,\eA := A \otimes A^{\rm opp} \,$.
Since $ A_V $ is commutative, the image of $ A_V $ under the natural inclusion $\,A_V \into \End\,V \otimes A_V \,$
lies in the center of this bimodule. Hence, we can regard $\,\End\,V \otimes A_V \,$ as $\,\eA$-$A_V$-bimodule.
Following Van den Bergh (see \cite{vdB}, Sect.~3.3), we now define the  functor
\begin{equation}
\la{E1}
(\mbox{---})_V:\ \Bimod(A) \to \Mod(A_V)\ , \quad M \mapsto M \otimes_{\eA} (\End\,V \otimes A_V)\ .
\end{equation}
As mentioned in the Introduction, this functor plays a key role in noncommutative geometry of smooth algebras, transforming noncommutative objects on $A$ to classical geometric objects on $ \Rep_V(A) $. Our aim is to construct the higher derived
functors of \eqref{E1}, which should replace \eqref{E1} when $A$ is not smooth.

\subsection{Representations of DG bimodules}
\la{S5.1}
We begin by extending the Van den Bergh functor to the DG setting.
Fix $\, R \in \DGA_k \,$ and let $\,\pi_V:\,R \to \End\,V \otimes \rtv{R} \,$ denote the universal DG algebra homomorphism,
see Proposition~\ref{S2P1}$(a)$. The complex
$\,\rtv{R} \otimes V \,$ is naturally a left DG module over $ \End\,V \otimes \rtv{R}  $ and right DG module
over $ \rtv{R} \,$, so restricting the left action via $ \pi $ we can regard  $\,\rtv{R} \otimes V \,$ as a DG
bimodule over $R$ and $\rtv{R}$.
Similarly, we can make $\,V^* \otimes \rtv{R}$ a $\,\rtv{R}$-$R$-bimodule.
Using these bimodules, we define the functor
\begin{equation}
\la{rtvm}
\rtv{-}:\ \DGBimod(R) \to \DGBimod(\rtv{R}) \ , \quad
M \mapsto (V^* \otimes \sqrt[V]{R}) \otimes_R M \otimes_R (\sqrt[V]{R} \otimes V) \ .
\end{equation}
Now, recall that $\,R_V := (\rtv{R})_\nn\,$ is a commutative DGA. Using the natural projection
$\,\rtv{R} \to R_V \,$, we regard $R_V$ as a DG bimodule over $ \rtv{R} $ and define
\begin{equation}
\la{abb}
(\mbox{---})_\nn:\ \DGBimod(\rtv{R}) \to \DGMod(R_V) \ , \quad
M \mapsto M_\nn := M \otimes_{(\rtv{R})^{\rm e}} R_V  \ .
\end{equation}
Combining \eqref{rtvm} and \eqref{abb}, we get the functor
\begin{equation}
\la{E11}
(\mbox{---})_V:\ \DGBimod(R) \to \DGMod(R_V)\ , \quad M \mapsto
M_V := (\rtv{M})_\nn \ .
\end{equation}
As suggested by its notation, \eqref{E11} is a DG extension of \eqref{E1}. In fact, if $ R = A $ is
a DG algebra with a single nonzero component in degree $0$, the category $ \Bimod(A) $ can be viewed as
a full subcategory of $\DGBimod(R)$ consisting of bimodules concentrated in degree $0$.
It is easy to check then that the restriction of \eqref{E11} to this subcategory coincides with \eqref{E1}.

The next lemma is analogous to Proposition~\ref{S2P1} for DG algebras; it holds, however, in greater generality:
for morphism complexes $\, \HOM $ of DG modules. We recall that, if $ R $ is a DG algebra and
$M$,$\,N$ are DG modules over $R$, $\, \HOM_R(M,\,N) \,$ is a complex of vector spaces with
$n$-th graded component consisting of all $R$-linear maps $\,f:\, M \to N\,$ of degree $ n $ and
the $n$-th differential given by $\,d(f) = d_N \circ f - (-1)^n f \circ d_M\,$.

\blemma
\la{L2}
For any $\, M \in \DGBimod(R)$, $\,N \in \DGBimod(\rtv{R}) \,$ and
$\, L \in \DGMod(R_V) \,$, there are canonical isomorphisms of complexes

\vspace{0.8ex}

$(a)$ $\,\HOM_{(\rtv{R})^{\rm e}}(\rtv{M},\,N)
\cong \HOM_{\eR}(M,\,\End\,V \otimes N) \,$,

\vspace{0.8ex}

$(b)$ $\, \HOM_{R_V}(M_V,\,L) \cong \HOM_{\eR}(M,\,\End\,V \otimes L) \,$.

\elemma

\begin{proof}
Regarding $M$ as a right DG module over $ R^{\rm e}$, we have the isomorphism
\begin{eqnarray}
\la{p2e1} (V^* \otimes \sqrt[V]{R}) \otimes_R M \otimes_R (\sqrt[V]{R} \otimes V)
&\cong& M \otimes_{R^{\rm e}} (\sqrt[V]{R} \otimes \End\,V \otimes \sqrt[V]{R})\,,\\
\nonumber    \varphi \otimes x \otimes m \otimes y \otimes v &\mapsto& \pm \,M \otimes x \otimes v\varphi \otimes y\,,
\end{eqnarray}
where the left $R^{\rm e}$-module structure on $\sqrt[V]{R} \otimes \End\,V \otimes \sqrt[V]{R}$ comes from the two canonical embeddings of $ \End\,V \otimes \sqrt[V]{R} \,$ and its right $(\rtv{R})^{\rm e}$-module structure is given by
$$
(x \otimes f \otimes y) \cdot (a \otimes b^{\circ})=\pm xa \otimes f \otimes by \ ,
$$
where $a,b,x,y \in \sqrt[V]{R}$ and $f \in \End\,V$. Here, `$\pm $'  denotes signs obtained by the Koszul sign rule.
Hence
\begin{eqnarray*}
  \HOM_{(\rtv{R})^{\rm e}}(\rtv{M},\,N) \,\,& \cong &  \HOM_{{\rtv{R}}^{\rm e}}( M \otimes_{R^{\rm e}} (\!\sqrt[V]{R} \otimes \End\,V \otimes \sqrt[V]{R})\,,N)\\
    &\cong&  \HOM_{{R}^{\rm e}}(M\,, \HOM_{{\rtv{R}}^{\rm e}}(\sqrt[V]{R} \otimes \End\,V \otimes \sqrt[V]{R}\,,N))\\
    &\cong& \HOM_{R^{\rm e}}(M\,,\End\,V \otimes N)\ .
\end{eqnarray*}
This proves $(a)$. Part $(b)$ is an immediate consequence of $(a)$ and the standard $\otimes$-Hom adjunction.
\end{proof}

\begin{example}
\la{ExO}
For a DG algebra $R$, let $ \Omega^{1} R $ denote the
kernel of the multiplication map $\,R \otimes R \to R\,$.
This is naturally a DG bimodule over $ R $, which, as in the case
of ordinary algebras, represents the complex of $k$-linear graded
derivations $\,\DER (R,\,M) \,$ (see, e.g., \cite{Q1}, Sect.~3).
Thus, for any $\, M \in \DGBimod(R) \,$, there is a canonical isomorphism of
complexes
\begin{equation}
\la{der}
\DER (R,\,M) \cong \HOM_{\eR}(\Omega^{1} R,\,M)\ .
\end{equation}
Lemma~\ref{L2} then implies canonical isomorphisms
\begin{equation}
\la{omvdb}
\rtv{\Omega^1 R} \cong \Omega^1(\rtv{R}) \quad , \quad
(\Omega^1 R)_V \cong  \Omega_{\rm com}^1(R_V)\ .
 \end{equation}
Indeed, for any $ N \in \DGBimod(\rtv{R}) $, we have
\begin{eqnarray*}
\HOM_{(\rtv{R})^{\rm e}}(\rtv{\Omega^1 R},\, N) &\cong& \HOM_{\eR}(\Omega^1 R,\, \End\,V \otimes N) \quad
[\,\mbox{see Lemma~\ref{L2}$(a)$}\,] \\
&\cong& \DER(R,\,\End\,V \otimes N)\\
&\cong& \DER_{\End(V)}(\End\,V \ast R,\ \End\,V \otimes N)\\
&\cong& \DER_{\End(V)}(\End\,V \otimes \rtv{R},\ \End\,V \otimes N)\quad [\,\mbox{see}\, \eqref{S2E8}\,]\\
&\cong& \DER(\rtv{R},\, N) \\
&\cong&\HOM_{(\rtv{R})^{\rm e}}(\Omega^1(\rtv{R}),\, N)\ ,
\end{eqnarray*}
where $\,\DER_{\End(V)}\,$ denotes the space of derivations vanishing on
$\, \End\,V \,$. Thus, the first isomorphism in  \eqref{omvdb} follows
from Yoneda's Lemma. The second isomorphism is proven in a similar fashion.

As a consequence of \eqref{der} and \eqref{omvdb}, for any DG bimodule $M$,
we have a natural map:
\begin{equation}
\la{trns}
\DER(R,M) \cong \HOM_{\eR}(\Omega^{1} R,\,M) \xrightarrow{(\mbox{--})_V}
\HOM_{R_V}(\Omega_{\rm com}^{1}(R_V),\,M_V) \cong \DER(R_V,M_V)\ .
\end{equation}
\end{example}

Next, we recall that if $ R $ is a DG algebra, every
DG module $M$ over $ R $ has a {\it semifree resolution}
$\,F \to M \,$, which is similar to a free (or projective)
resolution for ordinary modules over ordinary algebras
(see \cite{FHT}, Sect.~2). To construct the derived
functors of \eqref{E1} we now follow the standard procedure
in differential homological algebra ({\it cf.} \cite{HMS}, \cite{FHT}).

Given an algebra $ A \in \Alg_k $ and a complex $M$ of bimodules over $A$, we first
choose an almost free resolution $\,f: R \to A\,$ in $\, \DGA_k \,$ and consider $ M $
as a DG bimodule over $ R $ via $f$. Then, we choose a semifree resolution
$\,F(R,M) \to M \,$  in the category $ \DGBimod(R) $ and apply to $F(R,M)$
the functor \eqref{E11}. The result is described by the following theorem,
which is the main result of this section.
\begin{theorem}
\la{T2}
Let $ A \in \Alg_k $, and let $ M $ be a complex of bimodules over $A$.

$(a)$\ The assignment $\,M \mapsto F(R,M)_V\,$ induces a well-defined functor
between the derived categories
$$
\L(\mbox{---})_V:\ \D(\Bimod\,A) \to \D(\DGMod\,R_V)\ ,
$$
which is independent of the choice of the resolutions $ R \to A $ and $ F \to M $ up to auto-equivalence of
$\, \D(\DGMod\,R_V) $ inducing the identity on homology.

$(b)$\ Taking homology $\,M \mapsto \H_\bullet[\L(M)_V]\,$ yields a functor
$$
\H_\bullet(\,\mbox{--}\,,\, V):\ \D(\Bimod\,A) \to \GrMod(\H_\bullet(A,V))\ ,
$$
which depends only on the algebra $A$ and the vector space $ V $. We call $ \H_\bullet(M,V) $
the {\rm representation homology} of $M$ with coefficients in $V$.

$(c)$\ If $ M \in \Bimod(A) $ is viewed as a 0-complex in $\D(\Bimod\,A)$, then
$\,\H_0(M, V) \cong M_V \,$.
\end{theorem}
\begin{proof}
The proof of part $(a)$ is standard differential homological
algebra. Given an almost free resolution $R$ of $A$, and two semifree resolutions $F_1 \,\to\, M$ and $F_2 \,\to\, M$ in
$\D(\DGBimod\,R)$, one has a quasi-isomorphism
$F_1 \,\to\, F_2$ such that the composite map $F_1 \,\to\, F_2 \,\to\, M$ is homotopic to the resolution $F_1 \,\to\, M$ by an $R$-linear homotopy (see Proposition 2.1(ii) of \cite{FHT}). By Proposition 2.3(i) of \cite{FHT}, $(F_1)_V$ is quasiisomorphic to $(F_2)_V$. Given a morphism $ M \,\to\, N$ of complexes of $A$-bimodules, and resolutions $F_1 \,\to\, M$ and $F_2 \,\to\, N$ in $\D(\DGBimod\,R)$, one uses Proposition 2.1(ii) of \cite{FHT} to obtain a map $F_1 \,\to\,F_2$ (unique up to $R$-linear homotopy) such that the composite $F_1 \,\to\, F_2 \,\to\, N$ is homotopic to the composite $F_1 \,\to M \,\to N$ by an $R$-linear homotopy. Applying $(\mbox{---})_V$ to $F_1 \,\to\, F_2$ yields a well defined map $(F_1)_V \,\to\, (F_2)_V$ in $\D(\DGMod\,R_V)$ (which is a quasi-isomorphism if $M \,\to\, N$ is a quasi-isomorphism by Proposition 2.3(ii) of \cite{FHT} and which is clearly compatible with compositions of morphisms of complexes of $A$-bimodules). More generally, our arguments show that the assignment $M \mapsto\, F(R,M))_V$ yields a functor from $\D(\DGBimod\,R)$ to $\D(\DGMod\,R_V)$. We however, focus on the composition of this functor with the (fully faithful) embedding $\D(\DGBimod \,A) \,\to\, \D(\DGBimod\,R)$.

It remains to check that $ \L(\mbox{---})_V $ is independent of
the choice of almost free resolution of $A$. Let $R_1$ and $R_2$ be two almost free resolutions of $A$. By Proposition 3.2 of \cite{FHT}, there is a morphism $f: R_1 \,\to\, R_2$ in $\DGA_k$ such that $\H_0(f)=\text{id}_A$. Note that if $M$ is a DG $R_1$-bimodule, $M \otimes_{{R_1}^{\rm e}} R_2^{\rm e}$ is a DG $R_2$-bimodule, which is semifree if $M$ is semifree. Here, $R_2^{\rm e}$ gets a left $R_1^{\rm e}$-module structure via $f \otimes f$. It follows that $\, \F: \, M \mapsto F(R_1,M) \otimes_{{R_1}^{\rm e}}R_2^{\rm e}$ yields a well defined functor
$\,\D(\DGBimod\,R_1) \,\to\, \D(\DGBimod\,R_2)\,$. Similarly, $\, \F_V:\,N \mapsto F((R_1)_V,N) \otimes_{(R_1)_V} (R_2)_V\,$ is
a well defined functor $\D(\DGMod\,(R_1)_V) \,\to\, \D(\DGMod\,(R_2)_V)$. The following diagram commutes:
$$
\begin{diagram}[small]
\D(\DGBimod\,R_1)   &  \rTo^{\L(\,\mbox{--}\,)_V } & \D(\DGMod\,(R_1)_V)  \\
\dTo^{\F}           &                           & \dTo_{\F_V} \\
\D(\DGBimod\, R_2)   &  \rTo^{\L(\,\mbox{--}\,)_V}  & \D(\DGMod\,(R_2)_V)
\end{diagram}
$$
We claim that the vertical arrows in the above diagram are equivalences of categories. Indeed, by the standard $\otimes-\Hom$ adjunction and the fact that semifree $R_1$-modules are cofibrant in the sense of Section 8.3 of \cite{K1}, the functor induced by $M \mapsto F(R_1,M) \otimes_{{R_1}^{\rm e}}R_2^{\rm e}$ is the left adjoint of the restriction functor $\D(\DGBimod\,R_2) \,\to\, \D(\DGBimod\,R_1) $. That the latter functor is an equivalence of categories follows from Proposition 8.4 of \cite{K1}. Hence, $M \mapsto F(R_1,M) \otimes_{{R_1}^{\rm e}}R_2^{\rm e}$ induces an equivalence of categories $\D(\DGBimod\,R_1) \,\to\, \D(\DGBimod\,R_2)$. Similarly, one shows that the right vertical arrow in the above diagram is an equivalence of categories. To complete the proof of $(a)$, we note that the left vertical arrow commutes with the natural embeddings $\D(\DGBimod\,A) \,\to\, \D(\DGBimod\,R_1)$ and $\D(\DGBimod\,A) \,\to\, \D(\DGBimod\,R_2) $, respectively. Part $(b)$ is immediate from $(a)$.

To prove part $(c)$ we use Lemma 2$(b)$. For any $A_V$-module $L$  concentrated in degree $0$ (and
thought of as a DG $R_V$-module via $ R \onto A $), we have
\begin{eqnarray*}
\Hom_{A_V}(\H_0(M, V),\,L) & \cong & \Hom_{R_V}(F(R,M)_V,\,L)\\
 &\cong & \H_0(\HOM_{R_V}(F(R,M)_V,\,L))\\
 &\cong & \H_0(\HOM_{R^{\rm e}}(F(R,M),\ \End\,V \otimes L))\\
&\cong& \Hom_{R^{\rm e}}(F(R,M),\ \End\,V \otimes L)\\
&\cong& \Hom_{A^{\rm e}}(M, \ \End\,V \otimes L) \ .
\end{eqnarray*}
The result now follows from Yoneda's Lemma.
\end{proof}

\subsubsection{A model-categorical approach} There is an alternative way to
derive the Van den Bergh functor using model categories. This approach is
parallel to the one described in Section~\ref{S2.2}, so we only briefly
sketch it here.

Let $ \DGBA_k $ denote the category of pairs $ (A,M) $ with
$ A \in \DGA_k $ and  $ M \in \DGBimod(A) $; the morphisms in $ \DGBA_k $ are
given by $\, (f,\varphi): (A,M) \to (B,N) \,$, where $ f: A \to B $ is a morphism
of DG algebras and $\,\varphi: M \to N \,$ is a morphism of DG $A$-bimodules
(with $N$ viewed as a bimodule over $A$ via $f$). Similarly, we
define the category $ \cDGMA_k $, whose objects are pairs of commutative DG algebras and
DG modules over such algebras. Both these categories are fibred:
$ \DGBA_k $ is fibred over $ \DGA_k $ and $ \cDGMA_k $ is fibred over
$ \cDGA_k $, and they have natural model structures compatible
with the standard model structures on $\DGA_k$ and $ \cDGA_k $
(see \cite{R}). Now, the functors \eqref{S2E9} and \eqref{E11}
can be combined together to define
\begin{equation}
\la{fibr}
(\,\mbox{--}\,)_V :\ \DGBA_k \to \cDGMA_k\ ,\quad (A,M) \mapsto (A_V, M_V)\ .
\end{equation}
By Proposition~\ref{S2P1} and Lemma~\ref{L2}, this functor has the right adjoint
$\,\End\,V \otimes\, \mbox{--} \,$ preserving fibrations and acyclic fibrations.
Thus, we have a Quillen pair $\,(\,\mbox{--}\,)_V :\,\DGBA_k \rightleftarrows \cDGMA_k\,:
\End\,V \otimes\, {\mbox{--}}\,$.
By Quillen's Adjunction Theorem, \eqref{fibr} then has a total left derived functor
$\, \L(\,\mbox{--}\,)_V:\,\Ho(\DGBA_k) \to \Ho(\cDGMA_k) \,$, which is adjoint to
$\,\End\,V \otimes \mbox{--} \,$. One can show that this derived functor agrees with
the `semi-abelian' derived functor constructed in Theorem~\ref{T2}.

\subsection{Derived tangent spaces}
\la{S5.2}
We will use the above construction to compute the derived tangent spaces $
\T_\varrho \DRep_{V}(A)_\bullet $ at $\,\varrho \in \Rep_{V}(A)\,$.
It is instructive to compare our computation with the ones in \cite[Section~3.3]{vdB}.
Recall that if $X$ is an ordinary affine $k$-scheme and
$ x \in X $ is a closed $k$-point, the tangent space
$ \T_x X $ at $ x $ is defined by
$\,\T_x X := \Der(k[X],\,k_x)\,$.  Similarly ({\it cf.} \cite{CK}, (2.5.6)), if $ \X $ is a DG scheme, and $ x $ is a closed $k$-point of the underlying scheme $ \X^0 $,
the tangent DG space to $ \X $ at $x$ is defined by
$$
\T_x\,\X := \DER(k[\X],\,k_x)\ .
$$
Now, if $ X_\bullet := \boldsymbol{\Spec}\,\H_\bullet(k[\X]) $ is
the underlying derived scheme of $ \X $, then, by definition, the {\it derived tangent space}
$ (\T_x X)_\bullet $ to $ X_\bullet $ is given by
\begin{equation}
\la{tspace}
(\T_x X)_\bullet := \H_\bullet[\DER(k[\X],\,k_x)]\ .
\end{equation}

Let $\,\varrho:\,A \to \End\,V\,$ be a fixed representation of $A$. Choose a cofibrant resolution $ R \sonto A $,
and let $\,\varrho_V:\,R_V \to k \,$ be the DG algebra homomorphism corresponding to the representation
$\,\varrho:\,R \to A \to \End\,V$. Then, for $ M = \End\,V $, the canonical map \eqref{trns}
is an isomorphism of complexes: indeed,
\begin{eqnarray*}
\DER(R_V, \, k) &\cong& \HOM_{R_V}(\Omega_{\rm com}^1(R_V),\,k) \\
                 &\cong& \HOM_{R_V}((\Omega^1 R)_V,\,k)\qquad [\,\mbox{see}\,\eqref{omvdb}\,] \\
                 &\cong& \HOM_{\eR} (\Omega^1 R ,\,\End\,V) \qquad [\,\mbox{see\,Lemma~\ref{L2}$(b)$}\,]\\
                 &\cong& \DER(R,\,\End\,V)\ .
\end{eqnarray*}
This implies that $\, \T_\varrho\DRep_{V}(A)_\bullet := \H_\bullet[\DER(R_V, \, k)] \cong \H_\bullet[\DER(R,\,\End\,V)]\,$.
The following proposition is now a direct consequence of \cite{BP}, Lemma~4.2.1 and Lemma~4.3.2.
\bprop
\la{P2}
There are canonical isomorphisms
$$
\T_\varrho\DRep_{V}(A)_n \cong \left\{
\begin{array}{lll}
\Der(A,\,\End\,V)\ & \mbox{\rm if} &\ n = 0\\*[1ex]
\HH^{n+1}(A,\,\End\,V)\ & \mbox{\rm if} &\ n \ge 1
\end{array}
\right.
$$
where $ \HH^{\bullet}(A,\,\End\,V) $ denotes the Hochschild cohomology of the representation $ \varrho: A \to \End\,V $.
\eprop

\begin{remark}
As explained in Section~\ref{S2.3.5}, in case when $V$ is a single vector space and $ R \in \DGA_k $ is almost free,
one can show that $ \Rep_V(R) $ is isomorphic the DG scheme $ \RAct(R,\,V) $ constructed in \cite{CK}.
This implies that $\,\T_\varrho\DRep_{V}(A)_\bullet\,$ should be isomorphic to $\,\T_\varrho \RAct(R,\,V)_\bullet\,$,
which is indeed the case, as one can easily see by comparing our Proposition~\ref{P2} to \cite{CK},
Proposition~3.5.4$(b)$.
\end{remark}

\subsection{Traces on bimodules}
\la{S5.3}
If $M$ is a bimodule over a DG algebra $A$, a {\it trace} on $M$
is a map of complexes $ \tau: M \to N $ vanishing on the commutator subspace
$ [A,M] \subseteq M$.
Every trace on $M$ factors through the canonical projection $ M \onto M_\n := M/[A,M] $,
which is thus the universal trace.

Now, given a finite-dimensional vector space $V$, let $ \pi_V(M) $ denote the canonical map
corresponding to $ \id_{M_V} $ under the isomorphism of Lemma~\ref{L2}. The map of complexes
\begin{equation}
\la{S5E1}
\Tr_V(M)\,:\ M \xrightarrow{\pi_V(M)} \End\,V \otimes M_V \xrightarrow{\Tr_V \otimes \id} M_V\ ,
\end{equation}
is then obviously a trace, which is functorial in $M$.
Thus \eqref{S5E1} defines a morphism of functors
\begin{equation}
\la{S5E2}
\Tr_V:\ (\,\mbox{--}\,)_\n \to (\,\mbox{--}\,)_V\ .
\end{equation}
As in the case of DG algebras, \eqref{S5E2} induces a morphism of functors
$\, \D(\DGBimod\,A) \to \D(k)\,$:
\begin{equation}
\la{S5E3}
\Tr_V:\ \L(\,\mbox{--}\,)_\n \to \L(\,\mbox{--}\,)_V\ ,
\end{equation}
where $ \L(\,\mbox{--}\,)_V $ is the derived representation functor introduced in Theorem~\ref{T2}.
To describe \eqref{S5E3} on $ M \in \DGBimod(A) $ explicitly we choose an
almost free resolution $ p:\,R \sonto A $, regard $M$ as a bimodule over $R$ via $p$
and choose a semifree resolution of $ F(R,M) \sonto M $ in $ \DGBimod(R)$. Then
\eqref{S5E3} is induced by the map \eqref{S5E1} with
$ M $ replaced by $ F(R, M) $:
\begin{equation}
\la{S5E4}
\Tr_V(M):\ F(R, M)_\n \to F(R,M)_V\ .
\end{equation}
Note that, if $ A \in \Alg_k $ and $ M \in \Bimod(A) $, then
$\,\H_n[F(R, M)_\n] \cong \HH_n(A,M) \,$ for all $ n \ge 0 $,
so \eqref{S5E4} induces the trace maps on Hochschild homology:
\begin{equation}
\la{S5E5}
\Tr_V(M)_n:\ \HH_n(A,M) \to \H_n(M, V)\ ,\quad \forall\,n \ge 0\ ,
\end{equation}
where $ \H_n(M, V) := \H_n[\L(M)_V] $ is the representation homology of $M$,
see Theorem~\ref{T2}$(b)$.

\blemma
\la{S5L3}
Let $M$ be a DG bimodule over a DG algebra $R$. Then,
for any $ \partial \in \DER(R,M)$, the following diagram commutes:
\begin{equation}\la{S5D0}
\begin{diagram}[small, tight]
R &  \rTo^{\partial} & M \\
\dTo^{\Tr_V(R)}           &                           & \dTo_{\Tr_V(M)} \\
R_V&  \rTo^{\partial_V}  & M_V
\end{diagram}
\end{equation}
where $ \partial_V $ is the image of $ \partial $ under \eqref{trns}.
\elemma
\bproof
Let $ R \ltimes M  \in \DGA_k $ denote the semi-direct product of $R$ and $M$.
It is easy to check that $ R \ltimes M $ fits in the commutative diagram
\begin{equation}\la{S5D1}
\begin{diagram}[small]
R \ltimes M &       & \rTo^{(\Tr_V(R),\,\Tr_V(M))}  &        & R_V \ltimes M_V\\
  &\rdTo_{\Tr_V(R \ltimes M)}  &           &  \ruOnto_{p} &  \\
  &       &     (R \ltimes M)_V     &        &
\end{diagram}
\end{equation}
where $p$ is the canonical projection. Now, any derivation
$ \partial: R \to M $ gives an algebra map $ 1 + \partial:\,R \to R \ltimes M $.
Since $ \Tr_V(R) $ is functorial in $R$, we have a commutative diagram
\begin{equation}\la{S5D2}
\begin{diagram}[small, tight]
R              &  \rTo^{1+\partial}  & R \ltimes M \\
\dTo^{\Tr_V(R)}&                     & \dTo_{\Tr_V(R \ltimes M)} \\
R_V            &  \rTo^{(1+\partial)_V}& (R\ltimes M)_V
\end{diagram}
\end{equation}
Combining \eqref{S5D1} and \eqref{S5D2} yields \eqref{S5D0}.
\eproof

\subsection{Periodicity and the Connes differential}
\la{S5.33}
One of the fundamental properties of cyclic homology
is Connes' periodicity exact sequence ({\it cf.} \cite{L}, 2.2.13):
\begin{equation}
\la{ISB}
\ldots \to \rHH_n(A) \xrightarrow{I} \rHC_n(A) \xrightarrow{S}
\rHC_{n-2}(A) \xrightarrow{B} \rHH_{n-1}(A) \to \ldots
\end{equation}
This sequence involves two important operations on cyclic homology:
the periodicity operator $S$ and the Connes differential $B$.
Our aim is to study the effect of these operations
on representation homology. Since \eqref{ISB} arises naturally from
a cyclic bicomplex, we begin by constructing such a bicomplex from
an arbitrary resolution.

For $ A \in \Alg_k $, fix an almost free resolution
$\,p: R \sonto A \,$ in $ \DGA_k^+ $ and consider the
periodic complex
\begin{equation}
\la{Qper}
X(R) := [\,\ldots \xleftarrow{\bpar} \bar{R} \xleftarrow{\beta} \Omega^1(R)_\n
\xleftarrow{\bpar} \bar{R} \xleftarrow{\beta}  \ldots\,]\ .
\end{equation}
Recall that the differential $ \bpar $ in \eqref{Qper} is induced by the universal derivation
$\,\partial:\,R \to \Omega^1(R) \,$ (see Example~\ref{ExO}) and $ \beta
$ is defined by $\,\beta(x\,dy) = [x, y]\,$ for $ x, y \in R $.
Since $R$ is a free algebra, the complex \eqref{Qper} is exact
(see \cite{Q3}, Example~3.10). Truncating it at $ \bar{R}$-term in degree $0$ and changing
the sign of differentials on $ \Omega^1(R)_\n $, we define the first quadrant bicomplex
\begin{equation}
\la{Qper+}
X^+(R) := [\,0 \leftarrow \bar{R} \xleftarrow{\beta} \Omega^1(R)_\n
\xleftarrow{\bpar} \bar{R} \xleftarrow{\beta}  \ldots\,]\ .
\end{equation}
The sub-bicomplex of $ X^+(R) $ obtained by replacing the first (nonzero) column of $ X^{+}(R) $ by $ \im(\beta) $
has exact rows, hence its total complex is acyclic. It follows that the quotient map
\begin{equation}
\la{Totp}
\Tot\,X^+(R) \sonto \bar{R}_\n
\end{equation}
defined by the canonical projection $ \tau: \bar{R} \onto \bar{R}_\n\,$ from the first column
is a quasi-isomorphism. By Proposition~\ref{S4P1}, we  conclude
\begin{equation}
\la{tothc}
\H_n[\Tot\,X^+(R)] \cong \H_n[\bar{R}_\n] \cong \rHC_n(A)\ ,\quad \forall\,n\ge 0 \,.
\end{equation}
Next, consider the sub-bicomplex $ X_{2}^{+}(R) $ consisting of the first two columns of $ X^{+}(R) $;
its total complex is the c\^{o}ne of $ \beta $. The algebra map
$\, p: R \to A $ induces a morphism of complexes
\begin{equation}
\la{th1}
\Tot\,X_{2}^{+}(R)  = \cn[\bar{R} \xleftarrow{\beta} \Omega^1(R)_\n] \sonto
\cn[\bar{A} \xleftarrow{p \beta} \Omega^1(R)_\n]\ ,
\end{equation}
which is a quasi-isomorphism since so is $ p $. The complex
$$
\cn(p \beta) = [\,0 \leftarrow \bar{A} \xleftarrow{p\,\beta_0}
\Omega^1_\n(R)_0 \leftarrow \Omega^1_\n(R)_1 \leftarrow \ldots \,]
$$
computes the reduced Hochschild homology of $A$ ({\it cf.} \cite{FT}); hence, \eqref{th1} induces
\begin{equation}
\la{th2}
\H_n[\Tot\,X_{2}^{+}(R) ] \cong \rHH_n(A)\ ,\quad \forall\,n\ge 0 \,.
\end{equation}
Now, the natural exact sequence of bicomplexes
$$
0 \to X_{2}^{+}(R)  \to X^+(R) \to X^+(R)[2,0] \to 0
$$
gives the short exact sequence of total complexes
$$
0 \to \Tot\, X_{2}^{+}(R)  \to \Tot\,X^+(R) \xrightarrow{S} \Tot\,X^+(R)[2] \to 0\ ,
$$
which, with identifications \eqref{tothc} and \eqref{th2}, induces
Connes' long exact sequence \eqref{ISB}.

The periodicity map $S$ in \eqref{ISB} can be described as follows.
Let $ \bar{r}_n \in \bar{R}_\n $ be a cycle representing a homology class
in $ \rHC_{n}(A) $. Since \eqref{Totp} is a quasi-isomorphism,
there is a cycle in $ \Tot\, X^+(R) \,$, say
\begin{equation}
\la{cyeq}
(r_n,\, \omega_{n-1},\, r_{n-2},\, \ldots \,) \,\in\, \Tot[X^+(R)]_n = \bar{R}_n \oplus \Omega_\n^1(R)_{n-1} \oplus \bar{R}_{n-2} \oplus \ldots \ ,
\end{equation}
that projects onto $ \bar{r}_n $ under \eqref{Totp}. At the level of total complexes, the map $S$ is defined by
$$
S(r_n,\, \omega_{n-1},\, r_{n-2},\, \ldots \,) = (r_{n-2},\, \ldots \,) \ .
$$
Now, the condition $\, d_{\Tot}(r_n,\, \omega_{n-1},\, r_{n-2},\, \ldots \,) = 0 \,$ implies the equations
\begin{equation}
\la{cyeq1}
\tau(r_n) = \bar{r}_n\, ,\quad \beta(\omega_{n-1}) = - dr_n\ ,\quad \bpar(r_{n-2}) = d \omega_{n-1}\ ,\ \ldots
\end{equation}
which can be solved successively using the exact sequence
\begin{equation}
\la{Qper1}
0 \leftarrow \bar{R}_\n \xleftarrow{\tau} \bar{R} \xleftarrow{\beta} \Omega^1(R)_\n \xleftarrow{\bpar}
\bar{R}_\n \leftarrow 0\ .
\end{equation}
Note that the exactness of \eqref{Qper1} at the right implies that $ \bar{r}_{n-2} :=
\tau(r_{n-2}) $ is a cycle in $ \bar{R}_\n $, and with identification \eqref{tothc},
the map $\, S: \rHC_{n}(A) \to \rHC_{n-2}(A) $ is thus given by
$\, S[\bar{r}_n] = [\bar{r}_{n-2}]\,$.

To construct the analogue of a cyclic bicomplex for representation homology we now apply to \eqref{Qper}
our trace maps \eqref{E1S4} and \eqref{S5E1}. By Lemma~\ref{S5L3}, we have
the commutative diagram
\begin{equation}
\la{S5D7}
\begin{diagram}[small, tight]
\ldots &  \lTo^{\bpar} & \bar{R}         & \lTo^{\beta}  & \Omega^1(R)_\n &  \lTo^{\bpar} & \bar{R} & \lTo^{\beta} & \ldots \\
       &               & \dTo^{\Tr_V} &                  & \dTo^{\Tr_V}  & & \dTo^{\Tr_V} &  & \\
\ldots &  \lTo^{\bpar_V}     & \bar{R}_V       & \lTo^{0}& \Omega_{\rm com}^1(R_V) &  \lTo^{\bpar_V} & \bar{R}_V & \lTo^{0} & \ldots
\end{diagram}
\end{equation}
where $ \bpar_V $ is induced by the de Rham differential $ \partial_V: R_V \to \Omega_{\rm com}^1(R_V)\,$. The diagram \eqref{S5D7}
suggests that
\begin{equation}
\la{Qperv}
X(R)_V := [\,\ldots \xleftarrow{\bpar_V} \bar{R}_V \xleftarrow{0} \Omega_{\rm com}^1(R_V) \xleftarrow{\bpar_V}
\bar{R}_V \xleftarrow{0}  \ldots\,]
\end{equation}
should be viewed as a periodic bicomplex for representation homology. In analogy with cyclic theory, we define then
a non-negative bicomplex $ X^+(R)_V $ by truncating \eqref{Qperv} at  $ \bar{R}_V $ term in degree $0$  ({\it cf.}
\eqref{Qper+}). The total complex of $  X^+(R)_V $ is the direct sum of complexes
\begin{equation}
\la{Totv}
\Tot\,X^+(R)_V = \bar{R}_V\,\oplus\,\cn(\bpar_V)[1]\,\oplus\,\cn(\bpar_V)[3]\,
\oplus\,\ldots \ ,
\end{equation}
its $n$-th term is given by
$$
[\Tot\,X^+(R)_V]_n = (\bar{R}_V)_n \oplus \Omega_{\rm com}^1(R_V)_{n-1} \oplus (\bar{R}_V)_{n-2} \oplus \Omega_{\rm com}^1(R_V)_{n-3} \oplus \ldots
$$

Now, by Theorem~\ref{comp}, $ R_V $ is isomorphic to a (graded) polynomial algebra. Hence the reduced de Rham differential $ \bpar_V:\,\bar{R}_V \to \Omega_{\rm com}^1(R_V) $ is injective, and the canonical map $\,\cn(\bpar_V) \to \Omega_{\rm com}^1(R_V)/\partial R_V \,$ is a quasi-isomorphism. With identification \eqref{Totv}, we thus have a quasi-isomorphism
\begin{equation}
\la{Totv1}
\Tot\,X^+(R)_V \sonto \bar{R}_V\,\oplus\,
(\Omega_{\rm com}^1(R_V)/\partial R_V)[1] \,\oplus\, (\Omega_{\rm com}^1(R_V)/\partial R_V)[3] \,\oplus\,\ldots\
\end{equation}
mapping $\,
(r_n, \omega_{n-1}, r_{n-2}, \omega_{n-3}, \ldots) \mapsto
(r_n, \bar{\omega}_{n-1}, \bar{\omega}_{n-3}, \ldots) \,$,
where $\,\bar{\omega} = \omega\,(\mbox{mod}\,\partial R_V)\,$.
It follows that
\begin{equation*}
\la{Totv1e}
\H_n[\,\Tot\,X^+(R)_V] \cong \bar{\H}_n(A,V)\,\oplus\,
\H_{n-1}[\Omega_{\rm com}^1(R_V)/\partial R_V] \,\oplus\, \H_{n-3}[\Omega_{\rm com}^1(R_V)/\partial R_V]\,\oplus\,\ldots
\end{equation*}
We denote this last homology by $ \tH_n(A,V) $ and refer to it as the {\it cyclic representation homology} of $A$.
Note that $ \bar{\H}_n(A,V) $ appears as a direct summand of $ \tH_n(A,V) $, and the trace maps \eqref{char1} can be naturally extended to $ \tH_n(A,V) \,$. Precisely,
\begin{equation}
\la{trex}
\tTr_V(A)_n:\, \rHC_n(A) \to \tH_n(A,V)\ ,\quad [\bar{r}_n] \mapsto
(\Tr_V(r_n),\, \Tr_V(\bar{\omega}_{n-1}),\,\Tr_V(\bar{\omega}_{n-3}), \,\ldots\,)\ ,
\end{equation}
where, for $n$-cycle $ \bar{r}_n \in \bar{R}_\n $, the elements $ (r_n,\,\omega_{n-1}, r_{n-2},
\omega_{n-3}, \ldots) $ are defined as in \eqref{cyeq} by equations \eqref{cyeq1}.


Now, let $ X_{2}^+(R)_V $ be the sub-bicomplex of $ X^+(R)_V $ consisting of the first two columns. Its total
complex is  $\, \Tot\,X_{2}^+(R)_V = \bar{R}_V \oplus \Omega_{\rm com}^1(R_V)[1]\,$. Notice that, since $R$ is a free
algebra, $\,\Omega^1(R) \sonto \Omega^1(A)\,$ is a semifree resolution of $ \Omega^1(A)$ in the category of
$R$-bimodules. Hence, by  Theorem~\ref{T2},
$$
\H_\bullet[\Omega_{\rm com}^1(R_V)] \cong \H_\bullet[\Omega^1(R)_V] \cong \H_\bullet(\Omega^1 A, V)
$$
is the representation homology of the $A$-bimodule $ \Omega^1 A $, which is independent
of the resolution $R$. Thus
\begin{equation}
\la{Tot2}
\H_n[\,\Tot\,X_{2}^+(R)_V] \cong \bar{\H}_n(A,V)\,\oplus\, \H_{n-1}(\Omega^1 A, V) \ , \quad \forall\,n \ge 0 \ .
\end{equation}
The relation between $ \H_\bullet(A,V) $ and $ \H_\bullet(\Omega^1 A,V) $ is clarified by the following
\blemma
\la{LLL}
For any $ A \in \Alg_k $, there is a spectral sequence\footnote{To simplify the notation, for
$ M \in \Bimod(A) $, we denote by $ \HH_{p+1}(A,M) $ the usual Hochschild homology of $M$
if $ p \ge 1 $ and $\, (\Omega^1 A) \otimes_{\eA} M \,$ if $ p = 0 $.}
\begin{equation}
\la{sps1}
E^2_{p,q} = \HH_{p+1}(A,\,\End\,V \otimes \H_q(A,V))\ \Rightarrow \ \H_n(\Omega^1 A, V)\ .
\end{equation}
In particular, if $A$ is a smooth algebra, then $\,\H_n(\Omega^1 A, V) \cong
\HH_{n+1}(A,\,\End\,V \otimes A_V) = 0 \,$ for $\, n > 0 $.
\elemma
\bproof
By \eqref{omvdb}, $\, \Omega_{\rm com}^1(R_V) \cong \Omega^1(R)_V = \Omega^1(R) \otimes_{\eR} (\End\,V \otimes R_V)\,$.
Hence there is a standard Eilenberg-Moore spectral sequence of the form (see \cite{McC}, Theorem~7.6)
\begin{equation}
\la{elsp}
E^2_{p, \ast} = \Tor_{p}^{\H_\bullet(\eR)}(\H_\bullet(\Omega^1 R),\,\H_\bullet(\End\,V \otimes R_V))\
\Rightarrow \ \H_n[\Omega_{\rm com}^1(R_V)]\ .
\end{equation}
As mentioned above, the algebra map $\,p:\, R \sonto A \,$ induces quasi-isomorphisms
$\,\eR \sonto \eA \,$ and $\,\Omega^1(R) \sonto \Omega^1(A)\,$, and since
$\,\Tor^{\eA}_p(\Omega^1 A,\,\mbox{--}\,) \cong \HH_{p+1}(A,\,\mbox{--}\,)\,$, the spectral
sequence \eqref{elsp} can identified with \eqref{sps1}. The last claim follows from the fact
that $ \H_q(A,V)  = 0 $ for $ q > 0 $ if $A$ is a smooth algebra (see
\cite[Theorem~21]{BFR}).
\eproof

The short exact sequence of complexes
$$
0 \to \Tot\, X_{2}^{+}(R)_V  \to \Tot\,X^+(R)_V
\xrightarrow{S_V} \Tot\,X^+(R)_V [2] \to 0\ ,
$$
induces the long exact sequence
\begin{equation}
\la{chom}
\ldots \to \bar{\H}_n(A,V) \oplus \H_{n-1}(\Omega^1 A, V) \xrightarrow{I_V} \tH_n(A,V) \xrightarrow{S_V} \tH_{n-2}(A,V)
\xrightarrow{B_V} \bar{\H}_{n-1}(A,V) \oplus \H_{n-2}(\Omega^1 A, V) \to \ldots
\end{equation}
which is the analogue of the Connes periodicity sequence \eqref{ISB}. The exact sequence 
\eqref{chom} justifies our terminology and notation for $ \tH_{\bullet}(A,V) $.

The periodicity map $ S_V $ in \eqref{chom} can be described as follows. Let
$$
\bar{\alpha} = (r_n,\,\bar{\omega}_{n-1},\,\bar{\omega}_{n-3},\,\ldots\,)
\,\in\, \bar{R}_V \,\oplus\, [\Omega_{\rm com}^1(R_V)/\partial R_V]_{n-1}
\,\oplus\, [\Omega_{\rm com}^1(R_V)/\partial R_V]_{n-3} \,\oplus\, \ldots
$$
be an $n$-cycle representing a homology class in $ \tH_n(A,V) $. Then there is
a cycle $\, \alpha = (r_n, \,\omega_{n-1}, \,r_{n-2}, \,\omega_{n-3},\, \ldots) $ in
$ (\Tot\,X^+(R)_V)_n $ that maps onto $ \bar{\alpha} $ under \eqref{Totv1}, and
at the level of total complexes, the map $ S_V $ is  defined by
$$
S_V(r_n, \,\omega_{n-1}, \,r_{n-2}, \,\omega_{n-3},\, \ldots) = (r_{n-2}, \,\omega_{n-3},\, \ldots)\ .
$$
Now, to construct $ \alpha $ from $ \bar{\alpha} $ we choose representatives
$\, \omega_{n-1}, \,\omega_{n-3},\, \ldots \in \Omega_{\rm com}^1(R_V) $ for the corresponding
components of $ \bar{\alpha} $ and define $\, r_{n-2},\, r_{n-4},\,\ldots \in \bar{R}_V $
by solving the equations
\begin{equation}
\la{cyeq3}
\bpar_V(r_{n-2}) = d \omega_{n-1}\ ,\quad \bpar_V(r_{n-4}) = d \omega_{n-3}\ ,\quad \ldots
\end{equation}
Note that the above equations are solvable since $\, d \bar{\omega}_{n-2k-1} = 0 \,$ in
$\,\Omega_{\rm com}^1(R_V)/\partial R_V \,$ for all $\, k =1,\,2,\,\ldots $ Also note that
the elements $\, r_{n-2},\, r_{n-4},\, \ldots $ satisfying \eqref{cyeq3}
are necessarily cycles in $ \bar{R}_V $ since $ \Ker(\bpar_V) = 0 $.
The periodicity map $ S_V $ is thus given by
\begin{equation}
\la{cyeq4}
S_V :\, \tH_n(A,V) \to \tH_{n-2}(A,V) \ ,\quad
(r_n,\,\bar{\omega}_{n-1},\,\bar{\omega}_{n-3},\,\ldots\,) \mapsto (r_{n-2},\,\bar{\omega}_{n-3},\,\ldots\,)\ .
\end{equation}
Finally, observe that the map $\, I_V:\, \bar{\H}_n(A,V) \oplus \H_{n-1}(\Omega^1 A, V) \to \tH_n(A,V) \,$
takes the class of $ (r_n,\,\omega_{n-1}) \in
\bar{R}_V \oplus \Omega_{\rm com}^1(R_V) \,$ to the class of $\, (r_n, \, \bar{\omega}_{n-1}) $,
in particular, it induces the identity map on $ \bar{\H}_n(A,V) $. By exactness of \eqref{chom}, the map
$ B_V $ then takes its values in $ \H_\bullet(\Omega^1 A, V) $.

We summarize our calculations in the following theorem.
\bthm
\la{Tconnes}
For all $ n \ge 0 $, there are commutative diagrams
\begin{equation*}
\la{S5D3}
\begin{diagram}[small, tight]
\rHC_n(A)   &  \rTo^{S} & \rHC_{n-2}(A) \\
\dTo^{\tTr_V}&                 & \dTo_{\tTr_V} \\
\tH_n(A,\,V) &  \rTo^{S_V}     & \tH_{n-2}(A,V)
\end{diagram}
\qquad \qquad \quad
\begin{diagram}[small, tight]
\rHC_n(A)   &  \rTo^{B} & \overline{\HH}_{n+1}(A) \\
\dTo^{\Tr_V}&                 & \dTo_{\Tr_V} \\
\bar{\H}_n(A,\,V) &  \rTo^{B_V}     & \H_{n}(\Omega^1 A, V)
\end{diagram}
\end{equation*}
where the map $S_V$ is given by \eqref{cyeq4}, and $ B_V $ is induced by the derivation
$ \partial_V: R_V \to \Omega_{\rm com}^1(R_V) $.
\ethm

In the spirit of the Kontsevich-Rosenberg principle, the last
diagram in Theorem~\ref{Tconnes} can be interpreted by saying that
the NC derived cotangent complex of $ A $ is represented by
Hochschild homology, so that the Connes differential
$ B:\,\HC_\bullet(A) \to \HH_{\bullet+1}(A) $ induces
(through trace maps) the de Rham differential
$ \partial_V: R_V \to \Omega_{\rm com}^1(R_V) $ at the level of homology.

\subsection{Derived polyvector fields}
\la{S5.4}
Recall that if $\phi:A \to  B $ is a morphism in $\DGA_k$, a $\phi$-derivation from $A$ to $B$ of
degree $r$ is a map $s:A \rar B$ of degree $r$ such that
$$
\phi+\varepsilon.s:A \rar B[\varepsilon]\ ,\quad
a \mapsto \phi(a)+\varepsilon .s(a)\ ,
$$
is a morphism in $\dga_k$, where $\varepsilon$ has degree $-r$, $\,\varepsilon^2 = 0\,$ and $d\varepsilon=0$.

There is a natural generalization of this definition. Fix a $k$-tuple $\,(d_1,\ldots , d_k)\,$ of integers, and
let $ \{\varepsilon_1, \ldots, \varepsilon_k \} $ be variables such that $\, |\varepsilon_i| = - d_i \,$,
$ \varepsilon_i^2 = 0 $ and $\,d \varepsilon_i = 0 \, $ for all $ i = 1, \ldots, k $.
For any order subset $ S=\{i_1 < \ldots < i_p\}$ of $\{1, \ldots ,k\}$, put $\, \varepsilon_S :=
\varepsilon_{i_1} \ldots \varepsilon_{i_p}\,$. Then, a
{\it polyderivation of (multi)degree} $(d_1,..,d_k)$ with respect to a DG algebra map $\phi:\, A \rar B$ is
defined to be a collection of maps $\phi_S:A \rar B$, one for each $S \subset \{1,..,k\}$, such that
$\, \phi_{\varnothing} = \phi\,$ and
$$
\sum_{S \subset \{1,\ldots,k\}} \varepsilon_S \cdot \phi_S:\, A \rar B[\varepsilon_1,..,\varepsilon_k]\quad
\mbox{is a morphism in}\, \DGA_k\ .
$$

\noindent Equivalently, the polyderivations can be characterized by the following Leibniz rule
\begin{equation}
\la{leib}
\phi_k(a \cdot b)=\sum_{S \subset \{1,\ldots ,k\}} \pm \,\phi_S(a) \cdot \phi_{\{1,\ldots,k\}\setminus S}(b) \ ,
\end{equation}
where the signs are determined by the Koszul sign convention. More precisely,
for $\,S=\{i_1 < \ldots < i_p\}\,$, the sign involved in \eqref{leib} is given by
$$
(-1)^{(\sum_{j \in \{1,\ldots ,k\} \setminus S} d_j)(\sum_{j \in S} d_j+|a|)} \cdot \chi_S \ ,
$$
where $\,\chi_S\,$ is the sign of the $(p,k-p)$-shuffle mapping $\{1,\ldots ,p\}$ to $S$ and acting by block
permutations on blocks of size $\{d_1,\ldots ,d_k\}$.

\begin{lemma}
\la{pT.1.0}
A polyderivation $\{\phi_S\}$ of multidegree $(d_1,..,d_k)$ with respect
to $ \phi:A \rar B$ induces a polyderivation $\{(\phi_V)_S\}$ of
multidegree $(d_1,..,d_k)$ with respect to $\,\phi_V:A_V \to B_V\,$, such that
$$
\Tr_V \circ \phi_S=(\phi_V)_S \circ \Tr_V \quad, \quad \forall \,S \subset \{1, \ldots, k\}\ .
$$
\end{lemma}
\begin{proof}
Let $\pi_V:\,B \rar \End(V) \otimes B_V $ denote the universal representation. Consider the morphism
$$
(\pi_V \otimes \id) \circ \left(\sum_S \varepsilon_S \cdot \phi_S \right):\
A \rar \End(V) \otimes B_V[\varepsilon_1,..,\varepsilon_k]\ .
$$
By adjunction, this gives us a morphism of commutative DG algebras
$$
\sum_S \varepsilon_S\,(\phi_V)_S: A_V \rar B_V[\varepsilon_1, \ldots ,\varepsilon_k]\ .
$$
Hence, the family of operators $\{(\phi_V)_S:A_V \rar B_V\}$ is a polyderivation of multidegree $(d_1,..,d_k)$
with respect to $\phi_V$. By construction, $(\phi_S(a))_{ij}=(\phi_V)_S(a_{ij})$. Hence
\begin{equation}
\la{proppolyder1} \text{Tr}_V \circ \phi_S=(\phi_V)_S \circ \text{Tr}_V \text{,}
\end{equation}
which is the claim of the lemma.
\end{proof}
Note that a polyderivation $\{\phi_{\varnothing}, \phi_{\{1\}}\}$ of mutipdegree $d_1$ is nothing but an DG-algebra homomorphism
$\phi_{\varnothing}:\, A \rar B$ together with a derivation $\phi_{\{1\}}$ of degree $d_1$ with respect
to $\phi_{\varnothing}$. We also remark that
if $A$ and $B$ are in $\cDGA_k$, and if $\{\phi_S:A \rar B\}_{S} $ is a polyderivation of multidegree $(d_1,\ldots,d_k)$,
then for each $S \subset \{1,\ldots,k\}$, $\phi_S$ is a differential operator of order $|S|$ with
respect to $\phi_{\varnothing}$ in the sense of \cite{TT} (see {\it op. cit,}, Definition~2.2).

\begin{example}
\la{T.1.1}
Let $\mathcal R $ be a graded algebra (i.e, an object of $\DGA_k$ with zero differential). Let
$D_1$ and $D_2$ be two graded derivations on $\mathcal A $ of degrees $d_1$ and $d_2$ respectively. Then, setting
$\phi_{\varnothing}:=\id \,:\, \mathcal R \rar \mathcal R $, $\phi_{\{i\}} := D_i\,:\,\mathcal R \rar \mathcal R$ for $i=1,2$ and
$\phi_{\{1,2\}}:=D_2\circ D_1\,:\,\mathcal R \rar \mathcal R$, one obtains a polyderivation of
multidegree $(d_1,d_2)$ from $\mathcal R$ to itself.
\end{example}

The following proposition is the main result in this section. We use the same notation as in Lemma~\ref{pT.1.0}.

\begin{prop} \la{pT.1.1}
Let $ R \,\in\, \DGA_k$. The map $\,\tau_V\,:\, \DER(R) \to \DER(R_V)\,$,$\, D \mapsto D_V \,$,
is a morphism of DG Lie algebras.
\end{prop}
\begin{proof}
Applying Lemma~\ref{pT.1.0} (and Example~\ref{T.1.1}) to $\mathcal R:= R_{\#}$, the DG algebra $R$
with differential forgotten, we see that it suffices to verify that
\begin{equation}\la{eT.1.1} (D_2 \circ D_1)_V \,=\, (D_2)_V \circ (D_1)_V
\end{equation}
for any $ D_1, D_2 \,\in\, \der(\mathcal R)$.
Indeed, since $[D_1,D_2]:=D_1 \circ D_2 -{(-1)}^{|D_1||D_2|} D_2 \circ D_1$ and the bracket on $\DER(R_V)$ is similarly defined, equation~\eqref{eT.1.1} implies that $\tau_V$ is a morphism of
graded Lie algebras. That $\tau_V$ is a morphism of complexes follows from the fact that the differential on
$\,\DER(R)=[d_R,-]\,$ and the fact that $(d_R)_V=d_{R_V}$.

To verify \eqref{eT.1.1} note that the (graded) algebra homomorphism
\begin{equation}
\la{eT.1.2} f_{21}\,:\,\mathcal R \rar \mathcal R[\varepsilon_1,\varepsilon_2], \,\,\,\,\,\,
a \mapsto a +\varepsilon_1.D_1(a)+\varepsilon_2.D_2(a)+\varepsilon_1\varepsilon_2(D_2 \circ D_1)(a)
\end{equation}
coincides with the composite map
\begin{equation}\la{eT.1.3}
\mathcal R \xrightarrow{f_1} {\mathcal R}[\varepsilon_1] \xrightarrow{f_2 \otimes \id}
{\mathcal R}[\varepsilon_1,\varepsilon_2]
\,,\end{equation}
where $f_1(a):=a+\varepsilon_1D_1(a)$ and $f_2(a):=a +\varepsilon_2D_2(a)$. By construction, the coefficient of
$\varepsilon_1\varepsilon_2$ in the morphism $f_{21,V}\,:\,{\mathcal R}_V \rar {\mathcal R}_V[\varepsilon_1,\varepsilon_2]
$ obtained as in the proof of Lemma~\ref{pT.1.0} is $(D_2 \circ D_1)_V$.
Since the morphisms \eqref{eT.1.2} and \eqref{eT.1.3} coincide, the standard adjunction
$$
\Hom_{\DGA_k}(\mathcal R,\,  \End\,V \otimes {\mathcal R}_V[\varepsilon_1,\varepsilon_2])
\,\cong\, \Hom_{\cDGA_k}({\mathcal R}_V,  {\mathcal R}_V[\varepsilon_1,\varepsilon_2])
$$
implies that $\,(f_{21})_V\, = \, ((f_{2})_V \otimes \id) \circ (f_1)_V\,$.
The coefficient of $\varepsilon_1\varepsilon_2$ in the right hand side of this last equation is precisely
$\,(D_2)_V \circ (D_1)_V\,$. This proves the desired proposition.
\end{proof}

Combining Proposition~\ref{pT.1.1} and Lemma~\ref{pT.1.0}, we obtain the following
\begin{cor} \la{pT.1.2}
The following diagram commutes in $ \Com(k) $
\begin{equation*}
 \begin{diagram}
  \DER(R) \otimes \bar{R}_\n & \rTo & \bar{R}_\n\\
   \dTo^{\tau_V \otimes \text{Tr}_V(R)} & & \dTo_{\text{Tr}_V(R)}\\
      \DER(R_V) \otimes \bar{R}_V & \rTo & \bar{R}_V
 \end{diagram}
\end{equation*}
\end{cor}
Now, we apply Corollary~\ref{pT.1.2} to an almost free resolution $ R \in \DGA^+_k $ of an algebra
$ A \in \Alg_k $. We assume that $ R_\# $ is finitely generated as a graded $k$-algebra
(which implies, of course, that $A$ is finitely generated as a $k$-algebra).
First, since  $\Omega^1(R) $ is a projective $R$-bimodule,
the algebra map $ R \sonto A $ induces a quasi-isomorphism
\begin{equation*}
\DER(R) \cong \HOM_{R^e}(\Omega^1 R, R) \sonto  \HOM_{R^e}(\Omega^1 R,\,A) \cong \DER(R,\,A)\ .
\end{equation*}
By \cite{BP}, Lemma~4.2.1, this yields
\begin{equation}
\la{hh+}
\H_\bullet[\DER(R)] \cong \HH^{\bullet+1}(A)\ ,
\end{equation}
where we write $\,\HH^{\bullet+1}(A) := [\oplus_{n\ge 1} \HH^{n+1}(A,A)] \oplus \Der(A) \,$ with
$ \Der(A) $ located in homological degree $0$ and $ \HH^{n+1}(A,A) $ in homological degree
$ -n $. Note that $ \H_n[\DER(R)] $ has a natural structure of graded Lie algebra
with Lie bracket induced by the commutator in $ \DER(R) $. It is easy to check
that under \eqref{hh+} this bracket corresponds to the Gerstenhaber bracket on
$ \HH^{\bullet+1}(A) $, so that \eqref{hh+} is actually an isomorphism of graded Lie algebras.
Next, we identify ({\it cf.} \cite{vdB1}, Prop.~3.3.4)
\begin{eqnarray*}
\DER(R_V) &\cong  & \HOM_{R_V}(\Omega_{\rm com}^1 R_V, R_V)\\
          &\cong  & \HOM_{R^e}(\Omega^1 R, \End\,V \otimes  R_V)\\
          & \cong & \HOM_{R^e}(\Omega^1 R, \eR) \otimes_{\eR} (\End\,V \otimes R_V) \\
          & \cong & \DDER(R) \otimes_{\eR} (\End\,V \otimes R_V) \\
          & \cong &  \DDER(R)_V\ ,
\end{eqnarray*}
where $ \DDER(R) := \DER(R, \eR) $ denotes the complex of {\it double} derivations on $ R $.
Note that the third isomorphism above requires $ \Omega^1 R $ to be finitely generated
projective bimodule, which follows from our assumption on $R$.

Now, let $ D(A) \sonto A $ be the canonical cobar-bar resolution of $A$. By the standard
lifting property, there is a quasi-isomorphism of DG algebras $ D(A) \stackrel{\sim}{\to} R $
factoring this resolution through the given resolution $ R \sonto A $.
This induces a quasi-isomorphism of complexes of right $\eA$-modules
\begin{equation}
\la{hhh+}
\DER(R, \eA) \stackrel{\sim}{\to} \DER(D(A), \eA) \cong C^{\bullet +1}(A, \eA)\ ,
\end{equation}
where $\, C^{\bullet+1}(A, \eA) \,$ is the  Hochschild cochain complex of
$ \eA $ equipped with negative (homological) grading. On the other hand,
$ \eR \sonto \eA $ induces $\,\DDER(R) \stackrel{\sim}{\to} \DER(R, \eA) \,$, which together with
\eqref{hhh+} yields
$$
\DDER(R) \stackrel{\sim}{\to} C^{\bullet +1}(A, \eA)\ .
$$
Since $R$ is almost free,
$ \Omega^1 R $ is actually semifree left $\eR$-module, and its dual $ \DDER(R) $
is then semifree right $\eR$-module. Hence the above quasi-isomorphism is a semifree
resolution of $ C^{\bullet +1}(A, \eA) $ in $ \D(\DGBimod \,R) $. It follows that
\begin{equation}\la{hhd}
\H_\bullet[\DER(R_V)] \cong \H_\bullet[\DDER(R)_V] \cong \H_\bullet(C^{\bullet + 1}(A, \eA), V) \ ,
\end{equation}
where the last group is the representation homology of the complex $ C^{\bullet + 1}(A, \eA) $
viewed as an object in $ \D(\Bimod\,A) \,$ ({\it cf.} Theorem~\ref{T2}$(b)$). The isomorphism \eqref{hhd}
allows us to equip  $ \H_\bullet(C^{\bullet + 1}(A, \eA), V) $
with a graded Lie algebra structure induced from $ \DER(R_V) $. Since $ \bar{R}_V $ is naturally
a DG Lie module over $ \DER(R_V) $, its homology $ \bar{\H}_n(A,V) = \H_n(\bar{R}_V) $ is a graded Lie module
over $\, \H_\bullet(C^{\bullet + 1}(A, \eA), V) \,$; we write $\, \mathtt{act}_V \,$
for the corresponding action map. Corollary~\ref{pT.1.2} now implies
\begin{theorem}
\la{pT.1.3}
The following diagram commutes:
\begin{equation*}
 \begin{diagram}
   \HH^{\bullet+1}(A) \otimes \rHC_{\bullet}(A)& \rTo^{\mathtt{act}\quad} &  \rHC_{\bullet}(A)\\
      \dTo^{\tau_V \otimes \Tr_V(A)_{\bullet}} & & \dTo_{\Tr_V(A)_{\bullet}}\\
\H_\bullet(C^{\bullet + 1}(A, \eA), V) \otimes \bar{\H}_{\bullet}(A,V)&
\rTo^{\mathtt{act}_V} & \bar{\H}_{\bullet}(A,V)
 \end{diagram}
\end{equation*}
where $ \,\mathtt{act} \,$ denotes the natural action
of the Gerstenhaber Lie algebra $  \HH^{\bullet+1}(A) $
on the reduced cyclic homology $\rHC_{\bullet}(A)$ \mbox{\rm (}see, e.g., \cite{Kh} \mbox{\rm)}.
\end{theorem}
The graded Lie algebra $\, \H_\bullet(C^{\bullet + 1}(A, \eA), V) \cong \H_\bullet[\DER(R_V)] \,$
may be viewed as a Lie algebra of derived vector fields acting on
$ \H_\bullet(A,V) = \H_\bullet(R_V) $, the derived algebra of functions on $ \Rep_V(A)$.
By the Kontsevich-Rosenberg principle,
Theorem~\ref{pT.1.3} should then be interpreted by saying that the Gerstenhaber algebra
$ \HH^{\bullet +1}(A) $ is the Lie algebra of derived vector fields acting $ \rHC_\bullet(A) $,
the derived space of functions on the noncommutative `$ \Spec $' of $A$.

\section{Stabilization Theorem}
\la{S4.6}
A fundamental theorem of Procesi \cite{P} implies that the algebra map
\begin{equation}
\la{proc}
\Sym[\Tr_V(A)_0]:\ \Sym[\HC_0(A)] \to A_V^{\GL(V)}
\end{equation}
is surjective for all $ V $.  A natural question is whether this result extends to higher traces:
namely, is the {\it full} trace map
\begin{equation}
\la{E19S4}
\bL \Tr_V(A)_\bullet :\ \bL[\HC(A)] \to \H_{\bullet}(A, V)^{\GL(V)}
\end{equation}
surjective? By analogy with homology of matrix Lie algebras (see \cite{T, LQ}),
we approach this question in two steps. First, we `stabilize' the family of maps
\eqref{E19S4} passing to an infinite-dimensional limit $ V \to V_\infty $ and prove
that \eqref{E19S4} becomes an isomorphism in that limit. Then, for a finite-dimensional $V$,
we construct obstructions to $ \H_{\bullet}(A, V)^{\GL(V)} $ attaining its `stable limit'.
These obstructions arise as homology of a complex that measures the failure of \eqref{E19S4}
being surjective. Thus, we answer the above question in the negative.
This section is a preliminary report on our results in this direction. The proofs are based on
invariant theory of inductive algebraic groups acting on inverse limits of algebras and modules
(see \cite{T-TT}); they will appear in our subsequent paper \cite{BR}.

We will work with unital DG algebras $A$ which are {\it augmented} over $k$. We recall that the
category of such DG algebras is naturally equivalent to the category of non-unital DG algebras,
with $ A $ corresponding to its augmentation ideal $ \bar{A} $. We identify these two categories
and denote them by $ \DGA_{k/k} $. Further, to simplify the notation we take $ V = k^d $ and
identify $\, \End\,V = M_d(k) \,$, $\, \GL(V) = \GL_k(d) \,$; in addition, for $ V = k^d $,
we will write  $ A_V $ as $ A_d $.
Bordering a matrix in $ M_d(k) $ by $0$'s on the right and on the bottom gives an embedding
$\,M_d(k) \into M_{d+1}(k) \,$ of non-unital algebras. As a result, for each $ B \in \cDGA_k $,
we get a map of sets
\begin{equation}
\la{isoun0}
\Hom_{\DGA_{k/k}}(\bar{A},\,M_d(B)) \to \Hom_{\DGA_{k/k}}(\bar{A},\,M_{d+1}(B))
\end{equation}
defining a natural transformation of functors from $ \cDGA_k $ to $ \Sets $.
Since $B$'s are unital and $ A $ is augmented, the restriction maps
\begin{equation}
\la{isoun1}
\Hom_{\DGA_k}(A,\,M_d(B)) \stackrel{\sim}{\to} \Hom_{\DGA_{k/k}}(\bar{A},\,M_d(B)) \ ,\quad
\varphi \mapsto \varphi|_{\bar{A}}
\end{equation}
are isomorphisms for all $ d \in \N $. Combining \eqref{isoun0} and \eqref{isoun1},
we thus have natural transformations
\begin{equation}
\la{isoun}
\Hom_{\DGA_k}(A,\,M_d(\,\mbox{--}\,)) \to \Hom_{\DGA_k}(A,\,M_{d+1}(\,\mbox{--}\,))\ .
\end{equation}
By standard adjunction (see Proposition~\ref{S2P1}), \eqref{isoun} yield an inverse system of
morphisms $\,\{ \mu_{d+1, d}: A_{d+1} \to A_d \} \,$ in $ \cDGA_k $. Taking the limit of this
system, we define
$$
A_{{\infty}} := \varprojlim_{d\,\in\,\mathbb N} A_d \ .
$$
Next, we recall that the group $ \GL(d) $ acts naturally on $ A_d $ (see Section~\ref{S2.3.4}),
and it is easy to check that $\,\mu_{d+1, d}: A_{d+1} \to A_d\,$ maps the subalgebra
$ A_{d+1}^{\GL} $ of $\GL $-invariants in $ A_{d+1} $ to the subalgebra $ A_d^{\GL}$ of
$\GL $-invariants in $A_d$. Defining $ \GL(\infty) := \varinjlim\, \GL(d) $ through
the standard inclusions $ \GL(d) \into \GL(d+1) $, we extend the actions of $ \GL(d) $
on $ A_d $ to an action of $ \GL(\infty) $ on $ A_{\infty} $ and let $ A^{\GL(\infty)}_{{\infty}} $
denote the corresponding invariant subalgebra. Then ({\it cf.} \cite{T-TT}, Theorem~3.4)
\begin{equation}
\la{isolim}
A^{\GL(\infty)}_{{\infty}} \cong \varprojlim_{d\,\in\,\mathbb N} A^{\GL(d)}_d \ .
\end{equation}
This isomorphism  allows us to equip $ A^{\GL(\infty)}_{{\infty}} $ with a natural topology:
namely, we put first the discrete topology on each $ A^{\GL(d)}_d $ and equip
$\,\prod_{d \in \N} A^{\GL(d)}_d \,$ with the product topology; then, identifying
$ A^{\GL(\infty)}_{{\infty}} $ with a subspace in
$\,\prod_{d \in \N} A^{\GL(d)}_d \,$ via \eqref{isolim}, we put on
$ A^{\GL(\infty)}_{{\infty}} $ the induced topology. The
corresponding topological DG algebra will be denoted $ A^{\GL}_{{\infty}} $.

Now, for each $ d \in \N $, we have the commutative diagram
\[
\begin{diagram}[small, tight]
  &       &     \FT(A)     &        & \\
  &\ldTo^{\Tr_{d+1}(A)_\bullet}  &           &  \rdTo^{\Tr_{d}(A)_\bullet} &  \\
A_{d+1}^{\GL} &       & \rTo^{\mu_{d+1, d}}  &        &  A_d^{\GL}
\end{diagram}
\]
where $\, \FT(A) = \FT(k \bs A) \,$ is the cyclic functor \eqref{cycd} restricted to $ \DGA_{k/k} $.
Hence, by the universal property of inverse limits, there is a morphism of complexes $\, \Tr_{\infty}(A)_\bullet :\, \FT(A) \to A^{\GL}_{{\infty}}\,$ that factors $ \Tr_{d}(A)_\bullet $ for each $ d \in \N $. We extend these trace maps to homomorphisms of commutative DG algebras:
\begin{equation}
\la{trd}
\bL \Tr_{d}(A)_\bullet :\ \bL[\FT(A)] \to A^{\GL}_{d}\ ,
\end{equation}
\begin{equation}
\la{trinf}
\bL \Tr_{\infty}(A)_\bullet :\ \bL[\FT(A)] \to A^{\GL}_{{\infty}}\ .
\end{equation}

The following theorem is one of the key technical results of \cite{BR}; Part $(a)$ can be viewed
as an extension of Procesi's Theorem \cite{P} to the realm of differential graded algebras.
\bthm
\la{dense}
$(a)$  The maps \eqref{trd} are surjective for all $ d \in \N $.

$(b)$  The map \eqref{trinf} is {\rm topologically} surjective: i.e.,
its image is dense in $ A^{\GL}_{{\infty}} $.
\ethm

\vspace{1ex}

\noindent
Letting $ A_\infty^{\Tr}  $ denote the image of \eqref{trinf}, we
define the functor
\begin{equation}
\la{trfun}
(\,\mbox{--}\,)^{\Tr}_{\infty}\,:\ \DGA_{k/k} \to \cDGA_k\ ,\quad A \mapsto A_\infty^{\Tr}\ .
\end{equation}
The algebra maps \eqref{trinf} then give a morphism of functors
\begin{equation}
\la{morfun}
\bL \Tr_{\infty}(\,\mbox{--}\,)_\bullet :\ \bL[\FT(\,\mbox{--}\,)] \to (\,\mbox{--}\,)^{\Tr}_{\infty}\ .
\end{equation}

Now, to state the main result of \cite{BR} we recall that
the category of augmented DG algebras $ \DGA_{k/k} $ has a natural model structure induced from
$ \DGA_k $ ({\it cf.} \ref{2.1.2}). We also recall the derived cyclic functor
$\, \L\FT(\,\mbox{--}\,):\, \Ho(\DGA_{k/k}) \to \Ho(\cDGA_k)\, $ provided by Theorem~\ref{FTT}.
\bthm
\la{eqfun}
$(a)$ The functor \eqref{trfun} has a total left derived functor
$\,\L(\,\mbox{--}\,)^{\Tr}_{\infty}\,:\ \Ho(\DGA_{k/k}) \to \Ho(\cDGA_k)\,$.

$(b)$ The morphism \eqref{morfun} induces an isomorphism of functors
$$
\bL \Tr_{\infty}(\,\mbox{--}\,)_\bullet :\  \bL[\L\FT(\,\mbox{--}\,)] \stackrel{\sim}{\to}
\L (\,\mbox{--}\,)^{\Tr}_{\infty}\ .
$$
\ethm
\noindent
By definition, $\, \L(\,\mbox{--}\,)^{\Tr}_{\infty} $ is given by
$\,
\L(\gamma\,A)^{\Tr}_{\infty} = (QA)^{\Tr}_{\infty}\,$,
where $ QA $ is a cofibrant resolution of $A$ in $ \DGA_{k/k}\,$. For an ordinary augmented $k$-algebra $A \in \Alg_{k/k} $, we set
$$
\DRep_\infty(A)^\Tr :=  (QA)^{\Tr}_{\infty} \ .
$$
By part $(a)$ of Theorem~\ref{eqfun}, $ \DRep_\infty(A)^\Tr $ is well defined. On the other hand, part $(b)$ implies
\begin{corollary}\la{corf1}
For any $ A \in \Alg_{k/k} $, $\, \bL \Tr_{\infty}(A)_\bullet $ induces an isomorphism of graded commutative algebras
\begin{equation}
\la{funhc}
\bL[\rHC(A)] \cong \H_\bullet[\DRep_\infty(A)^\Tr]\ ,
\end{equation}
where $ \rHC(A) $ is the reduced cyclic homology of $A$.
\end{corollary}

In fact, one can show that $ \H_\bullet[\DRep_\infty(A)^\Tr] $
has a natural structure of a graded Hopf algebra,
and the isomorphism of Corollary~\ref{corf1} is actually an
isomorphism of Hopf algebras. This isomorphism should be
compared to the Tsygan-Loday-Quillen isomorphism \eqref{liehom5}.
Heuristically, it implies that the cyclic homology of an augmented
algebra is determined by its representation homology. Note that
the augmentation requirement is essential: for example, the
first Weyl algebra $ A_1(k) $ has trivial representation homology
(see Example~\ref{S2Ex1}) but its cyclic homology is nontrivial.

Next, we fix $ d \in \N $ and look at the trace homomorphism \eqref{trd}.
We let $\, K_\bullet(A,d) := \Ker\,[\bL \Tr_d(A)_\bullet]\, $ denote
its kernel: by part $(a)$ of Theorem~\ref{dense}, we then
have the short exact sequence of complexes
\begin{equation}
\la{obs}
0 \to K_\bullet(A,\,d) \to \bL[\FT(A)] \to A^{\GL}_{d} \to 0\ ,
\end{equation}
which is functorial in $A$. Replacing $A$ in \eqref{obs}
by its cofibrant resolution $QA$ and taking homology yields the long exact sequence
\begin{equation}
\la{obs1}
\ldots \to \HK_n(A,\,d) \to \bL[\rHC(A)]_n \to \H_n(A,\,d)^{\GL(d)}
\xrightarrow{\delta_n} \HK_{n-1}(A,\,d) \to \ldots
\end{equation}
where $\, \HK_n(A,\,d) := \H_n[K_\bullet(QA,\,d)]\,$. From \eqref{obs1}, we see that the obstructions to surjectivity
of $ \bL \Tr_d(A)_\bullet $ in degree $\, n\,$ `live' in the homology group $ \HK_{n-1}(A,\,d)$: these are the classes in the image of
the connecting homomorphisms $\delta_n $. In all negative degrees, the complex $\, K_{\bullet}(A,\,d) \,$ is acyclic,
which means, in particular, that the map $ \bL[\rHC(A)]_0 \to \H_0(A,\,d)^{\GL} $ is surjective. Since $ \bL[\rHC(A)]_0 =
\Sym[\rHC_0(A)] $ and $ \H_0(A,\,d)^{\GL} = A_d^\GL $, we recover the classical Procesi theorem, {\it cf.} \eqref{proc}.
On the other hand, in positive degrees ($ n > 0 $), the groups $ \HK_n(A,\,d) $ are nonzero, and there are examples
showing that $ \delta_n \not= 0 $. For further details we refer the reader to \cite{BR}.

\appendix
\section{Model categories}
\la{2}
A model category is a category enriched with a certain structure that allows one to do `homotopical algebra'
(see  \cite{Q1,Q2}). The prototypical example is the category of topological spaces; however, the theory
also applies to algebraic categories, including chain complexes, differential graded algebras and
differential graded modules. In this Appendix, we recall the definition of model categories and the basic
facts we need in the body of the paper.
These facts are mostly well known; apart from the original works of Quillen, good references are \cite{DS}, \cite{Hir}
and \cite{Ho}.

\subsection{}\la{2.1}\textbf{Definition.}
A {\it (closed) model category} is a category
$\C$ equipped with three distinguished classes of morphisms: {\it weak equivalences}
$ (\,\stackrel{\sim}{\to}\,) $, {\it fibrations} $ (\onto) $ and
{\it cofibrations} $ (\into) $. Each of these classes is closed under
composition and contains all identity maps. Morphisms that are both fibrations
and weak equivalences are called {\it acyclic fibrations} and denoted $\,\stackrel{\sim}{\onto}\,$.
Morphisms that are both cofibrations and weak equivalences are called {\it acyclic cofibrations}
and denoted $\,\stackrel{\sim}{\into}\,$.  The following five axioms are required.

\vspace{1ex}

\begin{lyxlist}{00.00.0000}
\item [{$\mathrm{MC}1$}]
$\C$ has all finite limits and colimits. In particular, $ \C $ has initial and terminal objects, which we denote
`$e$' and `$ * $', respectively. \smallskip{}

\item [{$\mathrm{MC}2$}] {\it Two-out-of-three axiom}:
If $f: X \to Y $ and $ g: Y \to Z $ are maps in $ \C $ and any
two of the three maps $f,$ $g,$ and $gf$ are weak equivalences, then so is the third. \smallskip{}

\item [{$\mathrm{MC}3$}]{\it Retract axiom}:
Each of the three distinguished classes of maps
is closed under taking retracts; by definition, $ f $ is a {\it retract} of $g$
if there is a commutative diagram
\begin{equation*}
\begin{diagram}[small, tight]
X               & \rTo  &  X'      &  \rTo  & X\\
\dTo^{f}        &       & \dTo^{g} &        & \dTo^{f} \\
Y               & \rTo  &  Y'      &  \rTo  & Y\\
\end{diagram}
\end{equation*}
such that the composition of the top and bottom rows is the identity.

\item [{$\mathrm{MC}4$}] {\it Lifting axiom}:
Suppose that
\begin{equation*}
\begin{diagram}[small, tight]
                   A &  \rTo^{} & X\\
\dInto^{}        &  \ruDotsto                      & \dOnto_{} \\
B                &  \rTo^{}                   & Y
\end{diagram}
\end{equation*}
is a square in which $ A \to B $ is a cofibration and $ X \to Y $ is a fibration.
Then, if either of the two vertical maps is a weak equivalence, there is a lifting $\,B \to X \,$ making the
diagram commute. We say that $ A \to B $ has the {\it left-lifting property} with respect to
$ X \to Y $, and $ X \to Y $ has a {\it right-lifting property} with respect to $ A \to B $.
\smallskip{}

\item [{$\mathrm{MC}5$}] {\it Factorization axiom}:
Any map $ A \to X $ in $\C$ may be factored in two ways:
$$
\mbox{\rm (i)}\ A \stackrel{\sim}{\into} B \onto X \ ,\qquad
\mbox{\rm (ii)}\ A \into Y \stackrel{\sim}{\onto} X \ .
$$
\end{lyxlist}

An object $ A \in\mathrm{Ob}(\C)$ is called {\it fibrant} if the unique morphism $\,A \to * \,$
is a fibration in $ \C$. Similarly, $\, A \in \mathrm{Ob}(\C)\,$ is {\it cofibrant} if the unique morphism $ e \to A $
is a cofibration in $ \C $. A model category $ \C $ is called {\it fibrant} (resp., {\it cofibrant})
if all objects of $ \C $ are fibrant (resp., {\it cofibrant}).

\subsection{} \la{2.2ex}
\textbf{Example.}
\la{Example1}
Let $A$ be an algebra, and let $ \Com^+(A) $ denote the category of complexes of $A$-modules that have zero terms in negative degrees\footnote{We will call such complexes non-negatively graded.}. This category has a standard model structure, where the weak equivalences are the quasi-isomorphisms, the fibrations are the maps that are surjective in all positive degrees and the
cofibrations are the monomorphisms whose cokernels are complexes with projective components (see \cite{DS},
Sect.~7).

The category $ \Com(A) $ of all (unbounded) complexes of $A$-modules also has a projective model structure
with quasi-isomorphisms being the weak equivalences and the epimorphisms being the fibrations. The cofibrations
in $ \Com(A) $ are monomorphisms with degreewise projective cokernels, however not all such
monomorphisms are cofibrations (see \cite{Ho}, Sect.~2.3).
In  $ \Com^+(A) $ and $ \Com(A) $ the initial and the terminal objects are the same, namely the zero complex. Thus
$ \Com^+(A) $ and $ \Com(A) $ are fibrant model categories.

\subsection{}
\la{newcon}
There are natural ways to construct a new model category from a given one:

\subsubsection{}
\la{2.1.1} The axioms of a model category are self-dual: if $ \C $ is a model category, then so is its opposite
$ \C^{\rm opp}$. The classes of weak equivalences in $ \C $ and $ \C^{\rm opp}$ are the same, while
the classes of fibrations and cofibrations are interchanged (see \cite{DS}, Sect.~3.9).

\subsubsection{}
\la{2.1.2}
If $ S \in \mathrm{Ob}(\C) $ is a fixed object in a model category $ \C $, then the category $ \C_S $ of arrows
$ \{S \to A \} $ starting at $S$ has a natural model structure, with a morphism $ f: A \rightarrow B $ being in a
distinguished class in
$ \C_S $  if and only if $f$ is in the corresponding class in $ \C $ (see \cite{DS}, Sect.~3.10). Dually, there
is a similar model structure on the category of arrows $ \{A \to S\} $ with target at $S$.

\subsubsection{}
\la{2.1.2m}
The category $ \Mor(\C) $ of all morphisms in a model category $ \C $ has a natural model structure, in which a morphism
$\,(\alpha, \beta):\, f \to f' \,$ given by the commutative diagram
\begin{equation*}
\begin{diagram}[small, tight]
A                 & \rTo^{\alpha}    &  A'         \\
\dTo^{f}          &                  & \dTo^{f'}   \\
B                 & \rTo^{\beta}     &  B'
\end{diagram}
\end{equation*}
is a weak equivalence (resp., a fibration) iff $ \alpha $ and $ \beta $ are weak equivalences (resp., fibrations)
in $ \C $. The morphism $ (\alpha, \beta) $ is a cofibration in $ \Mor(\C) $ iff $\, \alpha \,$ is a cofibration
and also the induced morphism $\,B \coprod_A A' \to B' \,$ is cofibration in $ \C $ ({\it cf.} \cite{R}, Corollary~7.3).

\subsubsection{}
\la{2.1.3}
Let $ \D := \{a \leftarrow b \rightarrow c\} $ be the category with three objects $ \{a,b,c\} $ and the two
indicated non-identity morphisms. Given a category $ \C $, let $ \C^\D $ denotes the category of
functors $\, \D \to \C\,$. An object in $ \C^\D $ is pushout data in $ \C $:
$$
X(a) \leftarrow X(b) \rightarrow X(c)\ ,
$$
and a morphism $ \varphi: X \to Y $ is a commutative diagram
\begin{equation*}
\la{pushd}
\begin{diagram}[small, tight]
X(a)                 & \lTo    & X(b)             & \rTo             & X(c) \\
\dTo^{\varphi_a}     &         & \dTo^{\varphi_b} &                  &\dTo^{\varphi_c}  \\
Y(a)                 & \lTo    & Y(b)             & \rTo             & Y(c)
\end{diagram}
\end{equation*}
If $ \C $ is a model category, then there is a (unique) model structure on  $ \C^\D $,
where  $ \varphi $ is a weak equivalence (resp., fibration) iff
$ \varphi_a $, $\, \varphi_b $, $\, \varphi_c $ are weak equivalences (resp., fibrations)
in $ \C $. The cofibrations in $ \C^\D $ are described as the morphisms $ \varphi = (\varphi_a,\,
\varphi_b,\, \varphi_c) $, with $ \varphi_b $ being a cofibration and also the two induced maps
$\, X(a) \coprod_{X(b)} Y(b) \to Y(a) \,$, $\, X(c) \coprod_{X(b)} Y(b) \to Y(c) \,$ being cofibrations
in $ \C $.
Dually, there is a (unique) model structure on the category of pullback data
$ \C^\D $, where $ \D := \{a \rightarrow b \leftarrow c\} $ (see \cite{DS}, Sects. 10.4 and 10.8)

\subsection{Homotopy category}
\la{2.2}
In an arbitrary model category, there are two different ways of defining the homotopy equivalence
relation. If $ \C $ is a fibrant model category, we can use only one definition (namely, the
so-called `left' homotopy), which is based on cylinder objects in $\C$.

From now on, we assume that $\C$ is a fibrant model category.

\subsubsection{}
\la{2.2.1}
Let $ A \in \mathrm{Ob}(\C) $. A {\it cylinder} on $ A $ is an object $\Cyl(A) \in \Ob(\C) $
given together with a diagram
$$
A\amalg A \stackrel{i}{\into} \Cyl(A)\overset{\sim}{\onto}A \ ,
$$
factoring the natural map $\,(\mathrm{id},\, \mathrm{id}):\,A\amalg A\rightarrow A $.
By MC5(ii), such an object exists for all $A$ and comes together with two morphisms
$\, i_0:\, A \to \Cyl(A) \,$ and $\, i_1:\, A \to \Cyl(A) \,$, which are the restrictions of
$ i $ to the two canonical copies of $A$ in $ A \amalg A$. In
the category of topological spaces, for each object $ A $, there is the natural cylinder
$ \Cyl(A) = A \times [0,\,1] $, which is functorial in $A$; however, in an arbitrary model category, the cylinder objects are neither unique nor functorial in $A$.

Dually, if $X\in \mathrm{Ob}(\C)$, a {\it path object} on $X$ is an object $\Path(X)$ together
with a diagram
$$
X \stackrel{\sim}{\into} \Path(X) \stackrel{p}{\onto} X \times X\,
$$
factoring the natural map $\,(\mathrm{id},\, \mathrm{id}):\,X \rightarrow X \times X $.
\subsubsection{}
\la{2.2.2}
If $\,f,g: A \rightarrow X\,$ are two morphisms in $\C$, a {\it homotopy}
from $f$ to $g$ is a map $\,H:\,\Cyl(A) \to X $ from a cylinder object on $A$ to $X$
such that the diagram
\begin{equation*}
\begin{diagram}[small, tight]
A & \rInto^{i_0} & \Cyl(A) & \lInto^{i_1} & A \\
  & \rdTo_{f} & \dDotsto^{H} & \ldTo_{g} & \\
  &           &   X
\end{diagram}
\end{equation*}
commutes. If such a map exists, we say that $ f $ is {\it homotopic} to $g$ and write $\,f \sim g \,$.

If $A$ is cofibrant, the homotopy relation between morphisms $\,f,g: A \rightarrow X\,$ can be
described in terms of path objects: precisely, $\,f \sim g\,$ iff there exists a map
$\,H:\,A \to \Path(X)\,$ for some path object on $X$ such that
\begin{equation*}
\begin{diagram}[small, tight]
  &           &   A\\
  & \ldTo^{f} & \dDotsto^{H} & \rdTo^{g} & \\
  X & \lOnto^{p_0} & \Path(X) & \rOnto^{p_1} & X
\end{diagram}
\end{equation*}
commutes. Also, if $A$ is cofibrant and $\,f \sim g\,$, then for any path object on $X$, there is a map
$\,H:A \to \Path(X)$ such that the above diagram commutes.

\subsubsection{}
\la{2.2.3}
Applying $\mathrm{MC}5(\mathrm{ii})$ to the canonical morphism
$ e\rightarrow A,$ we obtain a cofibrant object $QA$ with
an acyclic fibration $ QA\overset{\sim}{\twoheadrightarrow}A.$ This
is called a {\it cofibrant resolution} of $A.$ As usual, a cofibrant resolution
is not unique, but it is unique up to homotopy equivalence: for any pair of cofibrant
resolutions $QA,$ $Q'A,$ there exist morphisms
\[ QA\underset{g}{\overset{f}{\rightleftarrows}}Q'A\]
 such that $fg \sim \id $ and $gf \sim \id $.

\subsubsection{}
\la{2.2.4}
By $\mathrm{MC}4$, for any morphism $f: A \rightarrow X $ and any cofibrant resolutions
$QA \sonto A$ and $ QX \sonto A $ there is a map $\tilde{f}: QA \to QX $ making the
following diagram commute:
\begin{equation*}
\begin{diagram}[small, tight]
QA &  \rDotsto^{\tilde{f}} & QX\\
\dOnto       &                         & \dOnto \\
A              &  \rTo^{f}             & X
\end{diagram}
\end{equation*}
We call this map a cofibrant lifting of $\,f\,$; it is uniquely determined by $f$ up to homotopy.

\subsubsection{}
\la{2.2.5}
When $A$ and $X$ are both cofibrant objects in $ \C$, homotopy defines
an equivalence relation on $ \Hom_{\C}(A, X)$. In this case, we write
$$
\pi(A, X) :=  \Hom_{\C}(A, X)/\sim
$$
for the set of homotopy equivalence classes of maps from $A$ to $X$.
If $A'$ is another cofibrant object in $ \C $, then
\begin{equation}\la{comhom}
f \sim h:\,A \to A'\ ,\quad g \sim k:\,A'\to X\quad
\Rightarrow\quad gf \sim hk:\,A \to X\ .
\end{equation}
This last property allows one to give the following definition.

\subsubsection{}
\la{2.2.6e}
The {\it homotopy category} of $ \C $ is a category $ \Ho(\C) $ with $ \Ob(\Ho(\C)) = \Ob(\C) $
and
$$
\Hom_{\Ho(\C)}(A,\,X) := \pi(QA,\,QX)\ ,
$$
where $QA$ and $QX$ are cofibrant resolutions of $A$ and $X$.
Note that \eqref{comhom} ensures that composition in $ \Ho(\C) $
is well defined.

\subsubsection{}
\la{2.2.6}
There is a canonical functor $\gamma:\,\C \rightarrow \Ho(\C)$ acting as
the identity on objects while sending each morphism $ f \in \C $ to the homotopy
class of its lifting $ \tilde{f} \in \Ho(\C)$ (see \ref{2.2.4}).

\begin{theorem}\la{Tloc}
Let $\C $ be a model category, and $ \D $ any category.
Given a functor $ F: \C \rightarrow \D $ sending weak equivalences to isomorphisms,
there is a unique  functor $ \bar{F}:\Ho(\C)\rightarrow \D $ such that $ \bar{F} \circ \gamma =F \,$.
\end{theorem}
Theorem~\ref{Tloc} shows that the category $ \Ho(\C) $ is the (universal) localization of $ \C $
at the class $ W $ of all weak equivalences, and thus $ \Ho(\C) $ depends only on $ \C $ and $ W $.
However, the extra structure (the choice of fibrations and cofibrations) in $\C$ is needed to
describe the morphisms in $ \Ho(\C)$.

\subsubsection{}
\la{2.4.1}
The next proposition is an abstract version of the classical Whitehead Theorem, which asserts that a weak homotopy equivalence between CW-complexes is actually a homotopy equivalence.
\begin{proposition}
\la{WhP}
Let $ f: A \to B $ be a morphism between two cofibrant objects in a (fibrant) model category $\C$.
Then $ f $ is a weak equivalence in $ \C $ if and only if $ f $ has a homotopy inverse, i.e.
there exists a morphism $ g: B \to A $ such that $ gf $ and $ fg $ are homotopic to
$ \id_A $ and $ \id_B $, respectively.
\end{proposition}
For the proof of Proposition~\ref{WhP} we refer to \cite{DS}, Sect.~4.24.
We will often use the following special case which is easy to prove directly. Suppose that
$ f: A \stackrel{\sim}{\into} B $ be an acyclic {\it cofibration} between two cofibrant objects in $\C$. Applying MC4 to the diagram
\begin{equation*}
\begin{diagram}[small, tight]
A &  \rTo^{\id}                          & A \\
\dInto^{f}       &  \ruDotsto^{g}        & \dTo \\
B                &  \rTo                   & \ast
\end{diagram}
\end{equation*}
(and using the fact that $\C$ is fibrant), we see that there exists $ g: B \rar A $ such that $\,
gf = \id_A \,$.
Since $f$ and $g$ are weak equivalences, $\gamma(f)$ and $\gamma(g)$ are mutually inverse isomorphisms in $\Ho(\C)$. Hence $fg$ is homotopic to $\id_B$.

\subsubsection{Remark}
\la{2.2.7}
A model category is called {\it pointed} if its initial and terminal objects coincide: $\, e = \ast \,$.
Despite its modest appearance this condition has a profound effect on the homotopy category: if $ \C $ is
a pointed model category, then $ \Ho(\C) $ has a suspension functor $ \Sigma $, which (in case when it
is an equivalence) makes $ \Ho(\C) $ a triangulated category. This happens for the category $ \C = \Com(A) $
of unbounded complexes of $A$-modules equipped with the projective model structure (see Example~\ref{Example1}).
In this case, the homotopy category $ \Ho(\C) $ is isomorphic (as a triangulated category) to
the derived category $ \D(A) $ of $ \Com(A) $ (see \cite{Ho}, Sect.~7.1).

\subsection{Derived functors}
\la{2.3}
Let $ F:\C \to \D$ be a functor between model categories.
A {\it (total) left derived functor} of $F$ is a
functor $\,\L F: \,\Ho(\C) \to \Ho(\D) \,$ given together with a natural
transformation
\[
\L F: \,\Ho(\C) \to \Ho(\D)\ ,\qquad t:\,\L F\circ\gamma_{\C} \to \gamma_{\D} \circ F
\]
which are universal with respect to the following property: for any pair
\[
G:\Ho(\C) \rightarrow \Ho(\D),\qquad s: \, G\circ\gamma_{\C} \rightarrow \gamma_{\D} \circ F
\]
there is a unique natural transformation $\, s':\, G \rightarrow \L F\,$ such that
\[
\begin{diagram}[small, tight]
G\circ\gamma_{\C} &                        &  \rTo^{s}      &                &  \gamma_{\D} \circ F  \\
             & \rdDotsto_{s' \gamma}  &                &  \ruTo_{t} &  \\
             &                      & \L F \circ \gamma_{\C} &               &
\end{diagram}
\]

There is a dual notion of a {\it right derived functor} $\R F$ obtained by
reversing the arrows in the above definition ({\it cf.} \ref{2.1.1}).
If they exist, the functors $ \L F $ and $ \R F $ are unique up to
canonical natural equivalence. If $ F $ sends weak equivalences to weak
equivalence, then both $ \L F $ and $ \R F $ exist and, by Theorem~\ref{Tloc},
$$
\L F = \gamma\,\bar{F} = \R F\ ,
$$
where $ \bar{F}: \Ho(\C) \to \D $ is the extension of $F$ to $ \Ho(\C)$.
In general, the functor $F$ does not extend to $ \Ho(\C) $, and
$ \L F $ and $ \R F $ should be viewed as the best possible approximations to
such an extension `from the left' and `from the right', respectively.

\subsubsection{Remark}
\la{L1S4}
Let $\, F, G: \C \to \D \,$ be two functors between model categories. Assume that the
total derived functors $\, \L F, \L G : \Ho(\C) \to \Ho(\D) \,$ exist. Then, it follows
from the universal properties of $ \L F $ and $ \L G $ that any natural transformation
$\, \varphi: F \to G\, $ induces a (unique) natural transformation $\, \varphi:
\L F \to \L G \,$.  Abusing notation, we denote these two natural transformations
by the same symbol.

\subsection{Quillen's Theorem}
\la{2.4}
One of the main results in the theory of model categories is Quillen's Adjunction Theorem.
This theorem establishes the existence of derived functors for a pair of adjoint functors satisfying
certain natural conditions. We state these conditions in the following lemma, which is a simple
consequence of the basic axioms ({\it cf.} \cite{DS}, Remark~9.8).

\blemma
\la{Qpair}
Let $ \C $ and $ \D $ be model categories. Let
\begin{equation*}
F:\, \C\rightleftarrows\D \,: G
\end{equation*}
be a pair of adjoint functors. Then the following conditions are equivalent:

$(a)$ $F$ preserves cofibrations and acyclic cofibrations,

$(b)$ $G$ preserves fibrations and acyclic fibrations,

$(c)$ $F$ preserves cofibrations and $G$ preserves fibrations.
\elemma

\noindent
A pair of functors $ (F,\,G) $ satisfying the conditions of Lemma~\ref{Qpair}
is called a {\it Quillen pair}.

\begin{theorem}[Quillen]
\la{Qthm}
Let  $\,F:\,\C\rightleftarrows\D :G\,$ be a Quillen pair.
Then the total derived functors $\L F$ and $\R G$ exist and form an adjoint pair
$$
\L F:\Ho(\C)\rightleftarrows \Ho(\D):\R G\, .
$$
The functor $\L F$ is defined by
\begin{equation}
\la{derf}
\L F(A) = \gamma\,F(QA)\ ,\quad \L F(f) = \gamma\,F(\tilde{f})\ ,
\end{equation}
where $ QA \sonto A $ is a cofibrant
resolution in $ \C $ and $ \tilde{f} $ is a lifting of $ f $ (see \ref{2.2.4}).
\end{theorem}

\noindent
For the proof of Theorem~\ref{Qthm} we refer to \cite{DS}, Sect.~9; here, we only mention one
useful result on which this proof is based ({\it cf.} \cite{DS}, Lemma~9.9).
\blemma[Brown]
\la{BrL}
If $ F:\,\C \to \D $ carries acyclic cofibrations between
cofibrant objects in $ \C $ to weak equivalences in $\D$,
then $\L F $ exists and is given by formula \eqref{derf}.
\elemma

\subsubsection*{Remark}
\la{Qrem}
In the situation of Theorem~\ref{Qthm}, if $\D$ is a fibrant category,
then $ \R G = G $. This follows from the fact that the derived functor
$ \R G $ is defined by applying $ G $ to a fibrant resolution similar
to \eqref{derf}.

\subsection{}
\la{colim}
\textbf{Example.}
\la{Example2}
Let $ \C^\D $ be the category of pushout data in a model category $ \C $
(see \ref{2.1.3}). The colimit construction gives a functor
$\,\colim:\,\C^\D \to \C \,$ which is left adjoint to the
diagonal functor
$$
\Delta:\,\C \to \C^\D \ , \quad A \mapsto \{A \xleftarrow{\id} A \xrightarrow{\id} A\}\ .
$$
Theorem~\ref{Qthm} applies in this situation giving the adjoint pair
(see \cite{DS}, Prop.~10.7)
$$
\L \colim:\Ho(\C^\D)\rightleftarrows \Ho(\C):\R \Delta\, .
$$
The functor $ \L \colim $ is called the {\it homotopy pushout functor}.

\subsection{Proper model categories}
\la{2.5} The following lemma is an easy consequence of basic axioms (see \cite{DS}, Sect.~3.14).
\blemma
\la{popb}
Let $ \C $ be a model category. Then

$(a)$ the class of cofibrations (resp., acyclic cofibrations) is stable under pushouts in $ \C $,

$(b)$ the class of fibrations (resp., acyclic fibrations) is stable under pullbacks in $ \C $.
\elemma

If we replace the cofibrations (or fibrations) in the above lemma by the weak equivalences,
neither $(a)$ nor $(b)$ will longer be true in general. This motivates the following definition.

\subsubsection{}
\la{2.5.1}
A model category $\C$ is called {\it left proper} if the pushout of any weak equivalence along a cofibration is
a weak equivalence. Dually, a model category $\C$ is called {\it right proper} if the pullback of any weak equivalence along a fibration is a weak equivalence. A model category $\C$ is called {\it proper} if it is both left proper and right proper. It is easy to see that any cofibrant model category $\C$ is left proper, and dually, any fibrant model category $\C$ is right proper.
For basic properties and examples of proper model categories we refer to \cite{Hir}, Chapter 13. Here we only prove
one technical result which we use in the present paper.
\bprop
\la{prmcat}
Let $ \C $ be a (left) proper model category. Let $ f:A \into B $ be a cofibration in $ \C $, with $ A $ being a cofibrant object. Then, for any morphism $\, g: A \rar C \,$,
the homotopy pushout of $ g $ along $f$ and the ordinary pushout of $g$ along $f$ are isomorphic
in $\Ho(\C)$.
\eprop
\begin{proof}
Choose a factorization $g=pi$ where $i:A \rar D$ is a cofibration and $p:D \rar C$ is an acyclic fibration.
Consider the following diagram where the two smaller squares are cocartesian in $ \C \,$:
$$
\begin{diagram}[small, tight]
A         & \rInto^{i} & D             & \rOnto^{p} & C \\
\dInto^{f}  &          & \dInto^{\alpha} &          & \dTo \\
B         &  \rTo    & E             &  \rTo^{h} & F
\end{diagram}
$$
Clearly, the outer square in the above diagram is then also cocartesian. Hence, $F$ is the pushout of $g=pi$ along $f$. Further, since $A$ is cofibrant and $i,f$ are cofibrations and $p$ an acyclic fibration, $E$ is the homotopy pushout of $g$ along $f$. It therefore, suffices to check that $h$ is a weak equivalence. Since $\alpha$ is the pushout of a cofibration, $\alpha$ is a cofibration. Since $h$ is the pushout of the weak equivalence $p$ along the cofibration $\alpha$, it is a weak equivalence since
$\C$ is left proper.
\end{proof}
\section{Differential graded algebras}
\la{3}
We now recall basic facts about model categories of DG algebras. While most of these are well known, there are some results for which we could not find a reference: these include part of
Proposition~\ref{lhom} relating different notions of homotopies between DG algebras and Proposition~\ref{dgproper} on properness of model categories of DG algebras. For the reader's convenience, we provide proofs.

\subsection{}\la{3.1}
Recall that by a {\it DG algebra} we mean a $\Z$-graded associative $k$-algebra equipped with a
differential of degree $-1$. We write $ \DGA_k $ for the
category of all such algebras and denote by $ \cDGA_k $ the full subcategory
of $ \DGA_k $ consisting of commutative DG algebras. On these categories,
there are standard model structures which we describe in the next theorem.

\begin{theorem}
\la{modax}
The categories $ \DGA_k $ and $ \cDGA_k$ have model structures in which

(i)\ the weak equivalences are the quasi-isomorphisms,

(ii)\ the fibrations are the maps which are surjective in all degrees,

(iii)\ the cofibrations are the morphisms having the left-lifting property
with respect to acyclic fibrations ({\it cf.} MC4).

\noindent
Both categories $ \DGA_k $ and $ \cDGA_k $ are fibrant, with the initial object $ k $ and
the terminal $0$.
\end{theorem}

\ni
Theorem~\ref{modax} is a special case of a general result of Hinich on model structure on
categories of algebras over an operad (\cite{H}, Theorem~4.1.1 and Remark 4.2). Note that the
model structure on $ \DGA_k $ is compatible with the projective model structure on the category
$ \Com_k $ of complexes (see Example~\ref{Example1}).
Since a DG algebra is just an algebra object (monoid) in $ \Com_k $, Theorem~\ref{modax}
follows also from \cite{SS} (see {\it op. cit.}, Sect.~5).

\subsection{}\la{3.2}
It is often convenient to work with non-negatively graded DG algebras.
We denote the full subcategory of such DG algebras by $ \DGA_k^+ $ and the corresponding
subcategory of commutative DG algebras by $ \cDGA_k^+ $.

We recall that a DG algebra $\, R \in \DGA_k^+ \,$ is called {\it almost free} if
its underlying graded algebra $ R_\# $ is free (i.e. $ R_\# $ is
isomorphic to the tensor algebra $ T_k V $  of a graded vector space $V$).
More generally, we say that a DG algebra map $ f: A \to B $ in $ \DGA_k^+ $
is an (almost) {\it free extension} if there is an isomorphism $\,B_\# \cong A_\# \ast_k T_k V\,$,
of underlying graded algebras such that the
composition of $ f_\# $ with this isomorphism is the canonical map $\,A_\# \into A_\# \ast_k T_k V \,$.
Here, $ A_\# \ast_k T_k V $ is the free product (coproduct) of $ A_\#  $ with a free graded algebra.

Similarly, a commutative DG algebra $ S \in \cDGA_k^+ $ is called {\it almost free}
if $\, S_\# \cong \bL_k V \,$ for some graded vector space $V$. A morphism
$ f: A \to B $ in $ \cDGA_k^+ $ is an (almost) {\it free extension} if $ f_\# $ is
isomorphic to the canonical map $\,A_\# \into A_\# \otimes \bL_k V  \,$.

\begin{theorem}
\la{modax1}
The categories $ \DGA_k^+ $ and $ \cDGA_k^+ $ have model structures in which

(i)\ the weak equivalences are the quasi-isomorphisms,

(ii)\ the fibrations are the maps which are surjective in all {\rm positive} degrees,

(iii)\ the cofibrations are the retracts of almost free algebras ({\it cf.}\,MC3).

\noindent
Both categories $ \DGA_k^+ $ and $ \cDGA_k^+ $ are fibrant, with the initial object $ k $ and
the terminal $0$.
\end{theorem}

The model structure on $ \cDGA_k^+ $ described by Theorem~\ref{modax1} is a `chain' version
of a well-known model structure on the category of commutative {\it cochain} DG algebras. This
last structure plays a prominent role in rational homotopy theory and the verification of
axioms for $ \cDGA_k^+ $ can be found in many places (see, e.g., \cite{BG} or \cite{GM}, Chap.~V).
The model structure on $ \DGA_k^+ $ is also well known: a detailed proof of Theorem~\ref{modax1} for
$ \DGA_k^+ $ can be found in \cite{M}. The assumption that $ k $ has characteristic $0$ is essential
for the existence of the above model structures on $ \cDGA_k $ and $ \cDGA_k^+ $; on the other hand,
the model structures on $\DGA_k $ and $ \DGA_k^+ $ exist for an arbitrary field $k$.

\subsection{}\la{3.3}
Let $ S $ be a fixed DG algebra in $ \DGA_k $. Then, by \ref{2.1.2}, the category $ \DGA_S $ of DG
algebras over $S$ has a natural model structure induced from $ \DGA_k $. The initial object in $ \DGA_S $
is the identity map $ \id_S $. Hence, a cofibrant resolution of $ S \to A $ in $ \DGA_S $ has the form
\begin{equation}\la{ress}
S \stackrel{i}{\into} R \stackrel{p}{\onto} A\ ,
\end{equation}
where $ i $ is a cofibration and $ p $ is an acyclic fibration (i.e., a surjective quasi-isomorphism)
in $ \DGA_k $. Similarly, if $ S \in \DGA_k^+ $, the category $ \DGA^+_S $ of non-negatively graded
DG algebras over $S$ is a model category. In this category, there is a special class of cofibrant
resolutions, namely the almost free resolutions, which are the diagrams \eqref{ress},
with $ i $ being an almost free extension of $S$.

Let $ S $ be a non-negatively graded DG algebra. Then we can form
two model categories: $\DGA_S^+$  and $\DGA_S $. The relation between these model
categories is clarified by the following
\bprop
\la{3.2for}
The inclusion functor $\,\iota: \DGA_S^+ \into \DGA_S \,$ preserves the cofibrations and the
weak equivalences. In particular, it descends to the corresponding homotopy categories.
\eprop
\bproof
It is clear that $ \iota $ preserves the weak equivalences. On the other hand,
it has a right adjoint functor $ \tau:\,\DGA_S \into \DGA^+_S  $, which is given by truncation
$$
\tau(\ldots \to A_1 \to A_0 \xrightarrow{d_0} A_{-1} \to \ldots) = (\ldots \to A_1 \to Z_0(A) \to 0)\ .
$$
where $\,Z_0(A) := \ker(d_0)\,$. Now, by construction, $ \tau $ preserves the acyclic fibrations. Since the
cofibrations in any model category can be characterized as the maps having the
left-lifting property with respect to acyclic fibrations (see \cite{DS}, Prop.~3.13), we conclude
(by adjunction) that $ \iota $ preserves the cofibrations. The last statement follows from
Theorem~\ref{Tloc}.
\eproof
Proposition~\ref{3.2for} allows one to use almost free resolutions of morphisms $ S \to A $
of non-negatively graded DG algebras as cofibrant resolutions both in $ \DGA^+_S $ and $ \DGA_S $.
A similar result also holds for the categories of commutative DG algebras.

\subsection{}\la{3.4} The homotopy equivalence relation in model categories of DG algebras can be described in explicit terms. Consider the (chain) complex of $k$-vector spaces
$$
\textbf{V}:=\,[\, 0 \rar k.t \xrightarrow{d} k.dt \rar 0 \,]\ ,
$$
where $ t $ has degree $0$ and $d$ maps $t$ to $dt$. Put $ \Omega := \bL_k \textbf{V} $
and $ \Omega_{\rm As} := T_k\textbf{V}$. Notice that $ \Omega $ is just the algebraic de Rham complex
of the affine line (although with differential having degree $-1$), and there is a natural DG algebra map
$ \Omega_{\rm As} \to \Omega $. For each $a \in k$, the assignment $t \mapsto a,\,\,\,\,\,dt\mapsto 0$ gives a
map of complexes $\,\textbf{V} \rar k\,$, which extends to a (unique) morphism of DG algebras
$\, \text{ev}_a:\, \Omega_{\rm As} \rar k\, $. Similarly, one obtains DG algebra maps
$ \text{ev}_a: \Omega \rar k$ in $\cDGA_k$. Clearly, for each $a\in k$, the following diagram commutes

\begin{equation}\la{evl}
\begin{diagram}[small, tight]
\Omega_{\rm As} &  \rTo & \Omega\\
       &  \rdTo_{\text{ev}_a}            & \dTo_{\text{ev}_a}\\
                 &              & k
\end{diagram}
\end{equation}

Fix a DG algebra $ S \in \DGA_k $ and let $ \DGA_S $ denote the category of DG algebras over $S$ with model
structure induced from $ \DGA_k $ ({\it cf.} \ref{2.1.2}).
\bprop
\la{lhom}
Let $ f,\, g:\, A \to B $ be two morphisms in $ \DGA_S $, where $A$ is a cofibrant object in $ \DGA_S $.
The following conditions are equivalent.\\
$(i)$ $\, f \sim g\,$ in $ \DGA_S \, $.\\
$(ii)$\ There exists a morphism $\,h:\, A \rar B*_k \Omega_{\rm As}\,$ in $ \DGA_S $
such that $\,h(0):=\ev_0 \circ h=f$ and $\, h(1):=\ev_1 \circ h=g$.\\
$(iii)$\ There exists a morphism $h:A \rar B \otimes \Omega$ in $\DGA_S $ such that $h(0)=f$ and $h(1)=g$.\\
$(iv)$\ There exists a family $\phi_t:A \rar B$ of morphisms in $\DGA_S$ and a family of $S$-linear derivations
$\,s_t:A \rar B\,$ of degree $1$ with respect to $\,\phi_t\,$, such that $\phi_t,s_t$ vary polynomially
with $t$, $\phi_0=f,\,\,\,\phi_1=g$ and $\frac{d}{dt}(\phi_t)=[d,s_t]$.

\eprop
\bproof
$(i) \,\,\Leftrightarrow \,\,(ii)$:  Note that the extension $B \rar B*_k \Omega_{\rm As} $ is an acyclic cofibration and that the composite map
$$B \rar B*_k \Omega_{\rm As} \xrightarrow{(\text{ev}_0,\,\text{ev}_1)} B \times B$$
equals the map $(\text{id},\text{id}):B \rar B \times B$. Since $(\text{ev}_0,\text{ev}_1)$ is clearly a fibration, $ B*_k \Omega_{\rm As} $ is a path object on $B$. By \ref{2.2.2}, $(i)$ is equivalent to $(ii)$.

$(ii)\,\,\Leftrightarrow\,\,(iii)$: Given $h:A \rar B*_k \Omega_{\rm As} $ such that $h(0)=f,\,\,\,h(1)=g$, we simply compose it with the natural surjection $B*_k \Omega_{\rm As} \rar B \otimes \Omega$ and use the diagram~\eqref{evl} to see that
$(ii)$ implies $(iii)$. For the converse, note that the natural surjection $p: B*_k \Omega_{\rm As} \rar B \otimes \Omega $ is an acyclic fibration in $ \DGA_S $. Given $h:A \rar B \otimes \Omega$, consider the commutative diagram
\begin{equation*}
\begin{diagram}[small, tight]
 S &  \rTo & B*_k\Omega_{\rm As} \\
    \dTo   &             & \dOnto_p\\
      A           &   \rTo^{h}        & B\otimes \Omega
\end{diagram}
\end{equation*}
Since $A$ is cofibrant and $p$ is an acyclic fibration, there exists $\tilde{h}:A \rar B*_k\Omega_{\rm As} $
such that $p\tilde{h}=h$. That $\tilde{h}(0)=f$ and $\tilde{h}(1)=g$ follows from the diagram~\eqref{evl}.

$(iii)\,\,\Leftrightarrow\,\,(iv)$: $h$ and $\,(\phi_t,s_t)\,$ are related by the formula
$\, h(a)=\phi_t(a) + dt.s_t(a)\,$, where $ a \in A $.
\eproof
\subsubsection{Remark}\la{3.4.1} Proposition~\ref{lhom} holds word for word for morphisms  $f,g:A \rar B$ in $\DGA_S^+$
(with $S$ in $\DGA_k^+$).

\subsubsection{Remark}\la{3.4.2} An analogous (in fact, simpler) result holds for the categories $\cDGA_S$ and $\cDGA_S^+$
since $ B \otimes \Omega$ together with the maps $\text{ev}_0,\,\,\text{ev}_1$ gives a path object for $B$. For $f,g:A \rar B$ in $\cDGA_S $ with $A$ cofibrant, $f \sim g$ iff
 there exists $h:A \rar B \otimes \Omega$ in $\cDGA_S $ with $h(0)=f,\,\,h(1)=g$ iff there exists a family $\phi_t:A \rar B$ of morphisms in $\cDGA_S $ and a family $s_t:A \rar B$ of degree $1$ derivations with respect to $\phi_t$ such that $\phi_t,s_t$ vary polynomially with $t$, $\phi_0=f,\,\,\,\phi_1=g$ and $\frac{d}{dt}(\phi_t)=[d,s_t]$. An analogous statement with the changes outlined in Remark~\ref{3.4.1} holds for $\cDGA_S^+$.

\subsubsection{Remark} \la{3.4.3} Proposition~\ref{lhom} is a stronger version of \cite{FHT}, Proposition~3.5,
and \cite{CK}, Proposition~3.6.4. Following \cite{CK},
we will refer to the data $ (\phi_t,\,s_t) $ as a {\it polynomial M-homotopy} from $f$ to $g$.
For a good overview of the homotopy theory of DG algebras we refer to \cite{FHT}, Sect.~3. In a more general form
(namely, for algebras over an operad), Proposition~\ref{lhom} appears as a remark (without proof)
in Hinich's paper~\cite{H}, see {\it op. cit.}, Remark~4.8.10. The proof we provide works also
for a number of operads, including the operad $\mathfrak{Poiss}$ governing the Poisson algebras.

\subsubsection{Remark} \la{3.4.4} We remark that if $h:A \rar B \otimes \Omega$ is a morphism in $\DGA_S$ or $\cDGA_S $, then $h(0)$ and $h(1)$ induce the same map at the level of homology. This follows (since we are working over a field of characteristic $0$) from the fact that for any $a \in k$, $\text{ev}_a:\Omega \rar k$ is homotopic to $\text{ev}_0:\Omega \rar k$ as maps of complexes
(the corresponding homotopy $\,s_a:\,k[t]dt \rar k\,$ is given by integration $\,g(t)dt \mapsto \int_0^a g(t)dt\,$).

\subsection{Properness of $\DGA_k$ and $\cDGA_k$} The following result should be known to the experts; however; we could not
find it in the literature.
\la{3.5}
\bprop \la{dgproper}
The model categories $\DGA_k$ and $\cDGA_k$ are proper.
\eprop
\bproof
We first prove the proposition for $\DGA_k$. Since cofibrations are retracts of free extensions ({\it cf.}
\cite{H}, Remark~2.2.5), it suffices to show that the pushout of a weak equivalence by a free extension is a weak equivalence. Further, since any free extension is a direct limit of finitely generated free extensions and since homology commutes with direct limits, it suffices to verify that the pushout of a weak equivalence by a finitely generated free extension is a weak equivalence. By induction on the number of generators, one reduces the latter verification to the case of the pushout of a weak equivalence by a free extension with one generator. Let $f:A \rar B$ be a weak equivalence in $\DGA_k$. Suppose that $A \rar A\langle x \rangle$ is a free extension. Suppose that $x$ has degree $i$. Then $ A\langle x \rangle \cong
\mbox{\rm Tot}^{\oplus}(\text{C}(A,x)) $, where $\text{C}(A,x)$ is the double complex concentrated in the right half-plane
whose column in homological degree $n$ is the complex
$$
(A\,x\,A\,x\,\ldots\,x\,A)[-n]
$$
with $n-1$ copies of $x$ and whose vertical differential is induced by that of $A$ (keeping in mind the Koszul sign rule and the fact that $x$ is of degree $i$) and whose horizontal differential is induced by differentiating with respect to $x$. Note that the column in homological degree $n$ is isomorphic to $A[i]^{\otimes n-1} \otimes A[-n]$. The pushout of $f$ is then seen to be the map $\text{Tot}(\text{C}(A,x)) \to \text{Tot}(\text{C}(B,x))$ induced by the morphism
$\tilde{f}$ of double complexes coinciding with $f[i]^{\otimes n-1} \otimes f[-n]$ on the column of homological degree $n$ for each $n$. Clearly, $\tilde{f}$ induces a quasi-isomorphism on columns. It therefore, induces a quasi-isomorphism of total complexes. This proves that $\DGA_k$ is left proper. Since $\DGA_k$ is fibrant, it is right proper (see section~\ref{2.5.1}) and therefore, proper.

To show that $\cDGA_k$ is left proper, we reduce to the case of a free extension by a single generator as we did for $\DGA_k$. Let $f:A \rar B$ is a weak equivalence in $\cDGA_k$ and $A \rar A[x]$ be a free extension by a single generator. If $x$ has degree $i$, $A[x]$ is the total complex of a double complex $\text{C}(A,x)$ concentrated in the right half-plane, whose column in homological degree $n$ is $A.x^{n-1}[-n]$ (with vertical differential induced by that of $A$ and horizontal differential being induced by differentiating with respect to $x$). The remainder of the argument showing that $\cDGA_k$ is proper is a trivial modification of the argument showing that $\DGA_k$ is proper.
\eproof

\subsubsection{Remark} \la{3.5.1}\la{dgaSpr} Proposition~\ref{dgproper} implies that the category $\DGA_S$ is proper for any $S \in \DGA_k$. This is because the category $\DGA_S$ is fibrant (and hence, right proper) and because the pushout of a weak equivalence along a cofibration in $\DGA_S$ can be viewed as the pushout of the same weak equivalence by the same cofibration in $\DGA_k$ (which means that $\DGA_S$ is left proper by Proposition~\ref{dgproper}).

\subsubsection{Remark} \la{3.5.2} The proof of Proposition~\ref{dgproper} works for non-negatively graded DG algebras.
Thus $ \DGA_k^+ $ and $ \cDGA_k^+ $ are also proper model categories.

\bibliographystyle{amsalpha}

\end{document}